%% file: main_final.tex
\documentclass[11pt]{amsart}
\usepackage{bbm, amsmath, amsfonts, amsthm, amssymb,hyperref,mathrsfs,comment}
\usepackage{abstract}
\usepackage{imakeidx}
\usepackage[all]{xy}
\usepackage{mathtools}
\usepackage{caption}
\usepackage[dvipsnames]{xcolor}

\makeatletter
\newif\if@check@engine  \@check@enginetrue 
\makeatother
\usepackage[usenames,dvipsnames]{pstricks}

\newtheorem{theor}{\hspace{1cm}{\sc Theorem}}[section]
\newtheorem{utver}[theor]{\hspace{1cm}{\sc Proposition}}
\newtheorem{sledst}[theor]{\hspace{1cm}{\sc Corollary}}
\newtheorem{lemma}[theor]{\hspace{1cm}{\sc Lemma}}

\newtheorem*{utver*}{\hspace{1cm}{\sc Proposition}}
\theoremstyle{definition}

\newtheorem{assumption}[theor]{\hspace{1cm}{\sc Assumption}}

\newtheorem{rem}[theor]{\hspace{1cm}{\sc Remark}}

\newcommand{\Vol}{\mathop{\rm Vol}\nolimits}
\newcommand{\conv}{\mathop{\rm conv}\nolimits}
\newcommand{\id}{\mathop{\rm id}\nolimits}
\newcommand{\im}{\mathop{\rm im}\nolimits}
\newcommand{\ind}{\mathop{\rm ind}\nolimits}
\newcommand{\MV}{\mathop{\rm MV}\nolimits}

\newcommand{\sB}{\mathscr{B}}

\newcommand{\sS}{\mathscr{S}}
\newcommand{\sU}{\mathscr{U}}

\newcommand{\wt}{\widetilde}
\newcommand{\cP}{\mathcal{P}}

\newcommand{\cI}{\mathcal{I}}

\newcommand{\cV}{\mathcal{V}}
\newcommand{\cM}{\mathcal{M}}

\newcommand{\cL}{\mathcal{L}}

\newcommand{\uA}{\underline{A}}

\newcommand{\ttor}{(\cC)^2}

\newcommand{\sN}{\scalebox{0.6}{$\cN$}}
\newcommand{\sM}{\scalebox{0.6}{$\cM$}}
\newcommand{\sL}{\scalebox{0.7}{$\cL$}}
\newcommand{\svp}{\scalebox{0.8}{$\vp$}}
\newcommand{\inds}{\underline{\ind}_A}

\newcommand{\indnp}{\underline{\ind}_{\sN}}

\newcommand{\bn}{B_{\sN}^\star}
\newcommand{\bl}{B_{\sL}^\star}
\newcommand{\bm}{B_{\scalebox{0.6}{$\cM$}}^\star}
\newcommand{\bla}{B_{\sL\scalebox{0.7}{$,2$}}^\star}
\newcommand{\bnd}{B_{\sN\scalebox{0.7}{$,d$}}^\star}
\newcommand{\bp}{B_{\svp}}
\newcommand{\mb}{\mu_{\svp}}
\newcommand{\pin}{\pi_{\sN}}
\newcommand{\pim}{\pi_{\scalebox{0.6}{$\cM$}}}
\newcommand{\symn}{\sym_{\sN}}
\newcommand{\symna}{\sym_{\sN\scalebox{0.7}{$,A$}}}
\newcommand{\symm}{\sym_{\sM}}
\newcommand{\symmv}{\sym_{\sM,V}}
\newcommand{\symq}{\sym_{\scalebox{0.6}{$\cQ$}}}
\newcommand{\symnd}{\sym_{\sN\scalebox{0.7}{$,d$}}}
\newcommand{\gp}{G_{\svp}}
\newcommand{\rtn}{R_{\sN}^\vt}
\newcommand{\indn}{\ind_{\sN}}

\newcounter{idx}

\newcommand{\rotraise}[1]{%
  \StrLen{#1}[\slen]
  \forloop[-1]{idx}{\slen}{\value{idx}>0}{%
    \StrChar{#1}{\value{idx}}[\crtLetter]%
    \IfSubStr{tlQWERTZUIOPLKJHGFDSAYXCVBNM}{\crtLetter}
      {\raisebox{\depth}{\rotatebox{180}{\crtLetter}}}
      {\raisebox{1ex}{\rotatebox{180}{\crtLetter}}}}%
}

\DeclareMathOperator{\Log}{Log}
\DeclareMathOperator{\Arg}{Arg}

\DeclareMathOperator{\Conf}{C}
\DeclareMathOperator{\UConf}{C}

\DeclareMathOperator{\Mon}{Mon}

\newcommand{\vp}{\varphi}
\newcommand{\vt}{\vartheta}
\newcommand{\vr}{\varrho}

\renewcommand{\emph}[1]{{{\color{NavyBlue} #1}}}

\renewcommand{\wr}{{\rm \, wr\,  }}

\def\sym{\mathfrak{S}}

\def\R{\mathbb R}
\def\N{\mathbb N}
\def\cN{\mathcal N}
\def\Z{\mathbb Z}

\def\C{\mathbb C}
\def\sC{\mathscr C}
\def\CC{({\mathbb C}^\star)}
\def\cC{{\mathbb C}^\star}

\def\cQ{\mathcal Q}

\def\Ax{\underline A_1}

\makeindex[columns=3, title=Index of Notations]

\begin{document}

\title[Permuting the roots of univariate polynomials]{Permuting the roots of univariate polynomials whose coefficients depend on parameters}

\author{Alexander Esterov}
\author{Lionel Lang}


\maketitle

\begin{abstract}
We address two interrelated problems concerning the permutation of roots of univariate polynomials whose coefficients depend on parameters. First, we compute the Galois group of polynomials $\vp(x)\in\C[y_1,\ldots,y_k][x]$ over $\C(y_1,\ldots,y_k)$. Provided that the corresponding multivariate polynomial $\vp(x,y_1,\ldots,y_k)$ is generic with respect to its support $A\subset \Z^{k+1}$, we determine the associated Galois group for any such $A$.
Second, we determine the Galois group of systems of polynomial equations of the form $p(x,y)=q(y)=0$ where $p$ and $q$ have fixed supports $A_1\subset \Z^2$ and $A_2\subset \{0\}\times \Z$, respectively. For each problem, we determine the image of an appropriate braid monodromy map in order to compute the sought Galois group. Among the applications, we determine the Galois group of any rational function generic with respect to its support. We also provide general obstructions to the Galois group of enumerative problems over algebraic groups.
\end{abstract}

\tableofcontents

\section{Introduction}\label{sec:intro}
\subsection{Problem 1}\label{sec:resultspermuting}
Consider the univariate polynomial 
\[\emph{\vp(x)}:= c_0(y)+c_1(y)x+\cdots+c_N(y)x^N,\]
\index{phi@$\vp$} where the coefficients $c_j(y)$ are Laurent polynomials in $\emph{y}:=(y_1,\ldots,y_k)$. The projection 
\begin{equation}\label{e:projvp}
\begin{array}{rcl}
\big\lbrace (x,y)\in\CC^{k+1}:\, \vp(x)=0 \big\rbrace & \longrightarrow & \CC^k\\
(x,y)& \longmapsto& y
\end{array}
\end{equation}
is a branched covering of degree $N$\index{N@$N$}, ramified over the Zariski-closed subset
\[\emph{\sB}:=\big\lbrace y\in\CC^k:\, \#\{x\in\CC:\,\vp(x)=0\}<N \big\rbrace.\]
Beyond the degree, the most fundamental invariant of the covering \eqref{e:projvp} is its \textit{monodromy group}, which, according to \cite{Ha79}, coincides with the Galois group of $\vp(x)$ over the function field $\C(y)$. We denote this group by $\emph{\gp}$. 

Although the computation of $\gp$ is a classical problem with wide-ranging applications (see e.g. \cite{SY}), it has been carried out in only a few cases.

The most classical result, due to Galois, concerns the general polynomial of degree $N$, 
\begin{equation*}
\vp(x):=y_0+y_1x+\cdots +y_Nx^N.
\end{equation*}
It is a classical result in Galois theory, or alternatively, the consequence of a straightforward monodromy computation, that the Galois group is the full symmetric group: $\gp=\sym_N$. 

From the perspective of fewnomial theory, a natural generalisation is to consider the general polynomial supported on a finite subset $\emph{\uA} \subset \Z$,
\begin{equation}\label{e:generalpolynomialA}
\vp(x):=\sum_{j\in\uA} \ y_j\ x^j.
\end{equation} 
Let $\emph{d}$\index{d@$d$} denote the index of the affine lattice generated by $\uA$ in $\Z$. Then $\vp$ becomes a polynomial in $x^d$ after multiplication by a monomial $x^a$. Thus, the group $\emph{U_d}$\index{Ud@$U_d$} of $d^{\text{th}}$ roots of unity acts on the set of roots $\{\vp(x)=0\}$,  and the Galois group $\gp$ consists only of $U_d$-equivariant permutations. It was only recently shown that the converse also holds: the group $\gp$ is precisely the group of all $U_d$-equivariant permutations and is therefore isomorphic to the wreath product $U_d \ \wr \ \sym_{N/d}$. This was first established in \cite{DZ} using Galois theory and later reproven in \cite{EL, EL2} using topological methods. Earlier partial results are surveyed in \cite{EL2}.

To our knowledge, general computations of the group $\gp$  have thus far been limited to cases in which the coefficients $c_j(y)$ are of degree at most $1$. In this text, we compute the group $\gp$ for any polynomial $\vp$, provided it is generic with respect to its support.
Recall that the \textit{support} of $\vp$, viewed as a polynomial in $k+1$ variables, is the subset $\emph{A}\subset \Z^{n+1}$ such that 
\begin{equation*}\label{e:support}
\vp(x)= \sum_{a=(a_0,\ldots,a_k)\in A} c_a\, x^{a_0} y_1^{a_1}\cdots y_k^{a_k}
\end{equation*}
where $c_a\in\cC$ for all $a\in A$.
For convenience, we denote by $\emph{\uA}\subset\Z$\index{uA@$\uA$} the projection of $A\subset \Z^{n+1}$ on the first factor, i.e. the support of the univariate polynomial $\vp(x)$ at a generic $y\in\CC^k$. Multiplying $\vp$ by a monomial if necessary, we make the convenient assumption that $\{0,N\}\subset \uA\subset\{0,\ldots, N\}$. 

We compute $\gp$ as the monodromy group of the covering \eqref{e:projvp}. To do so, we fix a base point $\emph{y_0}\in \CC^k\setminus\sB$\index{yzero@$y_0$}, and denote its fiber by $\emph{\cN}:=\{x\in\cC: \vp(x,y_0)=0\}$\index{Nc@$\cN$}. Hence, we view $\gp$ as a subgroup of the group $\emph{\symn}$\index{symn@$\symn$} of permutations of the set $\cN$. 

Similarly as in \eqref{e:generalpolynomialA}, the group $U_d$ acts on the set $\cN$, where $\emph{d}:=\gcd(\uA)$.\index{d@$d$}
If we denote by $\emph{\symnd}\subset \symn$ the subgroup of $U_d$-equivariant permutations, then we have $\gp\subset \symnd$.
 
It turns out that the inclusion $\gp\subset \symnd$ may be strict. To see this, consider the group homomorphism \index{indnunderline@$\indnp$}
\begin{equation*}
\begin{array}{rcl}
    \emph{\indnp}\; :\; \symnd&\longrightarrow & U_d  \\
     \sigma & \longmapsto &  \prod_{x\in E}\frac{\sigma(x)}{x}
\end{array}
\end{equation*}
where $E\subset\cN$ is any subset that intersects each $U_d$-orbit in $\cN$ exactly once. Since permutations in $\symnd$ are $U_d$-equivariant, one easily verifies that the map $\indnp$ indeed takes values in $U_d$ and that its definition is independent of the choice of the set $E$.

The group $\gp$ turns out to be the preimage under $\indnp$ of a certain subgroup of $U_d$, which we now define. 
We say that $A$ is \textit{sharp} if the sets $A\cap (\{0\}\times \Z^k)$ and $A\cap (\{N\}\times \Z^k)$ each consist of a single point. In this case, the trailing and leading coefficients $c_0(y)$ and $c_N(y)$ of $\vp$ satisfy 
$$\frac{c_0(y)}{c_N(y)}=c y_1^{a_1} \cdots y_k^{a_k}$$ for some $c\in\cC$ and $(a_1,\ldots, a_k)\in \Z^k$. We define $\emph{\vt}$\index{theta@$\vt$} to be the integer $\gcd(a_1,\ldots, a_k)$, adopting the convention that $\gcd(0,\ldots,0):=0$. If $A$ is not sharp, we set $\vt:=1$. 

\begin{theor}\label{thm:maingalois}
Let $A\subset\Z^{k+1}$ be a finite set not contained in any affine line. Then there exists a Zariski-open subset $\mathscr{O}\subset \CC^A$ in the space of polynomial supported on $A$ such that, for any $\vp$ in $\mathscr{O}$, the Galois group $\gp$ is:\\[0,2cm] 
\indent
-- the full symmetric group $\symn$ if $d=1$,\\[0,2cm] 
\indent
-- the strict subgroup $\symnd\subset \symn$ if $d>1$ and $\gcd(d,\vt)=1$,\\[0,2cm] 
\indent
-- the strict subgroup $\indnp^{\;\;-1}\big(\left\langle e^{2i\pi\vt/d}\right\rangle\big)$ of $\symnd$ if $d>1$ and $\gcd(d,\vt)\neq 1$.
\end{theor}

Let us mention that the first item of the above theorem can also be obtained via the study of spatial symmetric curves (see \cite[Theorem 1.3]{EL3}).

\begin{rem}\textbf{1.}
The equivariant subgroup $\symnd\subset\symn$ is non-canonically isomorphic to the wreath product $U_d \ \wr \ \sym_{N/d}$. Via this identification, the map $\indnp$ reads as the product
\begin{equation*}
    \begin{array}{rcl}
        U_d \ \wr \ \sym_{N/d} & \longrightarrow & U_d \\
        (\xi_1,\ldots,\xi_{N/d},\sigma) & \longmapsto & \prod_j \xi_j
    \end{array}.
\end{equation*}
If we denote $\emph{g}:=d/\gcd(d,\vt)$, Theorem \ref{thm:maingalois} states that $\gp$ is isomorphic to 
\begin{equation}\label{eq:wreathgaloisgroup}
    \left\lbrace (\xi_1,\ldots,\xi_{N/d},\sigma)\in U_d\ \wr \ \sym_{N/d}: \prod (\xi_j)^g=1 \right\rbrace
\end{equation}
whenever $\vp$ belongs to $\mathscr O$.

\textbf{2.} The Galois group of any rational function $p(x)/q(x)\in \C(x)$ coincides with the group $\gp$ associated to the polynomial $\vp(x):=p(x)+y\ q(x)$. When $p(x)$ and $q(x)$ are generic with respect to their supports, Theorem \ref{thm:maingalois} determines the Galois group of $p(x)/q(x)$. In this case, the integer $d$ can be arbitrary, while $\vt$ can only be $0$ or $1$.

\textbf{3.} 
The case of rational functions discussed above illustrates the necessity of the genericity assumption  $\vp\in\mathscr{O}$ in Theorem \ref{thm:maingalois}. 
Indeed, the Galois groups of general rational functions are far more varied than those appearing in the theorem (see e.g. \cite{Kui,Neu}).  
For $k=1$, a candidate for the Zariski-open subset $\mathscr{O}$ is given in \cite[Section 2]{EL3}. The construction of $\mathscr{O}$ in loc.\ cit. admits an elementary generalisation to arbitrary $k$.
\end{rem}

In order to prove Theorem \ref{thm:maingalois}, we compute the braid monodromy group associated with the covering \eqref{e:projvp}. Braid monodromy is a classical tool for computing the fundamental groups of hypersurface complements (see \cite{EL} and the survey \cite{Lib21}).

To define braid monodromy, we consider the \textit{unordered configuration space} $\emph{\UConf_N(\cC)}$\index{CN@$\UConf_N(X)$} of subsets of $\cC$ with $N$ elements, and its fundamental group $\emph{\bn}:=\pi_1(\UConf_N(\cC), \cN)$\index{BN@$\bn$}, known as the \textit{braid group} on $N$ strands in $\cC$ (see \cite{FN, farbmarg}).
The map $\CC^k\setminus \sB \rightarrow \UConf_N(\cC)$ that assigns to each tuple $y$ the set of roots of $\vp(x,y)$ induces, at the level of fundamental groups, the \textit{braid monodromy map} 
\[\emph{\mb}: \pi_1\big(\CC^k\setminus \sB,y_0\big)\longrightarrow \bn\, .\index{muphibr@$\mb$}\]
We denote the image of $\mb$ by $\emph{\bp}$\index{Bphi@$\bp$} and refer to it as the \textit{braid monodromy group} of the covering \eqref{e:projvp}. 

Let $\emph{\pin}: \bn\rightarrow \symn$\index{piN@$\pin$} denote the natural projection sending a braid to its underlying permutation. Then the monodromy map
of the covering \eqref{e:projvp} is given by the composition $\pin\circ\mb$, and in particular
\[\pin(\bp)=\gp.\]
Our objective is therefore to compute $\bp$.

As with $\gp\subset\symn$, the inclusion $\bp\subset\bn$ may be strict. Indeed, let $\UConf_d\subset\UConf_N(\cC)$ be the subset of point configurations that are invariant under multiplication by $U_d$, and let $\emph{\bnd}:=\pi_1(\UConf_d, \cN)$\index{BNd@$\bnd$} denote the subgroup of $U_d$-invariant braids in $\bn$. In particular, we have 
\[\pin(\bnd)=\symnd.\]
Since $\{x\in\cC:\vp(x,y)=0\}\subset \UConf_d$ for any $y\in\CC^k\setminus \sB$, we have the inclusion $\bp \subset \bnd$.

The latter inclusion may also be strict. To see this, consider the multiplication map $\UConf_N(\cC)\rightarrow \cC$, $\;c\mapsto \prod_{x\in c}x$. Denote the induced map on fundamental groups by \index{indN@$\indn$}
\begin{equation}\label{eq:indN}
\emph{\indn}: \bn \rightarrow \pi_1(\cC,x_0)\simeq \Z,
\end{equation}
where $\emph{x_0}:=\prod_{x\in \cN}x$\index{xzero@$x_0$} is the image of $\cN$. Finally, define the subgroup $\emph{\rtn}:=\indn^{\;\;-1}(\vt\Z)$\index{RNtheta@$\rtn$} of $\bn$, where $\vt$ is the integer introduced before Theorem \ref{thm:maingalois}.

\begin{theor}\label{thm:main}
Let $A\subset\Z^{k+1}$ be a finite set not contained in any affine line. Then there exists a Zariski-open subset $\mathscr{O}\subset \CC^A$ in the space of polynomial supported on $A$ such that, for any $\vp$ in $\mathscr{O}$, the braid monodromy group $\bp$ is the subgroup $\bnd\cap \rtn$ of $\bn$.
\end{theor}

\begin{rem}\label{rem:line}\textbf{1.} 
If $A\subset \Z^{k+1}$ is contained in an affine line, then it is sharp and $\vp(x,y)$ is equal to $\wt\vp(x^d y_1^{a_1} \cdots y_k^{a_k})$ for some non-singular univariate polynomial $\wt\vp$, after multiplication by a monomial. The integer $\vt$ equals $N\cdot\gcd(a_1,\ldots,a_k)$ and the bifurcation set $\sB$ is the union of the coordinate hyperplanes $\{y_j=0\}$ for which $a_j\neq 0$. The image of $\mb$ is the cyclic group generated by $\tau^{\vt}$ where $\tau$ is defined in Figure \ref{fig:braids}. The corresponding Galois group is isomorphic to the cyclic group $\vt \cdot (\Z/d\Z)$.

\textbf{2.} The map $\indn$ measures the winding number of the product of the roots of $\vp$ around $0\in\C$ as $y$ traces a loop in $\CC^k\setminus\sB$. At the level of permutations, the map $\indnp$ records the reduction modulo $d$ of this winding number. This explains why the Galois group $\gp$ can be strictly smaller than $\symnd$, as described in the third item of Theorem \ref{thm:maingalois}. As we will see in Section \ref{sec:proofmaingalois}, Theorem \ref{thm:maingalois} follows from Theorem \ref{thm:main} via the commutative diagram
\begin{equation*}
    \xymatrixcolsep{3pc}
\xymatrixrowsep{2pc}
\xymatrix{
    \bnd \ar[d]_{\indn\;} \ar@{->>}[r]^{\pin} & \symnd \ar[d]^{\;\indnp} \\
	\Z \ar@{->>}[r]  & U_d	}
\end{equation*}
where the bottom arrow is the reduction modulo $p$.
\end{rem}

\subsection{Problem 2}\label{sec:resultsreducible}
Given a pair $\emph{A}:=(\emph{A_1},\emph{A_2})$\index{A@$A$} of finite subsets of $\Z^2$, we consider the general polynomial system supported on $A$, namely the system
\begin{equation}\label{e:systpb2}
 p(x,y)=q(x,y)=0,   
\end{equation}
where $(p,q)$ is a pair of Laurent polynomials in $\emph{\CC^{A}}:=\CC^{A_1}\times\CC^{A_2}$.\index{CstarA@$\CC^{A}$} The projection
\begin{equation}\label{e:projA}
\begin{array}{rcl}
\big\lbrace (x,y,p,q)\in\CC^2\times\CC^{A}:\, p(x,y)=q(x,y)=0 \big\rbrace & \longrightarrow & \CC^{A}\\[0,1cm]
(x,y,p,q)& \longmapsto& (p,q)
\end{array}
\end{equation}
is a branched covering, ramified over a Zariski-closed subset $\emph{\sB}\subset \CC^{A}$.\index{B@$\sB$} The degree of the covering \eqref{e:projA}, which we denote by $\emph{N}$\index{N@$N$}, is the mixed volume $\MV(A_1,A_2)$, according to the Bernstein-Kouchnirenko-Khovanskii Theorem.

We are interested in the monodromy group of the covering \eqref{e:projA}, which we denote by $\emph{G_A}$\index{GA@$G_A$}. Again, this group can be interpreted as a Galois group. Its computation is part of a broader research program addressed in \cite{E19, EL2}, which aims to determine the Galois group of the general system in $k$ variables supported on any tuple $(A_1,\ldots,A_k)$, where $A_j\subset \Z^k$. 
While the desired Galois group has been computed for several large classes of supports (see again \cite{E19, EL2}), the problem remains open even in the case $k=2$. 

Here, we focus on the largely unexplored case of pairs $A$ referred to as \textit{reducible} or \textit{triangular} in the literature (see \cite{EL2,BRSY}). In suitable coordinates, these are the pairs for which $A_2\subset\{0\}\times \Z$, that is, the polynomial $q$ only depends on $y$. In this text, we compute the group $G_A$ for any reducible pair $A$.

\begin{assumption}\label{assumption}
Upon a harmless affine-linear transformation on the character lattice $\Z^2$ of $\ttor$, we may assume the following.
First, we assume that the projection $\emph{\Ax}\subset \Z$\index{Ab@$\Ax$} of $A_1$ onto the first factor of $\Z^2$ is of the form $\{0,n\}\subset \Ax \subset \{0,\ldots,n\}$ for some integer $\emph{n}\geqslant 1$\index{n@$n$}. 
We also assume that $A_2$ satisfies $\{0\}\times\{0,m\}\subset A_2 \subset \{0\}\times\{0,\ldots,m\}$ for some integer $\emph{m}\geqslant 1$\index{m@$m$}. In particular, we have $N=nm$. Finally, we assume that $0\in A_1\cap A_2$. 
\end{assumption}

We compute $G_A$ as the monodromy group of the covering \eqref{e:projA}. To do so, we fix a base point $\emph{(p_0,q_0)}$\index{pzero@$(p_0,q_0)$} in $\emph{\sC_A}:=\CC^A\setminus\sB$\index{CA@$\sC_A$} and denote its fiber by $$\emph{\cN}:=\big\lbrace{(x,y)\in(\CC)^2: p_0(x,y)=q_0(y)=0\big\rbrace}.$$\index{Nc@$\cN$}
Hence, we view $G_A$ as a subgroup of $\symn$. 

The covering \eqref{e:projA} admits a geometric structure that explains why the inclusion $G_A\subset \symn$ may be strict in general. The description of this structure relies on the following ingredients. First,
let $\emph{K}:=K(A)\subset\ttor$\index{K@$K$} be the largest subgroup such that the set $\{p(x,y)=q(y)=0\}$ is invariant under multiplication by $K$, for any $(p,q)\in\CC^{A}$. 
The group $K$ is naturally isomorphic to the quotient $\Z^2/\langle A_1,A_2\rangle$ (see \cite{EL2}). Second, the covering \eqref{e:projA} maps to the covering $$\{(y,q):q(y)=0\}\mapsto q.$$ In particular, the group $G_A$ acts on the blocks of the partition of $\cN$ induced by the $m$ horizontal lines $\{(x,y):q_0(y)=0\}$. Moreover, the induced permutation on the blocks, that is, on the set of roots of $q_0$, is an element of the Galois group $\emph{G_{A_2}}$\index{GA2@$G_{A_2}$} of the general polynomial supported on $A_2$, computed in \cite{EL}.
 
Let $\emph{\symna}\subset \symn$\index{symna@$\symna$} denote the subgroup of permutations that are both $K$-equivariant and preserve the partition of $\cN$ induced by the roots of $q_0$, with the additional requirement that the induced permutation on these roots lies in $G_{A_2}$. Thus, we have $G_A\subset\symna$.

The above structure of the covering \eqref{e:projA} and its impact on the Galois group are well known (see e.g. \cite[Theorem 1.19]{E19}). In this text, we uncover an additional structure that explains the possible strictness of the inclusion $G_A\subset \symna$. To describe this structure, write 
\[p(x,y)=:\emph{c_0(y)}+\emph{c_1(y)}x+\cdots+\emph{c_n(y)}x^n\]
and define \index{v(xy)@$v$}
\[\emph{v(x,y)}:=(-1)^{1+\frac{n}{d}}c_0(y)+c_n(y)x^d,\]
where $\emph{d}:=d(A)$\index{d@$d$} denotes the integer $\gcd(\Ax)$. To any pair $(p,q)\in\sC_A$, we can now associate the auxiliary system
\begin{equation}\label{e:vietasystem2}
    v(x,y)=q(y)=0.
\end{equation}
This system relates to the original system \eqref{e:systpb2} as follows.

Since $p(x,y)=\tilde p(x^d,y)$ for some polynomial $\tilde p$, the group $U_d$ acts on $\{x:p(x,r)=0\}$ for any root $r$ of $q$. Define a $d$-slice of $p(x,r)$ to be any subset $E\subset\{x: p(x,r)=0\}$ that intersects each $U_d$-orbit exactly once, and let $\emph{x_E}:=\prod_{x\in E}x$\index{xe@$x_E$}.
By Vieta's formula, we have 
\[
(-1)^n\frac{c_0(r)}{c_n(r)}\;=\; \prod_{p(x,r)=0}x\;=\;\prod_{x\in E}\prod_{\xi^d=1}\xi x\;=\; \;\prod_{x\in E}x^d(-1)^{d+1}\;=\;(-1)^{n+\frac{n}{d}}x^d_E\ .
\]
Therefore, the point $(x_E,r)$ is a solution to \eqref{e:vietasystem2}, and conversely, every solution to \eqref{e:vietasystem2} arises in this way, for some root $r$ of $q$ and some $d$-slice $E$ of $p(x,r)$.

This relation between the two systems induces a homomorphism at the level of Galois groups. To define the Galois group of the system \eqref{e:vietasystem2}, we denote $\emph{V}:=V(A):=(\emph{V_1},A_2)$\index{V@$V$} its support, let $\emph{v_0(x,y)}$ be the polynomial associated to $p_0(x,y)$, and denote 
$$\emph{\cM}:=\{(x,y)\in\ttor:v_0(x,y)=q_0(y)=0\}.$$ 
We define the Galois group $\emph{G_V}\subset \symm$\index{GV@$G_V$} analogously to $G_A$, that is, as the monodromy group of the covering 
$\{(x,y,v,q):v=q=0\}\to (v,q)$
based at $(v_0,q_0)$. 

Given a permutation $\sigma\in\symna$, denote by $\emph{\sigma_2}$\index{sigma2@$\sigma_2$}  the permutation of the roots of $q_0$ induced by $\sigma$. Observe that for any $d$-slice $E$ of $p_0(x,r)$, the image of the subset $E\times\{r\}\subset\cN$ under $\sigma$ is of the form $E_\sigma\times\{\sigma_2(r)\}$, where $E_\sigma$ is a $d$-slice of $p_0\big(x,\sigma_2(r)\big)$. This yields a group homomorphism
\begin{equation*}
    \begin{array}{rcl}
       \emph{\inds} \; : \;  \symna & \longrightarrow & \symm  \\
       \sigma & \mapsto & \big( (x_E,y) \longmapsto (x_{E_\sigma},\sigma_2(y)) \big) 
    \end{array}
\end{equation*}
The $K$-equivariance of $\symna$ and the inclusion $U_d\times\{1\}\subset K$ ensure that $\inds$ is well-defined.

\begin{theor}\label{thm:reduciblegalois}
Let $A:=(A_1,A_2)$ be a pair such that $A_2$ lies on a line and $A_1$ does not. Then the Galois group $G_A$ equals the subgroup $\inds^{\;\;-1}(G_V)$ of $\symna$.
\end{theor}

The case where both $A_1$ and $A_2$ are contained in a line is treated in Section \ref{sec:reducibletwolines}.

\begin{rem}
The above theorem reduces the computation of $G_A$ to that of $G_V$. If $A_1$ is not sharp, we show in Theorem \ref{thm:bwgv} that $G_V$ equals the subgroup $\symmv\subset\symm$. This subgroup is defined exactly as $\symna\subset\symn$ and depends upon the integer $d(V)$ and the group $K(V)\subset\ttor$. Note that $d(V)=d$ while the inclusion $K\subset K(V)$ may be strict.

In the case where $A_1$ is sharp, the support $V$ does not satisfy the assumption of the theorem. The group $G_V$ is determined in Section \ref{sec:reducibletwolines}, as mentioned above.
\end{rem}

As in Problem 1, we deduce the Galois group $G_A$ from a certain braid monodromy group associated with the covering \eqref{e:projA}. Let $\emph{\UConf_N(\ttor)}$\index{CNCstar2@$\UConf_N(\ttor)$} denote the \textit{unordered configuration space of} $\ttor$. A natural candidate for the braid monodromy map is the homomorphism
\[ \pi_1(\sC_A)\longrightarrow \pi_1\big(\UConf_N(\ttor)\big)\]
induced by the map $(p,q)\in\sC_A \mapsto \{p=q=0\}\in \UConf_N(\ttor)$. However, the point configurations of the form $\{p=q=0\}$ in $\UConf_N(\ttor)$ have extra structure. To capture this, we consider the subset $\emph{\sU_A}:=\sU_A(K)\subset\UConf_N(\ttor)$\index{UA@$\sU_A$} of configurations that are 
\begin{equation}\label{e:defUA}
\begin{array}{l}
\text{- equidistributed among the }m\text{ connected components of }\{q(y)=0\}\subset\ttor\text{}\\
\;\text{ for some non-singular polynomial }q\in\CC^{A_2},\\[0,1cm]
\text{- invariant under multiplication by }K\subset\ttor.
\end{array}
\end{equation}
It is clear that the map $(p,q)\mapsto \{p=q=0\}$ sends $\sC_A$ into $\sU_A$. The induced map on fundamental groups is our \textit{braid monodromy map} \index{muAbr@$\mu_A$}
\begin{equation}\label{eq:mbreducible}
    \emph{\mu_A}\; : \; \pi_1\big(\sC_A,(p_0,q_0)\big)\longrightarrow \bn,
\end{equation}
where we denote $\emph{\bn}:=\pi_1(\sU_A,\cN)$\index{BN@$\bn$}. 
We refer to $\emph{B_A}:=\im(\mu_A)$\index{BA@$B_A$} as the \textit{braid monodromy group} of the covering \eqref{e:projA}. 

Let $\emph{\pin}: \bn\rightarrow \symn$\index{piN@$\pin$} denote the natural projection sending a braid to its underlying permutation. Then the monodromy map
of the covering \eqref{e:projA} is given by the composition $\pin\circ\mu_A$, and in particular
\[\pin(B_A)=G_A.\]
Our objective is therefore to compute $B_A$.

The main geometric structure of the covering \eqref{e:projA} is already accounted for in the definition of $\bn$. In particular, one can verify that $\pin(\bn)=\symna$. For the remaining structure, consider yet another auxiliary system, namely
\begin{equation}\label{e:vietasystem}
    w(x,y)=q(y)=0,
\end{equation}
where $\emph{w(x,y)}:=(-1)^{n+1}c_0(y)+c_n(y)x$\index{w@$w$}. Again, this system relates to the original one via Vieta's formula. Indeed, write every configuration in $\sU_A$ as a union $\bigcup_{q(y)=0}F_y\times\{y\}$ of $m$ subsets $F_y\subset\cC$ of size $n$, and consider the multiplication map 
\begin{equation}\label{e:prod}
    \begin{array}{rcl}
        \sU_A & \longrightarrow  &  \UConf_m(\CC^2)\\[0,2cm]
        \displaystyle \bigcup_{q(y)=0}F_y\times\{y\} & \longmapsto & \displaystyle\bigcup_{q(y)=0} \big(\prod_{x\in F_y}x,y\big)
    \end{array}.
\end{equation}
This maps sends the solution set $\{p(x,y)=q(y)=0\}$ to $\{w(x,y)=q(y)=0\}$.

This relation between the two systems induces a homomorphism at the level of braid groups. To define the braid group of \eqref{e:vietasystem}, we denote $\emph{W}:=(\emph{W_1},A_2)$\index{W@$W$} its support, let $\emph{w_0(x,y)}$ be the polynomial associated with $p_0(x,y)$, and denote $$\emph{\cL}:=\{(x,y)\in\ttor: w_0(x,y)=q_0(y)=0\}.$$\index{calL@$\cL$}
Finally, denote $\emph{\bl}:=\pi_1(\sU_W,\cL)$\index{bl@$\bl$}, where $\emph{\sU_W}:=\sU_W(K(W))$. We define the braid monodromy group $\emph{B_W}$\index{bw@$B_W$} associated with $W$ analogously to $B_A$, namely as the image of the braid monodromy map $\mu_W$.

The image of $\sU_A$ under the map \eqref{e:prod} is contained in the superset $\sU_W(\{1\})\supset\sU_W$ consisting of configurations that are equidistributed among the $m$ connected components of $\{q(y)=0\}$. Denote its fundamental group by $\emph{\bla}$\index{bla@$\bla$}, and let
\index{indA@$\ind_A$}
\begin{equation*}
    \emph{\ind_A} : \bn \longrightarrow \bla
\end{equation*}
be the map induced by \eqref{e:prod} at the level of fundamental groups. Observe that the inclusion $\sU_W\subset\sU_W(\{1\})$ induces an inclusion $\bl\hookrightarrow\bla$ at the level of fundamental groups. This allows us to consider $B_W\subset\bl$ as a subgroup of $\bla$.

\begin{theor}\label{thm:reduciblebraid}
Let $A:=(A_1,A_2)$ be a pair such that $A_2$ lies on a line and $A_1$ does not. Then the braid monodromy group $B_A$ equals the subgroup $\ind_A^{\;\;-1}(B_W)$ of $\bn$.
\end{theor}

\begin{rem}
\textbf{1.} The above theorem reduces the computation of $B_A$ to that of $B_W$. If $A_1$ is not sharp, we show in Theorem \ref{thm:bwgv} that $B_W$ equals $\bl$.
If $A_1$ is sharp, we show in the same theorem that $B_W$ isomorphic to $\pi_1(\cC) \times B_{A_2}$, where $B_{A_2}$ is the braid monodromy group of the univariate polynomial supported on $A_2$ (see \cite{EL}).

\textbf{2.} The map \eqref{e:prod} assigns to each root $r$ of $q$ the product of the roots of the polynomial $p(x,r)$ in the variable $x$. The map $\ind_A$ records the trajectory of each product as $(p,q)$ traces a loop in $\sC_A$. At the level of permutations, the map $\inds$ can be interpreted as a reduction modulo $p$ of the map $\ind_A$, reflecting the fact that the system \eqref{e:vietasystem2} is obtained from the system \eqref{e:vietasystem} via a $d$-folded covering. We will see in Section \ref{sec:pb2} that the associated braid monodromy groups $B_V$ and $B_W$ are in fact isomorphic, and that Theorem \ref{thm:reduciblegalois} follows from Theorem \ref{thm:reduciblebraid} via the commutative diagram
\begin{equation*}
    \xymatrixcolsep{3pc}
\xymatrixrowsep{2pc}
\xymatrix{
    \;B_A\; \ar[d]_{\ind_A\;} \ar@{->>}[r]^{\pin} & \;G_A\; \ar[d]^{\;\inds} \\
	\;B_W\simeq B_V\; \ar@{->>}[r]^{\pi_{\scalebox{0.6}{$\cM$}}}  & \;G_V\;	} \quad .
\end{equation*}
\end{rem}

To conclude the introduction to Problems~1 and~2, we acknowledge that this manuscript involves substantial notation. We see no convenient way to
avoid this and refer the reader to the index of notation at the end of
the paper.

\subsection{Discussions}

\subsubsection{Braid monodromy groups versus Galois groups}
There are at least three motivations for considering braid monodromy. First, the Galois groups sought in Problems~1 and~2 are obtained as natural projections of the corresponding braid monodromy groups. Second, the computation of braid monodromy groups enjoys functorial properties that the Galois groups do not possess.
 In particular, the braid monodromy groups computed here are invariant under coverings of $\ttor$. This allows us to reduce to simpler support sets $A$ in both problems (see Sections \ref{sec:red} and \ref{sec:12}). Finally, the braid monodromy group is a finer invariant of the underlying covering than its monodromy group. It gives a better approximation of the fundamental group of the base, which is a discriminant complement in both problems. This is of great interest in the context of \cite{DL} (see \cite{Lib21} for a recent account). 

\subsubsection{Towards generalisations of Zariski's Theorem}
Theorem \ref{thm:main} generalises \cite[Theorem 1]{EL}, which describes the braid monodromy of the general (univariate) polynomial supported on $\uA\subset \Z$. 
Indeed, consider a generic polynomial $\vp(x,y)$ support on the set $A:= \uA \times\{0\}\cup \uA\times\{1\}\subset \Z^2$, so that $\vp$ corresponds to a generic line in $\C^{\uA}$. By Zariski's Theorem \cite[Théorème]{Ch}, the braid monodromy group of $\uA$ is the image of $\mb$.
More generally, Zariski's Theorem relates the $j^{th}$ homotopy group of the complement of a (quasi-)projective hypersurface with the $j^{th}$ homotopy group of the complement of its pullback under a linear map. It is natural to investigate what happens when linear maps are replaced with arbitrary algebraic maps. Theorem \ref{thm:maingalois} is one step in this direction: we restrict ourselves to the $1^{st}$ homotopy group but consider polynomial maps $\C\to\C^{\uA},\, y\mapsto\vp(x,y)$, of arbitrary degree.

\subsubsection{Sectional monodromy}
Theorem \ref{thm:maingalois} indicates that the sectional monodromy group
of a smooth projective curve need not coincide with the monodromy arising from a given linear subsystem of sections.
Indeed, according to Theorem \ref{thm:maingalois} and \cite[Proposition 1.1]{Kad}, it is possible to find a planar curve $\{\vp(x,y)=0\}$ with sectional monodromy group $\symn$ for which the monodromy group of horizontal sections, namely the group $\gp$, is strictly smaller.
In other words, restricting to a non-complete linear system may
reduce the sectional monodromy group for certain planar curves.
It would be interesting to study whether this phenomenon persists for spatial curves and in higher codimension.

\subsubsection{Relation between the two problems} Both problems deal with the permutation of roots of univariate polynomials whose coefficients depend on parameters, although this is less obvious in Problem~2. There, one may restrict to the subgroup of either $B_A$ or $G_A$ consisting of elements acting trivially on the set of roots of $q_0$. One is therefore considering the problem of simultaneous permutation of the roots of the collection of polynomials $p_0(x,r)$ indexed by the roots of $q_0$. The result can be compared with Corollary~4 in \cite{EL}.

Let us also mention that we can essentially assume that $k=1$ in Problem~1, thanks to Proposition \ref{prop:reductionk=1}. Under this assumption, Problem~1 relates to Problem~2 when the polynomial $q$ has degree 1, with the major difference being that the coefficients of $p(x,y)$ are allowed to vary in Problem~2 while the coefficients of $\vp$ are fixed once and for all in Problem~1. This is reflected in the fact that the Galois group of Problem~1 is contained in that of Problem~2 and that the latter never assumes the form described in item 3 of Theorem \ref{thm:maingalois}.

\subsubsection{A word on the techniques} 
To prove Theorem \ref{thm:main}, we use considerations involving coamoebas, similar to those of \cite{EL}. In the case $k=1$ and $\vert A\vert =3$, we compute explicitly the image under $\mb$ of a set of generators of $\pi_1(\C\setminus \sB,y_0)$ in Section \ref{sec:tri}. This allows us to compute the kernel of the map and to obtain a presentation of the braid monodromy group $B_A$. In Section \ref{sec:discussion}, we discuss how methods from tropical geometry allow one to extend these results to arbitrary supports $A\subset \Z^2$. We also discuss how such results allow us to study isomonodromy loci, namely subsets of the form $\left\lbrace \vp\in \C^A:\, \im(\mb)=B\right\rbrace$ for a given subgroup $B\subset \bn$.

\subsection{Galois group of enumerative problems over algebraic groups}

\subsubsection{The Bernstein-Kouchnirenko-Khovanskii enumerative problem}\label{rem:reducibleGalois}
Theorems \ref{thm:main} and \ref{thm:reduciblebraid} illustrate that the product of the solutions to the systems of equations considered in Problems 1 and 2 play a central role in computing the corresponding Galois groups. 

This invites speculation on the Galois group of the general system supported on tuples $\emph{\wt 
A:=(\wt A_1,\ldots,\wt A_k)}$, $\wt A_i\subset\Z^k$, that are both \textit{irreducible} and 
\textit{non-reduced} (also referred to as  \textit{not triangular} and \textit{lacunary}). For such  a tuple, there is a finite covering $\emph{\psi}:\CC^k\rightarrow\CC^k$ of degree $>1$, such 
that any polynomial $\tilde f_i\in\CC^{\wt A_i}$ is of the form $f_i\circ\psi$ for some other 
polynomial $f_i$ whose support we denote by $\emph{A_i}$. In particular, the solution set of the system
$$\tilde f_1=\cdots=\tilde f_k=0, \quad \tilde f_i \in\CC^{\wt A_i}$$
is the preimage under $\psi$ of the solution set of the system
$$f_1=\cdots=f_k=0, \quad  f_i \in\CC^{A_i}$$
supported on the irreducible and reduced tuple $\emph{A}:=(A_1,\ldots,A_k)$.

According to \cite[Theorem 1.5]{E19}, the Galois group $\emph{G_{A}}$ associated to the tuple $A$ is the full symmetric group. The factorisation $\tilde f_i=f_i\circ\psi$, for $i\in\{1,\ldots,k\}$, implies that the Galois group $G_{\wt A}$  consists entirely of $K$-equivariant permutations, where $\emph{K}:=\ker(\psi)$. In particular, the group $G_{\wt A}$ is isomorphic to a subgroup of the wreath product $K\ \wr \ G_A$. 

In \cite{EL2}, we provide examples of tuples $\wt A:=(\wt A_1,\ldots,\wt A_k)$ for which $G_A$ is a strict subgroup of $K\ \wr\ G_A$. There is therefore another geometric structure, beyond the action of $K$ on the set $\{\tilde f_1=\cdots=\tilde f_k=0\}$, that affects the Galois group $G_{\wt A}$. 

As mentioned above, Theorems \ref{thm:main} and \ref{thm:reduciblebraid} suggest that special attention should be paid to the product of the solutions of the system $f_1=\ldots=f_k=0$. Indeed, as the tuple of polynomials $\emph{f}:=(f_1,\ldots,f_k)$ travels along a loop in the complement of the branching locus in $\CC^A$, the product of the solutions to $f=0$ traces a loop in $\CC^k$, and thereby defines an element of $\pi_1(\CC^k)$. Collecting all such loops yields a subgroup of $\pi_1(\CC^k)$, which a priori depends on the combinatorics of the tuple $A$. As we discuss in the next section, this  subgroup has an impact on the Galois group $G_{\wt A}$ that is not accounted for in the inclusion $G_{\wt A}\subset K\ \wr\ G_A$.

\subsubsection{The general case}
An enumerative problem $\emph{\cP}$ usually refers to the following data: 

-- a pair of smooth algebraic varieties $\emph{\mathscr{T}}$ and $\emph{\sC}$, with $\sC$ connected,

-- an algebraic variety $\emph{\sU}\subset \mathscr{T}\times \sC$ such that the projection $\mathscr{T}\times \sC\to \sC$ restrict to a finite covering $\emph{c}:\sU\to \sC$.\\
Working over $\C$, the monodromy group of the covering $c$ can be considered as a Galois group, that we denote by \emph{$G_\cP$} (see again \cite{Ha79}). This framework includes the coverings considered in Problems 1 and 2, as well as the BKK enumerative problem mentioned in the previous section.

Assume that $\mathscr{T}$ is an algebraic group. A finite, surjective homomorphism $\emph{\psi}: \wt{\mathscr{T}} \rightarrow \mathscr{T}$ between algebraic groups gives rise to a new enumerative problem $\emph{\wt \cP}$, with underlying data
$\widetilde{\mathscr{T}}$, $\wt{\sC}:=\sC$ and $\wt{\sU}:=(\psi,\id)^{-1}(\sU)\subset \wt{\mathscr{T}}\times \sU$. The group $\emph{K}:=\ker(\psi)$ acts on $\wt{\sU}$ and yields an inclusion
\begin{equation}\label{eq:inclusionwreath}
G_{\wt \cP}\hookrightarrow K \ \wr \  G_\cP,
\end{equation}
see \cite[Observation 2.5]{EL2}. The BKK enumerative problem studied in \cite{EL2} illustrates the fact that the above inclusion may be strict. 

Here, we present a geometric structure that explains the possible strictness of this inclusion in the general setting of enumerative problems over algebraic groups.
To do so, we consider the \textit{braid monodromy map} associated with $\cP$, namely the map 
$$\emph{\mu_\cP}:\pi_1(\sC)\rightarrow\pi_1(\UConf_N(\mathscr{T}))$$ 
induced by $f\in\sC\mapsto c^{-1}(f)\subset\mathscr{T}$. We denote by $\emph{B_\cP}:=\im(\mu_\cP)$ the associated \textit{braid monodromy group}. Again, there is a natural projection $\emph{\pi_{\scalebox{0.6}{$\cP$}}}: B_\cP\rightarrow G_\cP$ sending a braid to its underlying permutation.
Since $\mathscr{T}$ is an algebraic group, we have a product map $\UConf_N(\mathscr{T})\rightarrow \mathscr{T}$, $c\mapsto\prod_{z\in c}z$. We denote the induced map at the level of fundamental groups by
$$\emph{\ind_\cP}: \pi_1(\UConf_N(\mathscr{T}))\rightarrow \pi_1(\mathscr{T}).$$

We claim that the Galois group $G_{\wt\cP}$ of the enumerative problem $\wt\cP$ depends on the group $\emph{\cI}:=\ind_\cP(B_\cP)$ associated with the enumerative problem $\cP$. Indeed, on the one hand, the map $\UConf_N(T)\rightarrow \UConf_{\wt N}(\wt{\mathscr{T}})$ obtained by pulling back point configurations by $\psi$ induces an isomorphism $\psi^*: B_\cP \rightarrow B_{\wt\cP}$. On the other hand, we have a short exact sequence
\[0\longrightarrow \pi_1(\wt{\mathscr{T}}) \overset{\psi_*}{\longrightarrow} \pi_1(\mathscr{T}) \overset{m}{\longrightarrow} K\longrightarrow 0\]
where $\emph{m}$ is the monodromy map of the covering $\psi$. Since $\ker(\pi_{\wt\cP}\circ\psi^*)\subset \ker(m\circ \ind_\cP)$, there exists a map $\underline{\ind}_{\wt\cP}:G_{\wt\cP}\rightarrow K$ that fits into the following commutative diagram
\[
\xymatrixcolsep{3pc}\xymatrix{
	 \;B_{\cP}\; \ar[d]_-{\ind_\cP} \ar[r]^-{\pi_{\wt\cP}\circ\psi^*} & \;G_{\wt\cP}\; \ar[d]^-{\underline{\ind}_{\wt\cP}}\\
	\pi_1(T) \ar[r]^-{m}  & \; K \;
} \quad .
\]
The latter diagram yields the inclusion 
\begin{equation}\label{eq:inclusionind}
G_{\wt\cP}\subset \underline{\ind}_{\wt\cP}^{\;\;-1}(m(\cI)).
\end{equation}
Therefore, the following question looks natural. \begin{quote} Is the Galois group $G_{\wt\cP}$ of the enumerative problem $\wt\cP$ completely determined by the inclusions \eqref{eq:inclusionwreath} and \eqref{eq:inclusionind}? \end{quote}
This applies in particular to the enumerative problem associated to the irreducible non-reduced tuples $\wt A$ studied in \cite{EL2}. There, the obstruction described by \eqref{eq:inclusionind} appeared, in disguise, via the Poisson-type formula \cite[Theorem 1.1]{DAS}.

\section{Proofs for Problem 1}

In this section, we address Problem 1 described in Section \ref{sec:resultspermuting} and prove Theorems \ref{thm:maingalois} and \ref{thm:main}. We refer to the aforementioned section for notations.

\subsection{Obstructions}\label{sec:obs} In this section, we show that the braid monodromy group of $\vp$ is subject to the inclusion given in Theorem \ref{thm:main}.

\begin{lemma}\label{lem:constraint}
For any support $A\subset\Z^{k+1}$ not contained in any affine line, the group $B_A$ is a subgroup of $\bnd\cap\rtn$.
\end{lemma}

\begin{proof}
The inclusion $B_A\subset \bnd$ follows from the fact that $\vp(x,y)= \wt\vp(x^d,y)$ for some Laurent polynomial $\wt\vp$, whose support we denote by $\wt A\subset \Z^{k+1}$. Let us add a tilde to every piece of notation coming from $\wt\vp$. Thus, the map $(x,y)\mapsto (x^d,y)$ induces an isomorphism from $\C^{\wt A}$ to $\C^{A}$ that maps the branching locus $\wt\sB \subset \C^{\wt A}$ to  $\sB\subset \C^A$. 
In turn, we have the commutative diagram
\begin{equation}\label{eq:comdiag1}
\xymatrixcolsep{3pc}
\xymatrix{
\pi_1(\CC^k\setminus\wt\sB) \ar[r]^{\sim} \ar[d]^{\mu_{\wt\vp}} & \pi_1(\CC^k\setminus \sB) \ar[d]^{\mb} \\
\;\;B_{\scalebox{0.6}{$\wt\cN$}}^\star\;\; \ar[r]^{f_{d}}  & \;\;\bnd\;\;
}
\end{equation}
where $\emph{f_{d}}:B_{\scalebox{0.6}{$\wt\cN$}}^\star\to \bnd$\index{fd@$f_{d}$} is the isomorphism induced by the map $\UConf_{N/d}(\cC)\to \UConf_{N}(\cC)$ taking a configuration of points to its preimage under $x\mapsto x^d$. The inclusion $B_A\subset \bnd$ follows from the commutativity of the above diagram.

Since $\rtn=\bn$ when $\vt=1$, we can restrict ourselves to the case $\vt>1$. Thus, the coefficients $c_0$ and $c_{N}$ of $\vp$ are monomials and $c_0/c_{N}=cy_1^{a_1} \cdots y_k^{a_k}$ with $\vt:=\gcd(a_1,\ldots,a_k)$. For a given tuple $y:=(y_1,\ldots,y_k)$, the product of the $N$ roots of $\vp(x,y)$ is equal to $cy_1^{a_1} \cdots y_k^{a_k}$ up to sign, by Vieta's formula. Therefore, for any loop $\gamma\in \pi_1(\CC^k\setminus \sB,y_0)$, we have that $\indn(\mb(\gamma))$ is the rotational index of the composition $(y_1^{a_1} \cdots y_k^{a_k}) \circ \gamma$ around $0\in \C$. The latter index is necessarily in $\vt\Z$, which implies the inclusion  $B_A\subset \rtn$.
\end{proof}

\subsection{Reductions}\label{sec:red}
In this section, we argue that we can restrict to simpler cases while proving Theorem \ref{thm:main}. First, we show that we can restrict to a single parameter to determine $B_A$. 

\begin{utver}\label{prop:reductionk=1}
Theorem \ref{thm:main} holds if and only if it holds for $k=1$.
\end{utver}

\begin{proof}
Fix a support $A\subset \Z^{k+1}$ as in Theorem \ref{thm:main} and a generic polynomial $\vp(x,y)\in\C^A$. Assume that Theorem \ref{thm:main} holds for $k=1$. 

In order to prove the statement, it suffices to find a polynomial map $L:\cC\rightarrow\CC^k$, $s\mapsto L(s)$, such that $\vp_L(x,s):=\vp(x,L(s))$ is generic with respect to its support $A_L$, with the same invariants $N$, $d$ and $\vt$ as $\vp$. Moreover, the support $A_L$ should not be contained in a line. Indeed, observe that $\im(\mu_{\vp_L})\subset \im(\mb)$, that $\im(\mu_{\vp_L}) = \rtn \cap \bnd$ by assumption, and that $\im(\mb)\subset \rtn \cap \bnd$ by Lemma \ref{lem:constraint}. Thus, the sought equality $\im(\mb)= \rtn \cap \bnd$ follows.

We take $L$ to be a monomial map, that is, $L(s)=(s^{n_1},\ldots,s^{n_k})$. The dual map $\Z^k\rightarrow \Z$ between character lattices is a linear projection, and the map $\C^A\rightarrow \C^{A_L}$, $\vp\mapsto \vp_L$, between the corresponding spaces of polynomials is surjective. Consequently, the polynomial $\vp_L$ 
can be made generic with respect to its support $A_L$ by choosing the coefficients of $\vp$ accordingly. Moreover, the polynomials $\vp$ and $\vp_L$ share the same integers $d(A)=d(A_L)$ and $N(A)=N(A_L)$, since these integers only depend on the projection of the support $A$ (respectively $A_L$) onto the first coordinate axis of $\Z^{k+1}$ (respectively $\Z^{1+1}$). 

Let us show that we can choose the exponents $n_j$ involved in $L$ so that $\vt(A)=\vt(A_L)$. If $A$ is sharp, that is, $c_0/c_{N}=cy_1^{a_1}\cdots y_k^{a_k}$ for some $c\in\cC$ and some tuple $(a_1,\ldots,a_k)\in\Z^k$, then $(c_0/c_{N})\circ L =cy^{a_1n_1+\cdots 
+a_kn_k}$. If we take $n_1,\ldots,n_k$ such that $a_1n_1+\cdots +a_kn_k=\vt$, which is possible since $\vt=\gcd(a_1,\ldots,a_k)$, then we have $\vt(A)=\vt(A_L)$. If $A$ is not sharp, we can always take $(n_1,\ldots,n_k)$ such that the map $A\rightarrow A_L$ induced by $L$ is injective (take $(n_1,\ldots,n_k)$ not orthogonal to the difference between any pair of points in $A$). In particular, $(c_0/c_{N})\circ L$ is not a monomial and $\vt(A)=\vt(A_L)=1$. 

Finally, it remains to show that $L$ can be chosen so that $A_L$ is not contained in a line. If $A$ is not sharp, at least one of $c_0\circ L$ or $c_{N}\circ L$ has at least two monomials (recall that we took $L$ to be injective in that case). Thus, the support $A_L$ is not contained in a line. Assume that $A$ is sharp. Then the set of vectors $(n_1,\ldots,n_k)\in\Z^k$ such that $A_L$ is contained in a line is itself contained in an affine hyperplane $V\subset \R^k$ not parallel to $H:=\{(n_1,\ldots,n_k)\in\Z^k: n_1a_1+\cdots+n_ka_k=\vt\}$. This is because $A$ is not contained in a line. In particular, the set $H\setminus V$ is nonempty, and any vector $(n_1,\ldots,n_k)\in H\setminus V$ will do. The result follows.
\end{proof}

We conclude this section with another reduction.

\begin{utver}\label{prop:redd=1}
Theorem \ref{thm:main} holds if and only if it holds for supports $A$ with $d(A)=1$.
\end{utver}

\begin{proof}
Fix a support $A\subset \Z^{k+1}$ as in Theorem \ref{thm:main} and a generic polynomial $\vp(x,y)\in\C^A$. As in the proof of Lemma \ref{lem:constraint}, we can write 
\[
\vp(x,y)= \wt\vp(x^d,y)
\]
for some Laurent polynomial $\wt\vp$ supported on $\wt A\subset \Z^{k+1}$. Again, we add a tilde to every piece of notation coming from $\wt\vp$. In view of the commutative diagram \eqref{eq:comdiag1}, it suffices to prove that 
$f_{d}$ maps $R_{\scalebox{0.6}{$\wt\cN$}}^{\tilde \vt}$ into $\bn\cap \rtn$. To see this, observe first that $\tilde\vt=\vt$. Then consider the commutative diagram
\[
\xymatrixcolsep{3pc}
\xymatrix{
\UConf_{N/d}(\cC) \ar[r] \ar[d]_{c\,\mapsto\prod_{x\in c}x} & \UConf_{N}(\cC)  \ar[d]^{c\,\mapsto\prod_{x\in c}x} \\
\;\;\cC\;\; \ar[r]  & \;\;\cC\;\;
}
\]
where $\UConf_{N/d}(\cC)\rightarrow \UConf_{N}(\cC)$ is the pullback by $x\mapsto x^d$, and the bottom horizontal map is multiplication by $\big(\prod_{0\leqslant j<d}e^{2\pi i j/d}\big)^{N/d}=(-1)^{N(d-1)/d}$. At the level of fundamental groups, we obtain the following commutative diagram
\begin{equation}\label{eq:comdiag2}
\xymatrixcolsep{3pc}
\xymatrix{
\;\;B_{\wt \cN}^\star\;\; \ar[r]^{f_{d}} \ar[d]_{\ind_{\wt\cN}} & \;\;\bn \ar[d]^{\indn}\;\; \\
\pi_1(\cC,\tilde x_0) \ar[r]^\sim  & \pi_1(\cC,x_0)
}\quad .
\end{equation}
The result follows.
\end{proof}

\subsection{Trinomials}\label{sec:tri}
According to the previous section, it suffices to prove Theorem \ref{thm:main} for polynomials $\vp$ whose coefficients depend on a single parameter $y\in\cC$ and whose support $A$ satisfies $d:=d(A)=1$. 
In this section, we restrict our attention even further, namely to trinomials $\vp(x,y)=\alpha y^{a_1}x^{n_1}+\beta y^{a_2}x^{n_2}+\gamma y^{a_3}x^{n_3}$ with $\alpha, \beta, \gamma\in \cC$ and integers $n_1\leqslant n_2\leqslant n_3$. The aim of this section is to prove the following.

\begin{theor}\label{thm:trinomials}
Let $\{0\}\subset A\subset \Z^{2}$ be any support set not contained in a line, with associated integers $d:=d(A)$ and $\vt:=\vt(A)$, and satisfying $\vert A\vert=3$. For any generic Laurent polynomial $\vp\in \C^A$, the image of the braid monodromy map $\mb$ is the subgroup $\bnd\cap \rtn$ of $\bn$.
\end{theor}

For the sake of simplicity, we will assume that $\alpha=\beta=\gamma=1$. There is no loss of generality in doing so, since one of the coefficients $\alpha$, $\beta$, or $\gamma$ can be brought to $1$ by projective equivalence, while the remaining two can be compensated by a harmless change of coordinates $(x,y)\mapsto(ux,vy)$, where $u,v\in\cC$.

\subsubsection{Braids and their diagrams}\label{sec:braiddiagrams}
Braids in $\bn$ can be represented by braid diagrams similarly to the elements of the \textit{Artin braid group} $B_\cN:=\pi_1(\UConf_N(\C),\cN)$. Braids live in the 3-dimensional space $\C\times \text{timeline}$, and their diagrams are obtained by a real linear projection onto $\R\times \text{timeline}$ (see, for instance, \cite[Chapter 9.1.1]{farbmarg}). Here, we associate to any braid in $\bn$ a braid diagram using the argument map $\emph{\arg}:\cC\rightarrow S^1$\index{arg@$\arg$} as our projection, as done in \cite[Section 2.1]{EL}. 

Precisely, a braid $\beta\in \bn$ is an isotopy class of continuous maps $\beta:[0,1]\rightarrow \Conf_N(\cC)$
into the configuration space $\Conf_N(\cC)$ such that $\beta(0)=\beta(1)=\cN$. The argument map $\arg:\cC\rightarrow S^1$ defines a map from the open dense subset $U\subset \Conf_N(\cC)$ of configurations with pairwise distinct arguments to $\Conf_N(S^1)$. If $C\subset \Conf_N(\cC)$ is the complement of $U$, we can ensure, using an isotopy if necessary, that exactly two points of $\beta(s)$ have the same argument whenever $\beta(s)\in C$. At such a meeting point of the projection under $\arg$ of two strands of $\beta$, we can make sense of which of them has smaller modulus in $\cC$ than the other. We choose to represent the strand with smaller modulus as crossing over the other strand. 

This way, we can represent the element $\beta$ by the corresponding diagram in $S^1\times[0,1]$, picturing simultaneously the trajectories of the $N$ points $\beta(s)$, and keeping track of which strand is passing over which one at a crossing. In order to be able to draw braid diagrams in $S^1\times[0,1]$, we choose a fundamental domain $[\theta,\theta+2\pi[$ of $S^1$ and draw the diagrams in $[\theta,\theta+2\pi[\times [0,1]$ instead. Thus, strings in a diagram are allowed to hit boundary points $(\theta,s)$, disappear, and reappear at $(2\pi+\theta,s)$. 

In Figure \ref{fig:braids}, we represent the braid diagrams of some useful elements of $\bn$. From the braids $\emph{b_1}$\index{bj@$b_j$} and $\emph{\tau}$\index{tau@$\tau$}, we define $\emph{b_j}:=\tau^{j-1} b_1\tau^{-j+1}$, $j\in\{1,\ldots,N\}$. By a slight abuse of language, we will identify braids with their diagrams and therefore view braid diagrams as elements of the relevant braid group.

\begin{figure}
\centering
\def\svgwidth{\linewidth} 
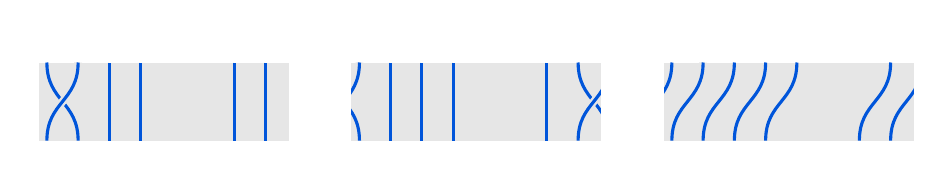
\caption{The braid diagrams of $b_1$, $b_N$ and $\tau$ in $\bn$.}
\label{fig:braids}
\end{figure}

\begin{lemma}\label{lem:kernel}
The elements $b_1,\ldots,b_N \in \bn$ generate $\ker(\indn)$.
\end{lemma}

\begin{proof}
Clearly, the elements $b_j$, $j\in\{1,\ldots, N\}$, belong to $\ker(\indn)$. As shown in \cite[Lemma 2.2]{EL}, the elements $b_j$ generate $\bn$ together with $\tau$. Moreover, any element of $\bn$ can be written as $b\tau^k$ for some integer $k$, where $b$ is a word in the $b_j$'s. Indeed, the commutator $\tau b_j \tau^{-1}b_j^{-1}$ equals the product $b_{j+1}b_j^{-1}$, or equivalently $\tau b_j = b_{j+1}\tau$. 

Finally, since $\indn(b\tau^k)=k$, the element $b\tau^k$ lies in $\ker(\indn)$ if and only if $k=0$. Thus, we have the inclusion $\ker(\indn)\subset \left\langle b_1,\ldots,b_N\right\rangle$, and the result follows.
\end{proof}

\subsubsection{Further reductions}

In this section, we state some elementary lemmas that allow us to restrict our attention to trinomials $\vp(x,y)$ with specific support sets.

\begin{lemma}\label{lem:red2}
Let $\{0\}\subset A\subset \Z^{2}$ be any support set not contained in a line, and let $\vp\in \C^A$ be any generic Laurent polynomial. Then, the conclusion of Theorem \ref{thm:main} holds for $\vp$ if and only if it holds for $\wt\vp(x,y):=\vp(x^\varepsilon,y^\nu)$, where $\vert\varepsilon\vert=\vert\nu\vert=1$.
\end{lemma}

\begin{proof}
The statement follows from the fact that the change of variables $(x,y)\mapsto(x^\varepsilon,y^\nu)$ induces an isomorphism between the two coverings
\[
\{(x,y)\in \cC\times(\C\setminus \sB) : \vp(x,y)=0\} \to \C\setminus \sB
\]
and
\[
\{(x,y)\in \cC\times(\C\setminus \wt\sB) : \wt\vp(x,y)=0\} \to \C\setminus \wt\sB.
\]
We obtain the corresponding commutative diagrams, analogous to \eqref{eq:comdiag1} and \eqref{eq:comdiag2}. The details are left to the reader.
\end{proof}

\begin{sledst}\label{cor:reduction}
Theorem \ref{thm:trinomials} holds if and only if it holds for at least one trinomial in each orbit of the following group action:
\[
(i,j,d,\nu)\in \Z^3\times \{-1,1\}, \quad (i,j,d,\nu) \cdot \vp(x,y):=x^i y^j \vp(x^{d},y^\nu).
\]
In particular, it suffices to prove Theorem \ref{thm:trinomials} for trinomials $\vp$ whose support set $A$ is of one of the following types:
\begin{itemize}
    \item[(1)] $A= \left\lbrace (0,0), (m,a), (n,b)\right\rbrace$ with $0<m<n$, $\gcd(m,n)=1$, and $na - mb > 0$,
    \item[(2)] $A= \left\lbrace (0,a), (0,b), (1,0)\right\rbrace$ with $a < b$.
\end{itemize}
\end{sledst}

\begin{proof}
The first part of the statement follows from Proposition \ref{prop:redd=1}, Lemma \ref{lem:red2}, and the fact that the covering
\[
\{(x,y)\in \cC\times(\C\setminus \sB) : \vp(x,y)=0\} \to \C\setminus \sB
\]
associated to $\vp$ is the same as the covering associated to $x^i y^j \cdot \vp$.

For the second part, observe that the involutions $x\mapsto x^{-1}$ and $y\mapsto y^{-1}$ induce reflections along the horizontal and vertical axes, respectively, in the monomial lattice $\Z^2$. The multiplication by $x^i y^j$ corresponds to translation by the vector $(i,j)$ in $\Z^2$.

If every vertical line intersects $A$ in at most one point, then we can translate $A$ so that $\{(0,0)\} \subset A \subset \N\times \Z$. Using the involution $y\mapsto y^{-1}$, we may assume $na - mb > 0$. Finally, acting by $d$, we can achieve $\gcd(m,n)=1$, and thus reduce to case (1).

If there exists a vertical line intersecting $A$ in two points, we can translate $A$ so that this line becomes the vertical axis, and the third point of $A$ lies on the horizontal axis. Then, using the involutions $x\mapsto x^{-1}$ and $y\mapsto y^{-1}$, we may assume $A = \left\lbrace (0,a), (0,b), (d,0) \right\rbrace$ with $a < b$ and $d > 0$. Rescaling the first coordinate by $d$ gives case (2).
\end{proof}

\subsubsection{Trinomials of type (2)} \label{sec:type2}
As we have seen in Corollary \ref{cor:reduction}, we can restrict ourselves to two types of trinomials when proving Theorem \ref{thm:trinomials}, namely trinomials of type (1) or (2). In this subsection, we prove the theorem for trinomials of type (2). Let $\vp(x,y)=(y^a+y^b)+x$ with integers $a < b$. In this case, $N=1$ and $B_1^\star = \pi_1(\cC)$. Additionally, we have $d = \vt = 1$.

The branching locus $\sB \subset \C$ is defined by the equation $y^a + y^b = 0 \Leftrightarrow y^a(y^{b-a} + 1) = 0$, which has $b-a$ simple roots in $\cC$. When $y$ travels along a small circle around one of these roots, the root $x = -(y^a + y^b)$ of $\vp(x)$ traces a simple loop around $0$. The corresponding element in $\im(\mb)$ is thus a generator of $B_1^\star$. Since $d = \vt = 1$ and $R_1^1 = B_1^\star$, this proves Theorem \ref{thm:trinomials} for trinomials of type (2). \hfill $\qed$

\subsubsection{The branching locus}\label{sec:bifurcation}
From now until the end of Section \ref{sec:prooftrinomial}, we restrict to trinomials of type (1), according to the dichotomy given in Corollary \ref{cor:reduction}. This means that $\vp(x,y) = 1 + y^a x^m + y^b x^n$ with $0 < m < n$, $\gcd(m,n) = 1$, and $na - mb > 0$. 

Let us determine the branching locus $\sB \subset \C$ of $\vp$. To that end, consider for a moment the polynomial $f(x) = 1 + u x^m + v x^n$. Aside from the trivial case $v = 0$, the polynomial $f(x)$ has strictly fewer than $n$ roots if the triple $(u, v, x)$ satisfies
\begin{equation}\label{eq:discriminant}
\left\lbrace
    \begin{array}{rl}
     1 + ux^m + vx^n &= 0 \\
     mux^{m} + nvx^n &= 0 
    \end{array}
\right.
\Leftrightarrow
\left\lbrace
    \begin{array}{rl}
     v &= \frac{m}{n-m} x^{-n} \\
     u &= \frac{n}{m-n} x^{-m} 
    \end{array}
\right. .
\end{equation}
Therefore, the polynomial $f$ is singular if and only if
\[
v \left( \left(v \cdot \frac{n - m}{m} \right)^m - \left( u \cdot \frac{m - n}{n} \right)^n \right) = 0, 
\]
which holds because $\gcd(m,n) = 1$.

It follows that the branching locus is defined by the equation
\[
y^b \left( \left(y^b \cdot \frac{n - m}{m} \right)^m - \left( y^a \cdot \frac{m - n}{n} \right)^n \right) = 0.
\]
Therefore, $\sB$ is the disjoint union of $0 \in \C$ with the $\emph{\delta} := na - mb$ \index{delta@$\delta$} non-zero roots of the binomial equation
\begin{equation*}
    \left( y^b \cdot \frac{n - m}{m} \right)^m = \left( y^a \cdot \frac{m - n}{n} \right)^n.
\end{equation*}
In particular, the points of $\sB \setminus \{0\}$ are equidistributed on the circle $\{ y \in \C : \, \vert y \vert = \rho \}$ for some $\emph{\rho} := \rho(a, b, m, n) > 0$.\index{rho@$\rho$} 

\subsubsection{The coamoeba of $\vp$}\label{sec:coamoeba}
Denote by $\emph{\Arg_\vp}$\index{Argphi@$\Arg_\vp$} the (closed) coamoeba of $\vp$, that is, the closure in $(S^1)^2$ of the subset 
\begin{equation*}
    \left\lbrace (\theta,\nu) \in (S^1)^2 :\, \theta = \arg(x), \; \nu = \arg(y), \; \vp(x,y) = 0 \right\rbrace.
\end{equation*}
The set $\Arg_\vp$ can easily be described in terms of $\Arg_\ell$, where $\emph{\ell}(x,y) := 1 + x + y$. 
Indeed, the polynomial $\vp$ is projectively equivalent to the composition of the polynomial $\ell$ with the map $\psi\colon (x,y) \mapsto (y^a x^m, y^b x^n)$.
Therefore, $\Arg_\vp$ is the pullback of $\Arg_\ell$ under the covering map $(\theta,\nu) \mapsto (m\theta + a\nu, n\theta + b\nu)$ from $(S^1)^2$ to itself.

Straightforward computations show that the set $\Arg_\ell$ is the union of the two closed triangles in $(S^1)^2\simeq (\R/2\pi\Z)^2$ bounded by the three geodesics
\begin{equation}\label{eq:geod}
    \{\theta = \pi\}, \quad \{\nu = \pi\}, \quad \text{and} \quad \{\theta - \nu = \pi\}.
\end{equation}
These two triangles are copies of the Newton polygon of $\ell$, rotated by $\pi/2$ and $-\pi/2$ respectively.
As a consequence, the closed coamoeba $\Arg_\vp$ consists of $2\delta$ triangles. Each such triangle is bounded by three geodesics. These are precisely the pullbacks, under the covering map $\psi$, of the three geodesics in \eqref{eq:geod}. 

We refer to the boundary geodesics of $\Arg_\vp$ with respective slopes $(b,-n)$, $(-a,m)$ and $(a-b,n-m)$ as the \emph{D-}, \emph{L-}, and \emph{R-geodesics}; see Figures \ref{fig:coamoebaline} and \ref{fig:coamoebaphi}. 
These slopes are the outer normal vectors to the edges of the convex hull $\conv(A)$. The letters D, L, and R stand for down, left, and right, respectively, in reference to the positions of the corresponding edges of $\conv(A)$. 

\begin{figure}
\centering
\def\svgwidth{\linewidth} 
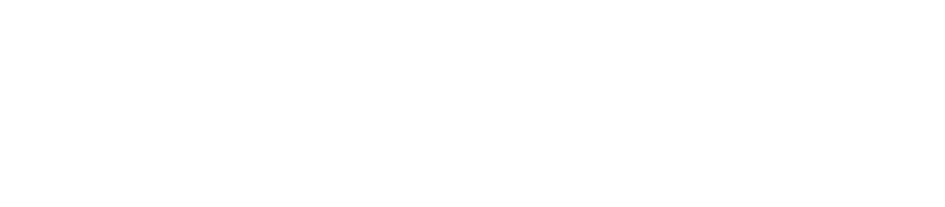
\caption{The coamoeba $\Arg_\ell$ of the line $\ell$ in the fundamental domain $[-\pi,\pi)$ of $(S^1)^2\simeq (\R/2\pi\Z)^2$.}
\label{fig:coamoebaline}
\end{figure}

The braid monodromy map $\mb$ tracks the configuration of points given by horizontal sections of the curve $\{(x,y) \in (\cC)^2 :\, \vp(x,y) = 0\}$.
Below, we argue that the horizontal sections of $\Arg_\vp$ can be used to determine the image of $\mb$.

While generic horizontal sections of $\{(x,y)\in (\cC)^2:\, \vp(x,y)=0\}$ consist of $n$ distinct points, generic horizontal sections of $\Arg_\vp$ consist of $n$ distinct line segments. 
Indeed, observe that each such segment has a single endpoint lying on a D-geodesic, and that the D-geodesics intersect any horizontal section exactly $n$ times in total.
This remains true unless the horizontal section passes through an intersection point of an L-geodesic with an R-geodesic, in which case the corresponding segment is bounded by two (possibly distinct) points on D-geodesics. 

For practical purposes, we refer to the intersection points between L- and R-geodesics as the \textit{singular vertices} of $\Arg_\vp$.
We denote by $\emph{\sS} \subset (S^1)^2$ \index{Ss@$\sS$} the set of all such points. \bigskip

\begin{figure}
\centering
\def\svgwidth{\linewidth} 
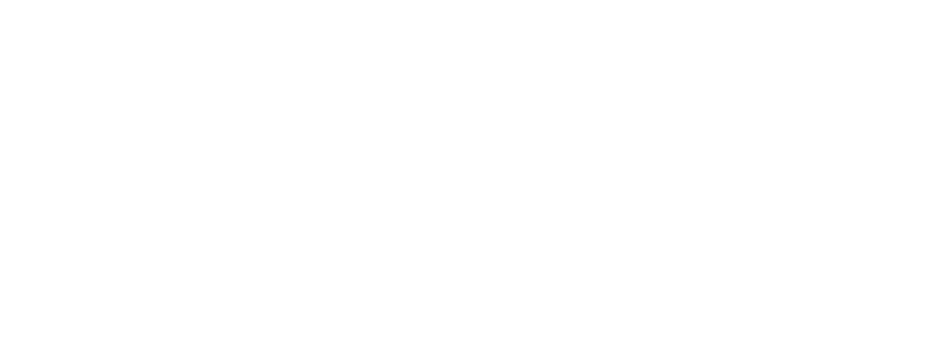
\caption{The coamoeba of $\vp$ with $A = \{(0,0), (2,4), (5,2)\}$.}
\label{fig:coamoebaphi}
\end{figure}

\begin{lemma}\label{lem:singvert}
The $\delta$ singular vertices of $\Arg_\vp$ have pairwise distinct $\nu$-coordinates. 
\end{lemma}

\begin{proof}
The set $\sS \subset \Arg_\vp$ is the preimage under the covering $(\theta,\nu) \mapsto (m\theta + a\nu, n\theta + b\nu)$ from $(S^1)^2$ to itself of the vertex $(0,\pi)$ of $\Arg_\ell$. 
This is because this covering maps L- and R-geodesics respectively to $\{\nu = \pi\}$ and $\{\theta - \nu = \pi\}$, which intersect at $(0,\pi)$.
Therefore, the set $\sS$ consists of $\delta$ points and, up to translation, coincides with the projection onto $(\R/2\pi\Z)^2$ of the lattice
\[
\frac{1}{\delta}
\begin{pmatrix}
    b & -a \\
    -n & m
\end{pmatrix}
\cdot (2\pi\Z)^2.
\]

Two points in this lattice have the same $\nu$-coordinate in $(\R/2\pi\Z)^2$ if their difference, which is of the form 
\[
\frac{2\pi}{\delta} \left( \lambda \begin{pmatrix} b \\ -n \end{pmatrix} + \mu \begin{pmatrix} -a \\ m \end{pmatrix} \right),
\]
has a second coordinate in $2\pi\Z$. Thus, the pair $(\lambda,\mu)$ has to satisfy $\mu m - \lambda n = \kappa \delta$ for some $\kappa \in \Z$. The general solution to this equation is the sum of a particular solution, for instance $(\lambda,\mu) = -\kappa \cdot (a,b)$, and a solution to the homogeneous equation $\mu m - \lambda n = 0$, which is of the form $(\lambda,\mu) = \ell \cdot (m,n)$ since $\gcd(n,m) = 1$. Thus,
\[
(\lambda, \mu) = (\ell m - \kappa a, \ell n - \kappa b),
\]
and the first coordinate of the corresponding lattice point is
\[
\frac{2\pi}{\delta} \left( b(\ell m - \kappa a) - a(\ell n - \kappa b) \right) = \frac{2\pi}{\delta} \ell (bm - an) = -2\pi \ell \in 2\pi\Z.
\]
Hence, the two points have the same image in $(\R/2\pi\Z)^2$. The result follows.
\end{proof}

\begin{lemma}\label{lem:coamoeba}
$\,$

1. For any $y \in \cC$, the intersection of $\{(\theta,\nu) \in (S^1)^2 : \nu = \arg(y)\}$ with $\Arg_\vp$ consists of $n$ connected components if $\{\nu = \arg(y)\}$ does not pass through $\sS$. Otherwise, the intersection consists of $n-1$ components. 

2. The map $\Arg\colon (\cC)^2 \to (S^1)^2$ induces a bijection between $\sS$ and the set of points $(x,y) \in (\cC)^2$ such that $x$ is a multiple root of $\vp(x,y)$. In particular, the projection of $\sS$ onto $\{1\} \times S^1$ equals $\arg(\sB \setminus \{0\})$.

3. For any $y \in \C \setminus \sB$, each connected component $\mathscr{C}$ of
\[
\big( \{ \nu = \arg(y) \} \cap \Arg_\vp \big) \setminus \sS
\]
contains the projection under $\Arg$ of a single point $p \in \{(x,y) \in (\cC)^2 : \vp(x,y) = 0\}$. If $|y|$ is arbitrarily small, then $\Arg(p)$ is arbitrarily close to the D-geodesic bounding $\mathscr{C}$. If $|y|$ is arbitrarily large, then $\Arg(p)$ is arbitrarily close to the other geodesic bounding $\mathscr{C}$.
\end{lemma}

\begin{proof}
1. This part follows from the discussion prior to Lemma~\ref{lem:singvert} and the lemma itself. 

2. Assume that $x$ is a multiple root of $\vp(x,y)$ and let $(\theta, \nu) := (\arg(x), \arg(y))$. From the right-hand side of \eqref{eq:discriminant}, we deduce that $n\theta + b\nu = 0$ and $m\theta + a\nu = \pi$. These two equations characterise the points in $\sS$, according to the proof of Lemma~\ref{lem:singvert}. Conversely, any point of $\sS$ is of the form $(\arg(x), \arg(y))$ for some pair $(x,y) \in (\cC)^2$ such that $x$ is a multiple root of $\vp(x,y)$, by definition of $\sS$. The rest follows.

3. Since the polynomial $\vp$ is of type (1) (see Corollary~\ref{cor:reduction}), the outer normal to the edge $\conv\{(0,0),(n,b)\}$ of $\conv(A)$ points downward, while the outer normals to the two remaining edges point upward. By the Newton–Puiseux theorem, the $n$ solutions to $\vp(x,y) = 0$ are asymptotically equivalent to
\[
(x,y) = \left( (-1/y^b)^{1/n}, y \right)
\]
for any of the $n$ determinations of $(-1/y^b)^{1/n}$ as $|y| \to 0$. Similarly, as $|y| \to \infty$, the $n$ solutions split into two groups:
\[
(x,y) \sim \left( (-1/y^a)^{1/m}, y \right) \quad \text{and} \quad (x,y) \sim \left( (-y^{a-b})^{1/(n-m)}, y \right).
\]

The images of these parametrisations under $\Arg$ map surjectively onto the D-, L-, and R-geodesics, respectively. This yields the second part of the statement.

To prove the first part, note that if $\{\nu = \arg(y)\}$ is disjoint from $\sS$, then for sufficiently small $|y|$, the above asymptotics ensure that each of the $n$ points in $\{(x,y): \vp(x,y) = 0\}$ contributes to a distinct component of $\{\nu = \arg(y)\} \cap \Arg_\vp$. This configuration persists as $|y|$ increases from $0$ to $\infty$, while fixing $\arg(y) = \nu$. This is because the set $\{\vp(x,y) = 0\}$ varies continuously with $y$ and always contains $n$ distinct points for $y \notin \sB$.

If instead $\{\nu = \arg(y)\}$ intersects $\sS$, and $y \notin \sB$, then the image of $\{x:\vp(x,y)=0\}$ under $\Arg$ still consists of $n$ distinct points. Assume for contradiction that several points in the image lie in the same component of $\left( \{ \nu = \arg(y) \} \cap \Arg_\vp \right) \setminus \sS$. By continuity, this configuration would persist upon a small perturbation of $\nu$, and therefore contradict the above paragraph. The result follows.
\end{proof}

\subsubsection{Proof of Theorem \ref{thm:trinomials}}\label{sec:prooftrinomial}
Here, we compute the image under $\mb$ of a generating set of $\pi_1(\C\setminus \sB,y_0)$, relying essentially on Lemma \ref{lem:coamoeba}. 

To do so, we fix a graph $\emph{\Gamma}\subset  \C\setminus \sB$\index{Gamma@$\Gamma$} with $y_0\in\Gamma$ and such that $\pi_1(\C\setminus \sB,y_0)=\pi_1(\Gamma,y_0)$. Concretely, we take $\Gamma$ to be the union of two circles $$C_\varepsilon:=\{y\in\cC:\, \vert y\vert=\varepsilon\} \; \text{ and }\; C_M:=\{y\in\cC:\, \vert y\vert=M\}$$ for $\varepsilon,M>0$ arbitrarily small and large respectively, together with $\delta$ many segments joining $C_\varepsilon$ to $C_M$. Each segment is contained in a half ray $\{y\in\cC:\, \arg(y)=\theta\}$, the segments have equidistributed arguments $\theta$ on $S^1$ and share the same distance to $\sB\setminus \{0\}$. This is possible since $\sB\setminus \{0\}$ is equidistributed on some circle $C_\rho$, see Section \ref{sec:bifurcation}. All together, the graph $\Gamma$ looks like a circular railway track as pictured in Figure \ref{fig:gamma}.

\begin{figure}
\centering
\def\svgwidth{\linewidth} 
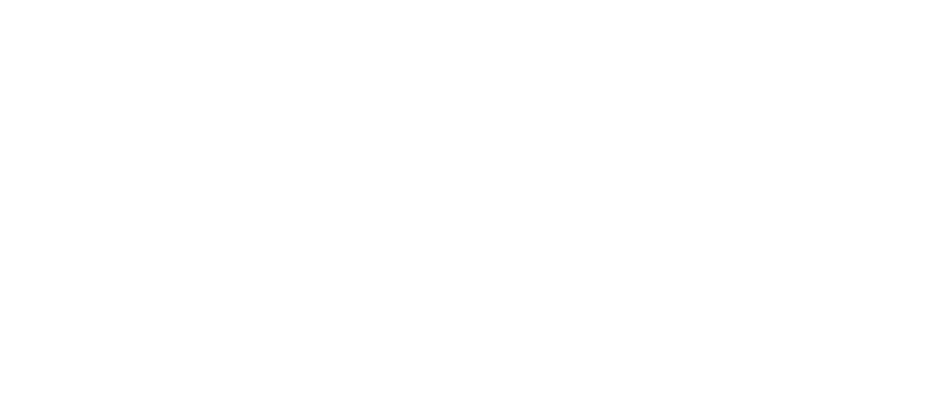
\caption{The graph $\Gamma$ and the generators $\ell_0,\ell_1,\ldots,\ell_\delta$ of $\pi_1(\Gamma,y_0)$.}
\label{fig:gamma}
\end{figure}

We fix the base point $y_0\in \Gamma$ as a midpoint on an edge of $\Gamma$ contained in $C_\varepsilon$. In particular, the modulus $|y_0|=\varepsilon$ is arbitrarily small. We also fix the set of generators $\ell_0,\ell_1,\ldots,\ell_\delta$ of $\pi_1(\Gamma,y_0)$ as pictured in Figure \ref{fig:gamma}. 
Recall the elements $b_1,\ldots,b_n,\tau \in \bn$ defined in Figure \ref{fig:braids}. We have the following.

\begin{lemma}\label{lem:braidtrinomial}
For any $j \in \{1,\ldots,\delta\}$, we have $\mb(\ell_j) = b_k$ for some $k \in \{1,\ldots,n\}$. Conversely, we have the inclusion $\{b_1,\ldots,b_n\} \subset \im(\mb)$. Moreover, we have $\mb(\ell_0) = \tau^{b}$.
\end{lemma}

Before proving the lemma, let us fix the labelling of the roots of $\vp(x,y_0)$ as follows: choose an arbitrary root to be labelled $1$ and label the remaining roots by increasing argument. Consequently, their projection under $\Arg$ appears in order on the horizontal section $\{\nu = \arg(y_0)\}$ of $\Arg_\vp$, see Figure~\ref{fig:ordering}. The horizontal section $\{\nu = \arg(y_0)\}$ splits the union of the D-geodesics into $n$ open segments, which we refer to as \textit{D-segments}. Recall that, by the construction of $\Gamma$, the choice of $y_0$, and Lemma~\ref{lem:coamoeba}, the upper endpoint of each D-segment is arbitrarily close to a labelled point of $\Arg(\{\vp(x,y_0)=0\})$. We label the D-segments accordingly. The union of the D-geodesics and the horizontal geodesic $\{\nu = \arg(y_0)\}$ splits $(S^1)^2$ into $n$ components, which we refer to as \textit{D-parallelograms}, and which we label according to the D-segment on their left, see again Figure~\ref{fig:ordering}. Finally, the $\gcd(n,b)$ many D-geodesics split $(S^1)^2$ into $\gcd(n,b)$ components, which we refer to as the \textit{D-stripes}.

\begin{figure}
\centering
\def\svgwidth{\linewidth} 
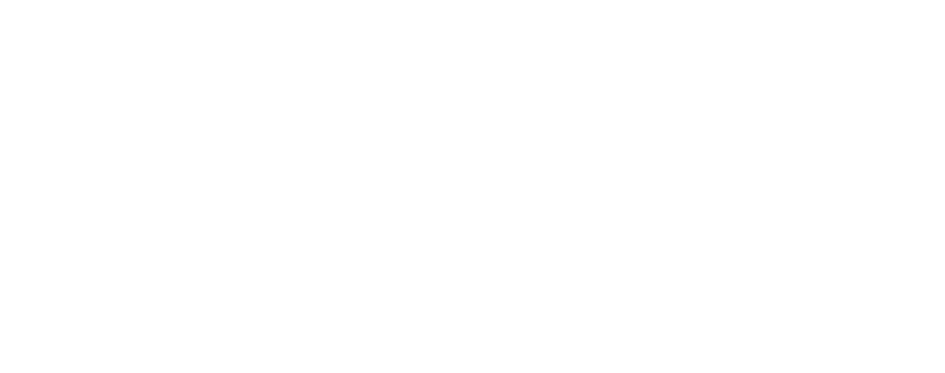
\caption{The labelling of the roots of $\vp(x,y_0)$, of the D-segments, and the first D-parallelogram $P_1$.}
\label{fig:ordering}
\end{figure}

\begin{proof}
Fix a parametrisation $s \mapsto y(s)$ of $\ell_0$.  
According to item~3 in Lemma~\ref{lem:coamoeba}, the projection under $\Arg$ of $\{(x,y(s)) \in (\cC)^2 : \vp(x,y(s)) = 0\}$ is arbitrarily close to the intersection of the horizontal section $\{\nu = \arg(y(s))\}$ with the union of the D-geodesics. Thus, the projection under $\Arg$ of each point of this set follows a trajectory as close as desired to a straight line of slope $(-b,n)$. Thus, we have $\mb(\ell_0) = \tau^b$.

For any $j \in \{1,\ldots,\delta\}$, we can find a locally injective parametrisation $[0,1] \rightarrow \cC$, $s \mapsto y(s)$ of $\ell_j$ such that there exist $0 \leqslant s_1 < s_2 \leqslant 1$ satisfying $y([0,s_1]) = y([s_2,1])$ and such that the restriction of $s \mapsto y(s)$ to $]s_1,s_2[$ is injective. Observe that $(s_1, s_2) = (0,1)$ exactly when $j = 1$. Let us restrict to the case $j > 1$, as the result for $j = 1$ is similar.

As we have seen above, the $n$ points of $\Arg\big(\{(x,y(s)) \in \cC : \vp(x,y(s)) = 0\}\big)$ follow tightly the D-geodesics downwards in $S^1 \times S^1$ when $s$ travels along $[0,s_1]$, and follow the same trajectory upwards between $s_2$ and $1$. We now focus on the segment $[s_1,s_2]$. Let $\nu_1, \nu_2 \in S^1$ be such that $\arg\big(y([s_1,s_2])\big) \subset S^1$ is the arc between $\nu_1$ and $\nu_2$ with $\arg(y(s_1)) = \arg(y(s_2)) = \nu_1$. We denote this arc, somewhat abusively, by $[\nu_1, \nu_2]$.

We claim that the intersection of $\Arg_{\vp}$ with $S^1 \times [\nu_1, \nu_2]$ consists of exactly $n-1$ connected components. Indeed, the intersection of $S^1 \times [\nu_1, \nu_2]$ with the union of the D-geodesics consists of $n$ segments. Each component of $\Arg_{\vp} \cap (S^1 \times [\nu_1, \nu_2])$ contains exactly one such segment, except the one component that contains a singular vertex of $\Arg_{\vp}$ (by item~2 in Lemma~\ref{lem:coamoeba}, there is exactly one such vertex since $\ell_j$ encloses exactly one point of $\sB$). The claim follows.

According to item~3 in Lemma~\ref{lem:coamoeba}, when $s$ goes from $s_1$ to $s_2$, each of the $n$ points of the subset $\Arg\big(\{(x,y(s)) \in \cC : \vp(x,y(s)) = 0\}\big)$ travels arbitrarily close to the boundary of one of the connected components of $\Arg_{\vp} \cap (S^1 \times [\nu_1, \nu_2])$, see Figure~\ref{fig:travel}. More precisely, each of these $n$ points travels horizontally when $y(s)$ traverses part of the circles $C_\varepsilon$ and $C_M$, and follows tightly part of a boundary geodesic of $\Arg_{\vp}$ otherwise. In particular, exactly two branches of $s \mapsto \Arg\big(\{(x,y(s)) \in \cC : \vp(x,y(s)) = 0\}\big)$ cross each other at the unique point of $\sS \cap (S^1 \times [\nu_1, \nu_2])$.
The image of $\ell_j$ under $\mb$ is therefore of the form $b_k^{\pm 1}$ for some $1 \leqslant k \leqslant n$. Plainly, the integer $k$ is the label of the D-parallelogram containing the singular vertex involved. Although it is not of crucial importance, it can be shown that, according to our conventions, we have $\mb(\ell_j) = b_k$, see Figure~\ref{fig:travel}.

In order to conclude, we need to prove that all $b_k$, $1 \leqslant k \leqslant n$, are in the image of $\mb$. It is not necessarily true that every $b_k$ appears as $\mb(\ell_j^{\pm 1})$, since it is not guaranteed that every D-parallelogram contains a singular vertex. However, we claim that any $b_k$ appears as $\mb(\ell_j^{\pm 1} \circ \ell_0^p)$ for an appropriate integer $p$. According to the above discussion, this is equivalent to saying that each D-stripe contains a singular vertex. But this is clear, since at least one such stripe contains a singular vertex and since, for two different stripes $S_1$ and $S_2$, we have that $\Arg_\vp \cap S_1$ is a translation of $\Arg_\vp \cap S_2$ in the argument torus $S^1 \times S^1$. Therefore, any such stripe contains a singular vertex. This implies the inclusion $\{b_1,\ldots,b_n\} \subset \im(\mb)$ and concludes the proof.
\end{proof}

\begin{figure}
\centering
\def\svgwidth{\linewidth} 
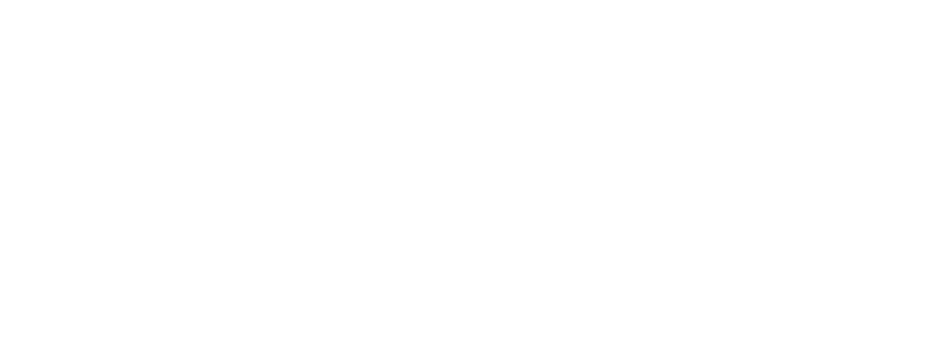
\caption{The image of $\ell_j$, $j>1$, under $\mb$.}
\label{fig:travel}
\end{figure}

\begin{proof}[Proof of Theorem \ref{thm:trinomials}]
By Corollary \ref{cor:reduction} and Section \ref{sec:type2}, it suffices to prove the theorem in the case of trinomials of type (1). In this case, we have $d = \gcd(n, m) = 1$, and thus $\bnd = \bn$. 
Therefore, we only need to show that $\im(\mb) = \rtn$.
Here, we have $\vt = b$. 
Since the elements $b_1, \ldots, b_n$ generate $\ker(\indn)$ by Lemma \ref{lem:kernel}, 
and $\indn(\tau^b) = b$, the elements $b_1, \ldots, b_n, \tau^b$ generate $\rtn$. 
The result now follows from Lemma \ref{lem:braidtrinomial}.
\end{proof}

\subsection{Proof of Theorem \ref{thm:main}}

According to Proposition \ref{prop:reductionk=1}, we can restrict our attention to the case $k = 1$ while proving Theorem \ref{thm:main}.
By Proposition \ref{prop:redd=1}, we can even restrict to supports $A$ such that $d(A) = 1$.
Thus, we fix a support $\{(0,0)\} \subset A \subset \Z^2$ with $d(A) = 1$ and horizontal width $N := N(A)$,
and fix a generic polynomial $\vp \in \CC^A$.

Let us briefly describe the overall strategy of the proof.
Let $\rho:[0,1] \to \C^A$, $s \mapsto \vp_s$, be a continuous path 
from $\vp$ to some other polynomial $\wt\vp \in \C^A$, and denote by $\sB_s \subset \C$ the branching locus associated with $\vp_s$.
Fix a continuous path $\gamma: s \mapsto y_{0,s} \in \C \setminus \sB_s$ of base points. 
The pair $(\rho, \gamma)$ is said to be \textit{$A$-suitable} if

-- the support $A_s$ of $\vp_s$ has horizontal width $N(A_s) = N$ for all $0 \leqslant s \leqslant 1$,

-- the family of fundamental groups $\pi_1(\C \setminus \sB_s, y_{0,s})$ is locally trivial for $0 \leqslant s < 1$.\\ 
Clearly, the pair $(\rho, \gamma)$ induces an injective morphism $\im(\mu_{\wt\vp}) \rightarrow \im(\mb)$.
In order to prove Theorem \ref{thm:main}, we will take $\wt\vp$ to be a trinomial and determine the image of the map $\im(\mu_{\wt\vp}) \rightarrow \im(\mb)$ using Theorem \ref{thm:trinomials}.
Using the Euclidean algorithm \cite[Proposition 4.1]{EL}, we will then argue that the images of these morphisms generate the expected braid monodromy group once sufficiently many trinomials $\wt\vp$ are considered.
It will be crucial to observe that different $A$-suitable pairs $(\rho, \gamma)$ lead to different maps $\im(\mu_{\wt\vp}) \rightarrow \im(\mb)$. We will choose these pairs carefully.
In that regard, it will be helpful to keep the following in mind. 

\begin{rem}\label{rem:B_n,d}
For a general support $\{(0,0)\} \subset A \subset \Z^2$, the base point $\cN$ of the braid group $\bn := \pi_1(\UConf_N(\cC), \cN)$ is globally invariant under the map $x \mapsto e^{2i\pi/d}x$, where $d := d(A)$. 
In turn, this map induces an automorphism $\iota$ on $\bn$, and 
$\bnd$ can be alternatively defined as the invariant subgroup under $\iota$. 
If the base point $\cN$ of $\bn$ is not globally invariant under $x \mapsto e^{2i\pi/k}x$ for some integer $k \geqslant 2$, a choice must be made to define the subgroup $B_{\sN,k}^\star \subset \bn$, namely, we must specify a path in $\UConf_N(\cC)$ joining $\cN$ to a configuration invariant under $x \mapsto e^{2i\pi/k}x$.
\end{rem}

Recall that the Hausdorff distance on compact subsets of $\cC$ induces a metric on the configuration space $\UConf_N(\cC)$. 
Fix a binomial $\beta \in \C^{A_0}$, where $A_0 \subset A$ has horizontal width $N$. 
Fix $\varepsilon > 0$ arbitrarily small, and choose an open ball $\cV \subset \C^A$ containing $\beta$ such that 

-- for any $\hat\vp \in \cV$, the associated branching locus $\sB_{\hat\vp} \subset \C$ does not contain $1$,

-- for any $\hat\vp \in \cV$, the configuration $\cN_{\hat\vp} := \{x \in \cC : \wt\vp(x,1) = 0\}$ is contained in the $\varepsilon$-neighbourhood of $\cN_\beta$ in $\UConf_N(\cC)$.\\ 
The existence of $\cV$ follows from the facts that $\sB_{\beta} = \{0\}$ (and thus $1 \notin \sB_{\beta}$), and that both $\sB_{\hat\vp}$ and $\cN_{\hat\vp}$ depend continuously on $\hat\vp$.

As discussed in Remark \ref{rem:B_n,d}, any path $\rho \subset \cV$ from $\hat\vp$ to $\beta$ allows one to define the subgroup $B_{\cN_{\hat\vp},\tilde d}^\star \subset B_{\cN_{\hat\vp}}^\star$ for any divisor $\tilde d$ of $N$. This is because the set $\cN_\beta$ is invariant under multiplication by $e^{2i\pi/N}$. 
Since $\cV$ is simply connected, this subgroup is independent of the path $\rho$.
Moreover, any such path defines an isomorphism between $B_{\cN_{\hat\vp}}^\star$ and $B_{\cN_\beta}^\star$, which is also path-independent for the same reason. 
In particular, the braid groups $B_{\cN_{\hat\vp}}^\star$, $\hat\vp \in \cV$, are mutually isomorphic. \hfill $(\ast)$

Since we are allowed to choose $\vp$ in some Zariski-open subset of $\CC^A$, we take $\vp \in \cV$. 
In particular, we can make sense of the subgroup $B_{\sN,\tilde d}^\star \subset \bn$ for any divisor $\tilde d$ of $N$. We have the following:

\begin{lemma}
For any $A_0 \subset \wt A \subset A$ with $\vert \wt A \vert = 3$ and any generic $\wt\vp \in \C^{\wt A} \cap \cV$, there exists a path $\rho \subset \cV$ from $\vp$ to $\wt\vp$ such that the pair $(\rho, \gamma)$ is $A$-suitable, where $\gamma$ is constant and equal to $1$. 
For any such path, the image of the corresponding morphism $\im(\mu_{\wt\vp}) \rightarrow \im(\mb)$ is 
$$
B_{\sN,d(\wt A)}^\star \cap R_{\sN}^{\vt(\wt A)} \subset \bn.
$$
\end{lemma}

\begin{proof}
It is an elementary fact that the cardinality of $\sB_{\hat\vp}$ is constant and maximal on a Zariski-open subset of $\C^A$.
In particular, the fundamental group $\pi_1(\C \setminus \sB_{\hat\vp}, 1)$ is locally constant on an open dense subset of $\cV$. This proves the existence of the $A$-suitable pair $(\rho, \gamma)$. 
By Theorem \ref{thm:trinomials}, we have
$$
\im(\mu_{\wt\vp}) = B_{\cN_{\wt\vp},d(\wt A)}^\star \cap R_{\cN_{\wt\vp}}^{\vt(\wt A)} \subset B_{\cN_{\wt\vp}}^\star.
$$
Therefore, it suffices to show that the isomorphism $(\ast)$ between $B_{\cN_{\wt\vp}}^\star$ and $\bn$ maps
$$
B_{\cN_{\wt\vp},d(\wt A)}^\star \text{ to } B_{\cN,d(\wt A)}^\star \quad \text{and} \quad R_{\cN_{\wt\vp}}^{\vt(\wt A)} \text{ to } R_\cN^{\vt(\wt A)}.
$$
The latter is clear since $(\ast)$ maps $\ind_{\cN_{\wt\vp}}$ to $\indn$. 
The former follows from the definition of the subgroup $B_{\cN_{\hat\vp},\tilde d}^\star \subset B_{\cN_{\hat\vp}}^\star$ for any divisor $\tilde d$ of $N$ and $\hat\vp \in \cV$. 
\end{proof}

Let $\{\wt A_i\}_{i \in I}$ be the set of all supports $\wt A$ as in the lemma above. 
Denote $\tilde d_i := d(\wt A_i)$ and $\tilde \vt_i := \vt(\wt A_i)$. 
Then, we have $\gcd(\{\tilde d_i\}_{i \in I}) = d(A) = 1$. 
By the lemma above, we have $B_{\cN,\tilde d_i}^\star \cap R_\cN^0 \subset \im(\mb)$ for all $i \in I$. 
By the Euclidean algorithm \cite[Proposition 4.1]{EL} and Lemma \ref{lem:kernel}, the subgroups $B_{\cN,\tilde d_i}^\star \cap R_\cN^0$ together generate $R_\cN^0 = \ker(\indn)$. 
Thus, we have $\ker(\indn) \subset \im(\mb)$.

Finally, it remains to construct an element $\sigma \in \im(\mb)$ such that $\indn(\sigma) = \vt(A)$. 
By the lemma above, there exists $\sigma_i \in \im(\mb)$ such that $\indn(\sigma_i) = \tilde \vt_i$. 
Thus, there is $\sigma \in \im(\mb)$ such that 
$$
\indn(\sigma) = \gcd(\{\tilde \vt_i\}_{i \in I}) = \vt(A).
$$
This concludes the proof of Theorem \ref{thm:main}. \hfill $\qed$

\subsection{Proof of Theorem \ref{thm:maingalois}}\label{sec:proofmaingalois}
In order to prove Theorem \ref{thm:maingalois}, it suffices to show that the image of $B_A = \rtn \cap \bnd$ under the map $\pin:\bn \rightarrow \symn$ is $\indnp^{-1}\big(\left\langle e^{2i\pi\vt/d} \right\rangle\big)$. To see this, we claim that there is a commutative diagram
\begin{equation}\label{e:comdiag}
    \xymatrixcolsep{3pc}
    \xymatrixrowsep{2pc}
    \xymatrix{
        \bnd \ar[d]_{\indn} \ar@{->>}[r]^{\pin} & \symnd \ar[d]^{\indnp} \\
        \Z \ar@{->>}[r]^{\emph{\rho}}  & U_d
    }
\end{equation}
such that $\pin\big(\ker(\indn)\big) = \ker(\indnp)$. If so, the commutativity of the diagram and the surjectivity of the horizontal arrows imply that $\indnp(G_A) = \rho(\rtn) = \left\langle e^{2i\pi\vt/d} \right\rangle$. To prove that $G_A = \indnp^{-1}\big(\left\langle e^{2i\pi\vt/d} \right\rangle\big)$, it remains to show that $\ker(\indnp) \subset G_A$. But since $\ker(\indn) \subset B_A$ and $G_A = \pin(B_A)$, the inclusion $\ker(\indnp) \subset G_A$ follows from the equality $\pin\big(\ker(\indn)\big) = \ker(\indnp)$.

Let us now verify the existence of the commutative diagram with the desired properties. Clearly, the image of $\bnd$ under $\pin$ is the $U_d$-equivariant subgroup $\symnd \subset \sym_{\cN}$. We identify the target $\pi_1(\cC, x_0)$ of $\indn$ with $\Z$ via the choice of the generator $t \mapsto x_0 \cdot e^{2i\pi t}$, $0 \leqslant t \leqslant 1$. Observe that any braid $\beta \in \bnd$ can be represented by a collection of paths $t \mapsto x(t)$, $0 \leqslant t \leqslant 1$, indexed by $x \in \cN$. From this representative of $\beta$, we construct the element $\gamma: t \mapsto \prod_{x \in \cN} x(t)$ of $\pi_1(\cC, x_0)$. With these conventions, the sought diagram is the following:
\[
\xymatrixcolsep{7pc}
\xymatrixrowsep{3pc}
\xymatrix{
    \bnd \ar[d]_{\beta \mapsto \frac{1}{2i\pi}\int_\gamma\frac{dz}{z}} \ar@{->>}[r]^{\beta \mapsto (x \mapsto x(1))_{x \in \cN}} & \symnd \ar[d]^{\sigma \mapsto \prod_{x \in E} \frac{\sigma(x)}{x}} \\
    \Z \ar@{->>}[r]^{\theta \mapsto e^{2i\pi \theta/d}}  & U_d
}
\]
where the formula $\sigma \mapsto \prod_{x \in E} \frac{\sigma(x)}{x}$ for $\indnp$ does not depend on the choice of $E \subset \cN$. The commutativity of the diagram is now an easy exercise in complex analysis, showing that the composition of the two arrows in the diagram equals the map $\beta \mapsto \exp\left(\frac{1}{2i\pi}\int_{\gamma_E}\frac{dz}{z}\right)$, where $\gamma_E: t \mapsto \prod_{x \in E} x(t)$.

To see that $\pin\big(\ker(\indn)\big) = \ker(\indnp)$, recall that $\symnd$ is isomorphic to the wreath product $U_d \ \textrm{wr}_{N/d}\ \sym_{N/d}$ and that $\indnp$ acts as $(\xi_1,\ldots,\xi_{N/q}, \sigma) \mapsto \prod_{j} \xi_j$. The kernel of $\indnp$ is generated by elements of the form $(1,\ldots,1,\sigma)$ and $(\xi_1,\ldots,\xi_{N/q}, \id)$ such that $\prod_j \xi_j = 1$. Clearly, such elements belong to $\pin\big(\ker(\indn)\big)$.
\hfill $\qed$

\begin{rem}\label{rem:O}
Let us briefly comment on the Zariski-open subset $\mathscr{O}$ appearing in Theorems \ref{thm:main} and \ref{thm:maingalois}. Recall that $\uA$ denotes the projection of $A$ onto the first coordinate axis in $\Z^{k+1}$. Then the branching locus $\sB$ is the pullback of the principal $\uA$-determinant $E_{\uA} \subset \C^{\uA}$ by the coefficient map $\mathscr{C}_\vp : (y_1, \ldots, y_k) \mapsto (c_j)_{j \in \Ax}$ (see \cite[Chapter 10]{GKZ}). Theorems \ref{thm:main} and \ref{thm:maingalois} apply to any polynomial $\vp$ such that the image of $\mathscr{C}_\vp$ intersects $E_{\uA}$ transversally. This is a Zariski-open condition in $\C^A$ (see \cite[Theorem 6.6.2]{ESV}). The actual Zariski-open subset $\mathscr{O}$ of polynomials $\vp$ for which the conclusions of Theorems \ref{thm:main} and \ref{thm:maingalois} hold may in fact be larger. We investigate this question further in \cite[Section 2]{EL3}, for support sets $A$ whose Galois group $G_A$ is the full symmetric group.
\end{rem}

\subsection{Discussions: kernel, presentation and isomonodromy}\label{sec:discussion} In this section, we discuss how the explicit methods of Section \ref{sec:coamoeba} allow us to determine the kernel of $\mb$, obtain a presentation of the braid monodromy group $B_A$, and study isomonodromy loci.

As discussed in the proof of Lemma \ref{lem:braidtrinomial}, we can explicitly determine the image under $\mb$ of the generators $\ell_j$, $0 \leqslant j \leqslant \delta$, of $\pi_1(\C \setminus \sB, y_0)$ when $\vp$ is a trinomial. Recall that $\mb(\ell_0) = \tau^b$ and that $\mb(\ell_j) = b_k$ for $j > 1$ and for some $1 \leqslant k \leqslant n$. The integer $k$ is determined as follows: $\ell_j$ encompasses exactly one point of $\sB$ in $\cC$, this point maps to a singular vertex under $\Arg$, and this singular vertex lies in the interior of one of the D-parallelograms covering $(S^1)^2$. The corresponding parallelogram is indexed by some integer $k$, so that $\mb(\ell_j) = b_k$.

To determine the kernel of $\mb$ as well as a presentation of the corresponding braid monodromy group, it suffices to:

-- know the relations between the elements $\tau^b$, $b_1, \cdots, b_n$ of $\bn$ (see \textit{e.g.} \cite{Lam}),

-- determine the image and cardinality of the fibers of the map $j \mapsto k$.\\
For the latter, observe that the singular vertices are equidistributed along a collection of $\gcd(n,b)$ geodesics parallel to the D-geodesics. These are the preimages under the map
$$S^1 \rightarrow S^1, \quad (\theta,\nu) \mapsto (m\theta + a\nu,\, n\theta + b\nu)$$
of the geodesic $\{\theta = 0\}$; see Section \ref{sec:coamoeba}. We refer to these geodesics as \textit{D'-geodesics}. Each D'-geodesic decomposes into labelled \textit{D'-segments}, which are the intersections of the geodesic with labelled D-parallelograms. Therefore, the image and cardinality of the fibers of $j \mapsto k$ are determined by the distribution, on each D'-geodesic, of the $\delta/\gcd(n,b)$ equidistributed singular vertices relative to the $n/\gcd(n,b)$ equidistributed intersection points with the horizontal section $\{\nu = \arg(y_0)\}$. This can be computed for any specific trinomial $\vp$ and depends solely on the arithmetic of the support $A := \{(0,0), (m,a), (n,b)\}$.\smallskip

Next, let us briefly discuss how the above observations on trinomials extend to arbitrary supports $A \subset \Z^2$ using tropical geometry. Consider a regular triangulation $\mathcal{T}$ of $\conv(A)$ whose vertex set is $A$, that is, there is a piecewise-linear convex function $f : \conv(A) \rightarrow \R$ whose domains of linearity are exactly the triangles in $\mathcal{T}$ (see \textit{e.g.} \cite[Section 2.1]{HPPS} for existence). We further assume that no inner edge of $\mathcal{T}$ is vertical. This can be achieved by taking a smaller subset $\wt A \subset A$ with the same invariants $N$, $d$, and $\vt$, and a triangulation $\mathcal{T}$ supported on $\wt A$ rather than $A$. Then, we fix
\begin{equation}\label{eq:phiViro}
\vp(x,y):= \sum_{a := (a_1,a_2) \in A} s^{f(a)} x^{a_1} y^{a_2}
\end{equation}
for $s > 0$ arbitrarily small. Such a polynomial is known as a Viro polynomial, a central object in tropical geometry (see \cite{BIMS}). The Viro polynomial $\vp$ defines a tropical curve $C \subset \R^2$, where $\R^2$ is the plane with coordinates $(\mathbf{x}, \mathbf{y}) := \Log(x,y) := (\log|x|, \log|y|)$. The tropical curve $C$ is a piecewise-linear graph dual to the subdivision $\mathcal{T}$ of $\conv(A)$ (see \cite[Proposition 2.5]{BIMS}). Any triangle $T \subset \mathcal{T}$ defines a trinomial
$$\vp_T(x,y) := \sum_{a := (a_1, a_2) \in T} s^{f(a)} x^{a_1} y^{a_2},$$
which is dual to a vertex $v_T$ of $C$, see Figure \ref{fig:isomonodromy}.$(a)$ and $(b)$. Viro's Patchworking Theorem states that for any neighbourhood $\cV_T \subset \R^2$ of $v_T$, the set $\Log^{-1}(\cV_T) \cap \{\vp(x,y) = 0\}$ is a small deformation of $\Log^{-1}(\cV_T) \cap \{\vp_T(x,y) = 0\}$ (see \cite[Section 2.3.2]{IMS}). Moreover, it describes how the pieces $\Log^{-1}(\cV_T) \cap \{\vp(x,y) = 0\}$ glue together, for all triangles $T \subset \mathcal{T}$. Using the trinomial case, this allows one to compute the image of the generators of $\pi_1(\C \setminus \sB, y_0)$ under $\mb$, determine the kernel of $\mb$, and obtain a presentation of $B_A$.\smallskip

We now explain how the explicit description of the map $\mb$ allows one to derive information on isomonodromy loci, i.e., subsets of the form
$$\Mon(B) := \{ \vp \in \C^A : \im(\mb) = B \}$$
for a prescribed subgroup $B \subset \bn$. The first instance is the set $\Mon(B_A)$. In Remark \ref{rem:O}, we mentioned that $\Mon(B_A)$ contains the Zariski-open subset of $\C^A$ consisting of polynomials $\vp$ for which the cardinality of the bifurcation set $\sB$ is maximal. We now illustrate that this containment is generally strict.

Consider, for instance, a sharp support set $A \subset \Z^2$ contained in a vertical strip $[0,n] \times \Z$, including $(0,0)$ and $(n,b)$ as vertices and lying in the upper half-plane defined by the line $\R \cdot (n,b)$. Assume further that $(n,b)$ is primitive, and that there exists a point $(m,a) \in A$ with $\gcd(m,n) = 1$. Then the trinomial $\wt A := \{(0,0), (m,a), (n,b)\}$ has the same invariants as $A$, namely $d(\wt A) = d(A) = 1$, $N(\wt A)= N(A) = n$, and $\vt(\wt A) = \vt(A)$. Define
$$D := \Vol(A) - \Vol(\wt A) + 1,$$
where $\Vol$ denotes twice the Euclidean area of the convex hull of its argument. Also, write $A_0 := \{(0,0), (n,b)\}$ and $A^\star := A \setminus A_0$. Then we claim that
\begin{equation}\label{eq:isomonodromy}
\left\{ \wt \vp \in (\cC)^{A_0} \times \C^{A^\star} : \sB \text{ has at least } D \text{ simple points in } \cC \right\} \subset \Mon(B_A).
\end{equation}
Whenever $A$ is sharp, the maximal cardinality of $\sB$ for $\vp \in \C^A$ is $\Vol(A) + 1$. This follows from the Bernstein-Kouchnirenko-Khovanskii Theorem, noting that $0 \in \C$ must lie in $\sB$ for sharp supports. In particular, the left-hand side of \eqref{eq:isomonodromy} strictly contains the set of polynomials $\vp$ such that $|\sB| = \Vol(A) + 1$.
\begin{figure}[t]
\centering
\def\svgwidth{\linewidth} 
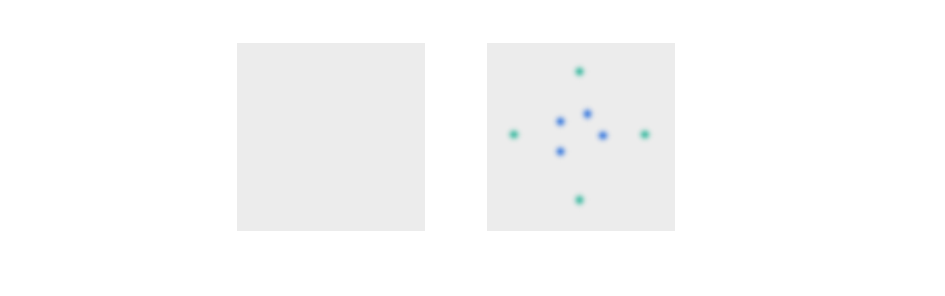
\caption{(a) The triangulated support set $A$ with $\conv(\wt A)$ in blue. (b) The corresponding tropical curve $C \subset \R^2$. Coloured halos illustrate the duality with the triangulation of $A$. (c) The bifurcation set $\sB$ for $\vp$ as in \eqref{eq:phiViro}. We can choose a skeleton $\Gamma$ of $\C \setminus \sB$ as the juxtaposition of two railway tracks and compute $\mb$ explicitly on the generators of $\pi_1(\Gamma)$. Again, colored halos illustrate duality. (d) The deformation of $\sB$ along a path from $\vp$ as in \eqref{eq:phiViro} to $\wt \vp$ as in the left-hand side of \eqref{eq:isomonodromy}. In this example, $D = 5$.}
\label{fig:isomonodromy}
\end{figure}

Let us outline the proof of \eqref{eq:isomonodromy} in the case where $A$ consists of four points and admits a triangulation with exactly two triangles, one of which is $\conv(\wt A)$ (see Figure \ref{fig:isomonodromy}). For $\vp$ as in \eqref{eq:phiViro}, the bifurcation set $\sB$ appears as in Figure \ref{fig:isomonodromy}.$(c)$. The generators of $\pi_1(\Gamma)$ encircling the blue points map to elements $b_k \in \bn$ under $\mb$, while the circle around $0$ maps to $\tau^b$. Now consider a path in $\C^A$ from $\vp$ to $\wt\vp$ as in the left-hand side of \eqref{eq:isomonodromy}. By cardinality, at least one of the blue points in $\sB$ remains simple along the deformation. We deduce that $\im(\mu_{\wt\vp})$ contains $\tau^b$ and some $b_k$. As $n$ and $b$ are coprime, every element $b_j$ arises as a conjugate of $b_k$ by a power of $\tau^b$. Hence, $\im(\mu_{\wt\vp})$ contains $\{b_1, \ldots, b_n\} = \ker(\ind_{\wt\cN})$. Since $\ind_{\wt\cN}(\tau^b) = b = \vt(A)$, we conclude that $\im(\mu_{\wt\vp}) \supset R_{\wt\cN}^b = B_A$. The latter inclusion, which must be an equality, proves the claim in \eqref{eq:isomonodromy}.

Finally, observe that the integer $D$ can be as small as $2$ for suitable choices of $A$, while $\Vol(A)$ can be made arbitrarily large. This illustrates that for certain choices of $A$, the isomonodromy locus $\Mon(B_A)$ contains polynomials $\wt\vp$ that are highly degenerate with respect to the projection $(x,y) \mapsto y$.

\section{Proofs for Problem 2}\label{sec:pb2}

\subsection{One two-dimensional support}\label{sec:12}
As in Section \ref{sec:resultsreducible}, we consider a pair $\emph{A}:=(\emph{A_1}, \emph{A_2})$\index{A@$A$}, $A_j \subset \Z^2$,  where $A_2\subset \{0\}\times\Z$. 
We are interested in the determination of the Galois group $\emph{G_A}$ of the general system of equations supported on $A$
\begin{equation}\label{eq:system2}
p(x,y)=q(y)=0
\end{equation}
where $p\in\CC^{A_1}$ and $q\in\CC^{A_2}$. For now, we restrict ourselves to the case where $A_1$ is not contained in a line. The remaining case is addressed in the Section \ref{sec:reducibletwolines}.

Recall the branching locus $\emph{\sB}\subset \CC^{A}$\index{B@$\sB$}, that is the hypersurface consisting of pairs $(p,q)$ for which the number of solutions to \eqref{eq:system2} in $(\cC)^2$ is not maximal. Loops in the complement of $\sB$ based at a chosen pair $\emph{(p_0,q_0)}$\index{pzero@$(p_0,q_0)$} induce permutations on the set of solutions $\emph{\cN}$\index{Nc@$\cN$} to the system \eqref{eq:system2} corresponding to $(p_0,q_0)$. The latter permutations form the  Galois group $\emph{G_A}\subset \symn$\index{GA@$G_A$} of $A$ that we aim to compute.
 
As explained in Section \ref{sec:resultsreducible}, we will deduce $G_A$ as the image of the braid monodromy group $\emph{B_A}$\index{BA@$B_A$} under the natural projection $\pin:\bn\rightarrow \symn$. In other words, Theorem \ref{thm:reduciblegalois}, which describes $G_A$, will come as a corollary of Theorem \ref{thm:reduciblebraid}, which describes $B_A$.

For convenience, we work under the conditions of Assumption \ref{assumption}. In particular, the polynomial $q(x,y)=q(y)$ is univariate of degree $m$, the polynomial $p(x,y)$ has degree $n$ in $x$ and the system \eqref{eq:system2} has $\emph{N}:=n\cdot m$\index{N@$N$} many solutions for generic $(p,q)$.\smallskip

We decompose the proof of Theorem \ref{thm:reduciblebraid} into the following 3 lemmas, whose proofs are given in subsequent sections.

\begin{lemma}\label{lem:reducibleobstruction}
For $A$ as in Assumption \ref{assumption}, we have $\ind_A(B_A)=B_W$.
\end{lemma}

\begin{lemma}\label{prop:reductionbraidreducible}
Theorem \ref{thm:reduciblebraid} holds if and only if it holds for reduced pairs $A$, that is, for pairs $A$ such that $\left\langle A_1,A_2 \right\rangle=\Z^2$.
\end{lemma}

\begin{lemma}\label{lem:ker}
For $A$ as in Assumption \ref{assumption} and reduced, the group $B_A$ contains $\ker(\ind_A)$.
\end{lemma}

\begin{proof}[Proof of Theorem \ref{thm:reduciblebraid}]
Thanks to Lemma \ref{prop:reductionbraidreducible}, we may assume that $A$ is reduced. By Lemmas \ref{lem:reducibleobstruction} and \ref{lem:ker}, we know that $\ind_A(B_A)=B_W$ and that $\ker(\ind_A)\subset B_A$. It follows that $B_A$ is equal to $\ind_A^{\;\;-1}(B_W)$.
\end{proof}

We postpone the proof of Theorem \ref{thm:reduciblegalois}, which we will derive as a corollary of Theorem \ref{thm:reduciblebraid}, until the end of this section.

\subsubsection{Proof of Lemma \ref{lem:reducibleobstruction}}\label{sec:obstructions}

We aim to prove that $\ind_A(B_A)=B_W$, where $A$ is as in Assumption \ref{assumption}.

Recall from \eqref{e:prod} that the map $\ind_A : \bn \to \bla$ is induced by the multiplication map 
\begin{equation*}
    \begin{array}{rcl}
       \displaystyle \bigcup_{q(y)=0}F_y\times\{y\} & \longmapsto & \displaystyle\bigcup_{q(y)=0} \big(\prod_{x\in F_y}x,y\big)
    \end{array}
\end{equation*}
from $\sU_A\subset\UConf_N(\ttor)$ to $\sU_W(\{1\})\subset\UConf_m(\ttor)$. The composition of the latter multiplication map with $\sC_A\to\sU_A$, $(p,q)\mapsto\{p=q=0\}$ admits a simple description, thanks to Vieta's formula. Indeed, writing 
\begin{equation*}
    p(x,y)\; =\; c_0(y)+c_1(y)x+\cdots+c_n(y)x^n,
\end{equation*}
the product of the roots of $p$ for a chosen $y\in\cC$ is the unique solution to the equation
\begin{equation*}
    w(x,y)=0, \, \text{ where } \emph{w(x,y)}:=(-1)^{n+1}c_0(y)+c_n(y)x.
\end{equation*}\index{w@$w$}
It follows that the composition of the map $\ind_A$ with $\mu_A$ takes values in the braid group $B_W$, corresponding to systems of the form
\begin{equation*}
    w(x,y)=q(y)=0.
\end{equation*}
Moreover, the linear map $\CC^{A}\to\CC^{W}$, which sends the pair $(p,q)$ supported on $A$ to the pair $(w,q)$ supported on $W$,  maps $\sC_A$ to $\sC_{W}$. We claim that this map is surjective and that any loop in $\sC_{W}$ lifts to a loop in $\sC_A$. From this claim, the desired equality $\ind_A(B_A)=B_W$ follows. 

Let us prove the claim. To show surjectivity, observe that any pair $(w,q)\in \sC_{W}$ satisfies $c_n(y)c_0(y)\neq0$ for any root $y\in\cC$ of $q$. For a generic choice of the remaining coefficients $c_a(y)$ of $p(x,y)$, each of the $m$ many polynomials $x\mapsto p(x,y)$, for $q(y)=0$, has $n$ distinct roots in $\cC$. Observe that this remains true even when $\Ax=\{0,n\}$. Thus, the map $\sC_A\to\sC_{W}$ is surjective. Furthermore, the fiber in $\sC_A$ over any point in $\sC_{W}$ is the complement to a hypersurface in an affine subspace of $\C^{A}$. In particular, the fiber is connected. Hence, any loop in $\sC_{W}$ lifts to a loop in $\sC_A$, completing the proof of the claim and the lemma.$\hfill\qed$

\subsubsection{Proof of Lemma \ref{prop:reductionbraidreducible}}\label{sec:reduction}

We aim to prove that Theorem \ref{thm:reduciblebraid} holds if and only if it holds for reduced pairs $A$. Recall that $A$ satisfies Assumption \ref{assumption} in the theorem.

As in Section \ref{sec:resultsreducible}, we denote by $\emph{K}:=K(A)\subset\ttor$ the largest subgroup such that the set $\{p(x,y)=q(y)=0\}$ is invariant under multiplication by $K$, for any $(p,q)\in\CC^{A}$. Pick generators ${a \choose \,b\,}$ and ${0 \choose \,c\,}$ of the lattice $\left\langle A_1,A_2\right\rangle\subset \Z^2$ (any lattice of rank 2 in $\Z^2$ admits such a basis). Then, the kernel of the group homomorphism 
$$\emph{\tau}:\ttor\to\ttor,\; (x,y)\mapsto(x^ay^b,y^c)$$
is exactly $K$. For any pair $(p,q)\in\CC^A$, we can write $p=\tilde p\circ\tau$ and $q=\tilde q\circ\tau$ for some auxiliary Laurent polynomials $\emph{\tilde p}$ and $\emph{\tilde q}$. Denote by $\emph{\wt A}:=(\wt A_1,\wt A_2)$ the support of the pair $(\tilde p,\tilde q)$, for generic $(p,q)\in\CC^A$, and add a tilde to every piece of notation associated to $\wt A$. 

Up to composition of $\tau$ with $(x,y)\mapsto(x^{\pm 1},y^{\pm 1})$, we can assume that $a$ and $c$ are positive, hence that $\wt A$ satisfies Assumption \ref{assumption} for the integers $\emph{\tilde n}:=n(\wt A)$ and $\emph{\tilde m}:=m(\wt A)$. Observe that the support $\wt A$ is reduced. In particular, we have $K(\wt A)=\{1\}$ and $d(\wt A)=1$. 

To prove Lemma \ref{prop:reductionbraidreducible}, we show that Theorem \ref{thm:reduciblebraid} holds for $A$ if and only if it holds for $\wt A$. To prove this, we rely on the existence of appropriate commutative diagrams.

First, it is clear that the map $(\tilde p, \tilde q) \mapsto (\tilde p\circ\tau, \tilde q\circ\tau)$ from $\CC^{\wt A}$ to $\CC^A$ induces an isomorphism $\emph{f}:\sC_{\wt A}\to\sC_A$. Next, the injective map $\UConf_{N/|K|}(\ttor)\to\UConf_{N}(\ttor)$, defined by taking preimages under $\tau$, induces a map $\emph{g}:\sU_{\wt A}\to\sU_A$. Finally, observing that $m=\tilde mc$, the injective map $\UConf_{\tilde m}(\ttor)\to\UConf_{m}(\ttor)$, defined by taking preimages under $(x,y)\mapsto\big((-1)^{n+\tilde n}xy^{\tilde n b},y^c\big)$, induces a map $\emph{h}:\sU_{\wt W}(\{1\})\to\sU_{W}(\{1\})$.

We claim that the maps $f$, $g$, and $h$ are isomorphisms and fit into the following commutative diagram
\begin{equation*}
\xymatrixcolsep{8pc}
\xymatrix{
	\sC_{\wt A}\; \ar[d]_{f} \ar[r]^{(p,q)\mapsto \{p=q=0\}} & \;\sU_{\wt A}\; \ar[d]_{g} \ar[r]	& \;\sU_{\wt W}(\{1\})\; \ar[d]^{h}\\
    \sC_{A}\; \ar[r]^{(\tilde p,\tilde q)\mapsto \{\tilde p=\tilde q=0\}} & \;\sU_{A}\; \ar[r]	& \;\sU_{W}(\{1\})\;
}
\end{equation*}
where the horizontal maps on the right-hand side are the multiplication maps \eqref{e:prod}. Passing to fundamental groups, we obtain the following diagram:
\begin{equation}\label{e:commudiag}
\xymatrixcolsep{3pc}
\xymatrixrowsep{2pc}
\xymatrix{
	\pi_1(\sC_{\wt A})\; \ar[d]_{f_*} \ar[r]^-{\mu_{\wt A}} & \;B_{\wt A}\; \ar[d]_{g_*} \ar[r]^-{\ind_{\wt A}}	& B_{\wt W}\; \ar[d]^{h_*}\\
    \pi_1(\sC_{A})\; \ar[r]^-{\mu_A} & \;B_A\; \ar[r]^-{\ind_{A}}	& \;B_{W}\;
}
\end{equation}
where all vertical maps are isomorphisms. This is because the respective images of $\mu_{\wt A}$ and $\mu_A$ are $B_{\wt A}$ and $B_A$, which are in turn mapped to $B_{\wt W}$ and $B_{W}$, by Lemma \ref{lem:reducibleobstruction}.

Upon these claims and the fact that $\ind_A$ and $\ind_{\wt A}$ are surjective, by Lemma \ref{lem:reducibleobstruction}, we conclude that $B_A=\ind_A^{-1}(B_{W})$ if and only if $B_{\wt A}=\ind_{\wt A}^{-1}(B_{\wt W})$. 

Let us now prove the above claims. The fact that $g$ is an isomorphism is a simple consequence of the definition given in \eqref{e:defUA}. Indeed, it is clear that the preimage of a configuration in $\sC_{\wt A}$ under $\tau$ is in $\sC_A$. Conversely, one easily verifies that the direct image of a configuration in $\sC_A$ lies in $\sC_{\wt A}$. The fact that $h$ is an isomorphism follows from the same principle. There, it may be helpful to decompose the map $(x,y)\mapsto\big((-1)^{n+\tilde n}xy^{\tilde n b},y^c\big)$ as the composition of the bijection $(x,y)\mapsto\big((-1)^{n+\tilde n}xy^{\tilde n b},y\big)$ with the covering $(x,y)\mapsto(x,y^c)$, and observe that $q(y)=\tilde q(y^c)$.

The commutativity of the left square in the first diagram is clear. For the right square, recall that a configuration $\bigcup_{\tilde q(y)=0} \wt F_y\times \{y\}$ in $\sC_{\wt A}$ maps to the configuration $\bigcup_{\tilde q(y)=0}\big(\prod_{x\in\wt F_y}x,y\big)$ in $\sU_{\wt W}(\{1\})$. The image of the former configuration under $g$ is of the form
\[
\bigcup_{\tilde q(y^c)=0} \left\lbrace \xi\sqrt[a]{x/y^b}:\xi^a=1,\; x\in\wt F_{y^c}\right\rbrace \times \{y\}
\]
for any chosen determination of $\sqrt[a]{x/y^b}$. The image under the multiplication map \eqref{e:prod} of this configuration is
\[
\begin{array}{rcl}
\displaystyle\bigcup_{\tilde q(y^c)=0} \left( \prod_{x\in\wt F_{y^c}}\prod_{\xi^a=1}\xi\sqrt[a]{x/y^b}, y\right)    & = &  
\displaystyle\bigcup_{\tilde q(y^c)=0} \left( \prod_{x\in\wt F_{y^c}}(-1)^{a+1}x/y^b, y\right)\\[0.8cm]
& = &  \displaystyle \bigcup_{\tilde q(y^c)=0} \left( \frac{(-1)^{n+\tilde n}}{y^{\tilde n b}}\prod_{x\in\wt F_{y^c}}x, y\right),
\end{array}  
\]
which is the preimage of the point configuration $\bigcup_{\tilde q(y)=0}\big(\prod_{x\in\wt F_y}x,y\big)$ under the map $(x,y)\mapsto\big((-1)^{n+\tilde n}xy^{\tilde n b},y\big)$ . This proves the commutativity of the right square.

The commutativity of the second diagram follows from the commutativity of the first and the fact that $\ind_A$ and $\ind_{\wt A}$ are surjective, by Lemma \ref{lem:reducibleobstruction}. The maps $f_*$, $g_*$, and $h_*$ are isomorphisms, since $f$, $g$, and $h$ are. The result follows.  
\hfill$\qed$\medskip

During the proof, we showed that $g_*: B_{\wt A} \to B_A$ is an isomorphism, without using that $\ker(\tau) = K$. This leads to the following, more general statement.

\begin{lemma}\label{lem:iso}
Let $\wt A$ and $A$ be two finite subsets of $\Z^2$ that satisfy Assumption \ref{assumption}, and suppose that $A = L(\wt A)$ for some injective homomorphism $L:\Z^2 \to \Z^2$. Denote the dual map by $\tau:\ttor \to \ttor$. Then, the map 
\[
\{\tilde p = \tilde q = 0\} \in \sC_{\wt A} \longmapsto \{\tilde p \circ \tau = \tilde q \circ \tau = 0\} \in \sC_A
\]
induces an isomorphism $B_{\wt A} \to B_A$.
$\hfill\qed$
\end{lemma}

\subsubsection{Proof of Lemma \ref{lem:ker}} 
To prove that $B_A$ contains the kernel of $\ind_A$ for reduced pairs $A$, we study the geometry of the branching locus $\sB$ of the covering \eqref{e:projA}. For $\hat y \in \cC$, define 
\begin{equation*}
\emph{\sB_{\hat y}}\index{By@$\sB_{\hat y}$}:=\{p\in\CC^{A_1}: p(x,\hat y) \text{ is singular, with } n \text{ roots in } \cC \text{ counted with multiplicities}\}.
\end{equation*}
The set $\sB_{\hat y}$ is the pullback under the linear map $\emph{L_{\hat{y}}}:\C^{A_1}\rightarrow \C^{\Ax}$, $p(x,y)\mapsto p(x,\hat{y})$, of the $\Ax$-discriminant in $\C^{\Ax}$. Provided that $n>1$, the latter discriminant is irreducible and non-empty (see \cite[Chapter 10]{GKZ}).  Therefore, the set $\sB_{\hat y}$ is an irreducible non-empty hypersurface of $\CC^{A_1}$.

\begin{lemma}\label{lem:that}
For $A_1$ as in Assumption \ref{assumption} and distinct $\hat y,\check y \in \cC$, the intersection $\sB_{\hat y}\cap \sB_{\check y}$ has codimension 1 in both $\sB_{\hat y}$ and $\sB_{\check y}$ unless $\hat y/\check y$ is a $k^{\text{th}}$ root of unity where $k$ is the index of the lattice $\left\langle A_1\right\rangle$ in $\Z^2$.
\end{lemma}

\begin{proof}
Since $\sB_{\hat y}$ and $\sB_{\check y}$ are irreducible, either they coincide or $\sB_{\hat y}\cap \sB_{\check y}$ has codimension 1 in both $\sB_{\hat y}$ and $\sB_{\check y}$. Assume that $\sB_{\hat y}$ 
and $\sB_{\check y}$ coincide. Take $B\subset A_1$ so that the projection $\Z^2\to\Z\times\{0\}$ induces a bijection between $B$ and $\Ax$. In particular, each of the evaluation maps $ev_{\hat y}:(x,y)\to(x,\hat 
y)$ and $ev_{\check y}:(x,y)\to(x,\check y)$ realises an isomorphism from $\C^B$ to $\C^{\Ax}$. Since $\sB_{\hat y}\cap \C^B =\sB_{\check y}\cap \C^B$, it follows that $ev_{\check y}\circ 
(ev_{\hat y})^{-1}:\C^{\Ax}\to\C^{\Ax}$ preserves the $\Ax$-discriminant. In other words, the $\Ax$-discriminant is invariant under translation by $\big((\hat y/\check y)^{b(a)}\big)_{a\in\Ax}$, where we denote by $(a,\emph{b(a)})\in B$ the point above $a\in \Ax$. 

By the Horn-Kapranov uniformisation, the vector $\big((\hat y/\check y)^{b(a)}\big)_{a\in\Ax}$ has to be of the form  $(u v^a)_{a\in\Ax}$ for some $u, v \in \cC$. Otherwise, the reduced $\Ax$-discriminant would still be invariant under translation by $\big((\hat y/\check y)^{b(a)}\big)_{a\in\Ax}$, contradicting the fact that the logarithmic Gauss map is birational (see \cite[Theorem 3.1]{GKZ}).

We deduce from the above argument that $\vert\hat y/\check y\vert^{b(a)}=\vert u v^a\vert$. Setting $\emph{\omega}:=\hat y/\check y$ and taking the logarithm of both sides, we obtain that either $\vert \omega\vert=1$, or 
$$b(a)=\frac{\log\vert v\vert}{\log\vert \omega\vert}\cdot a +\frac{\log\vert u\vert}{\log\vert \omega\vert}$$
for all $a\in \Ax$. 
In particular, the subset $B \subset A$ must lie on a line in $\mathbb{Z}^2$ if $\omega \neq 1$. However, since $A_1$ is not contained in a line, we can choose $B$ not contained in a line and thereby conclude that $|\omega| = 1$.

Now, for any pair $a_1,\, a_2\in \Ax$, and 
\[\emph{D}:=D(a_1,a_2)=\det \begin{pmatrix}a_1-0 & a_2-0\\b(a_1)-b(0) & b(a_2)-b(0)  
\end{pmatrix},\]
a straightforward computation shows that $\omega^D=1$, recalling that $\omega^{b(a)}=uv^a$.
By considering all subsets $B\subset A_1$ whose projection onto $\Ax$ is bijective, we deduce that $\omega^k=1$, where $$\emph{k}:=\gcd(\{D(a_1,a_2)\}_{a_1,\,a_2\in\Ax}).$$ Finally, we observe that $k$ is the index of the lattice $\left\langle A_1\right\rangle$ in $\Z^2$.
\end{proof}

The proof of Lemma \ref{lem:ker} also relies on a stronger form of \cite[Theorem 3]{EL}, which is already established there. Recall that for a support $\uA\subset \Z$ such that $\{0,n\}\subset \uA \subset \{0,\ldots,n\}$, we can consider the braid monodromy map
\[\mu_{\uA}^\star:\pi_1(\sC_{\uA})\rightarrow \bn\]
where $\sC_{\uA}\subset\C^{\uA}$ is the set of polynomials $p(x)=\Sigma_{a\in \uA}c_a x^a$ with $n$ distinct roots in $\cC$, and $\bn$ is denoted $B_n^\star$ in \cite{EL}. We have the following.

\begin{theor}\label{thm:EL}
For any reduced support $\uA\subset \Z$ as above and constants $k_0, \, k_n\in\cC$, the composition of $\pi_1(\sC_{\uA}\cap \{c_0=k_0, \,c_n=k_n\})
\rightarrow\pi_1(\sC_{\uA})$ with $\mu_{\uA}^\star$ maps surjectively onto the subgroup $R_{\sN}^0\subset 
\bn$.
In particular, the image under $\mu_{\uA}^\star$ of the kernel of the map
\[ \pi_1(\sC_{\uA}\setminus\{c_0c_n=0\})\;\longrightarrow\;\pi_1(\C^{\uA}\setminus\{c_0c_n=0\})\]
induced by the inclusion $\sC_{\uA}\setminus\{c_0c_n=0\}\hookrightarrow\C^{\uA}\setminus \{c_0c_n=0\}$ is $R_{\sN}^0$.
\end{theor}

\begin{proof}
Let us first assume that $k_0=k_1=1$. Recall from Lemma \ref{lem:kernel} that $R_{\sN}^0$ is the subgroup of $\bn$ generated by $b_1,\ldots,b_n$. Therefore, the proof of Theorem \ref{thm:EL} is identical to the proof of \cite[Theorem 3]{EL}, provided that we upgrade \cite[Corollary 5]{EL} accordingly. Using the notations in loc. cit., we need to show that the composition of $$\pi_1(\mathcal{C}_A\cap\{c_0=c_2=1\})\to\pi_1(\mathcal{C}_A)$$ with $\mu_A^\star$ maps onto $R_{\sN}^0$. This upgrade is obviously true, by \cite[Lemma 3.5]{EL}. This proves the result for $k_0=k_1=1$.

For arbitrary constants $k_0$ and $k_n$, observe that the image of $\pi_1(\sC_{\uA}\cap \{c_0=c_n=1\})$ under $\mu_{\uA}^\star$ is not affected by the coordinate change $p(x)\mapsto u\cdot p(vx)$. Choosing $u$ and $v$ appropriately, we can achieve any prescribed value $k_0$ and $k_n$.
\end{proof}

We are now ready to prove Lemma \ref{lem:ker}. To this end, consider the affine subspace $\emph{V}\subset\C^{A}$ consisting of all pairs $(p,q)$ such that $q=q_0$ and the coefficients $c_0(y)$ and $c_n(y)$ in the expansion
\begin{equation*}
    p(x,y) = c_0(y)+c_1(y)x+\cdots+c_n(y)x^n
\end{equation*}
are constant and equal to the corresponding coefficients of the fixed polynomial $p_0$. In particular, we have $(p_0,q_0)\in V$. By Vieta's formula, the composition of the natural map $\pi_1(V\setminus\sB)\to\pi_1(\sC_A)$ with $\mu_A$ takes values in $\ker(\ind_A)$. Our goal is to show that the image of this composition is $\ker(\ind_A)$.

To do so, observe that $\ker(\ind_A)$ is a product. Indeed, the map $\bn\to\symq$, which records the permutation induced on the set $\emph{\cQ}$\index{Q@$\cQ$} of roots of $q_0$, factors through $\ind_A$. This implies that $\ker(\ind_A)$ is a subgroup of the kernel of $\bn\to\symq$. Moreover, the restriction of $\ind_A$ to this kernel is the product map $\prod_{y\in\cQ}\ind_{\cN_{y}}$, where $\emph{\cN_y}:=\{x\in\cC:p(x,y)=0\}$ and each map $\emph{\ind_{\cN_y}}$ is defined as in \eqref{eq:indN}. It follows that $\ker(\ind_A)=\prod_{y\in\cQ}\ker(\ind_{\cN_y})=\prod_{y\in\cQ} R^0_{\cN_y}$.

To show that $\pi_1(V\setminus\sB)$ surjects onto $\ker(\ind_A)$, we invoke Zariski's Theorem \cite[Théorème]{Ch}, which states that for a generic line $L\subset V$ passing through $(p_0,q_0)$, the images of $\pi_1(L\setminus\sB)$ and $\pi_1(V\setminus\sB)$ in $\bn$ coincide. We claim that for any pair $\hat y, \, \check y \in \cQ$, the intersection $\sB_{\hat y}\cap \sB_{\check y}$ has codimension 1 in both $\sB_{\hat y}$ and $\sB_{\check y}$. This ensures that the finite sets $L\cap \sB_{\hat y}$, for $\hat y \in \cQ$, are pairwise disjoint. It then follows from Theorem \ref{thm:EL} that the image of $\pi_1(L\setminus\sB)$ contains $\prod_{y\in\cQ}R^0_{\cN_y}$, and hence $\ker(\ind_A)$.

It remains to prove the claim. By Lemma \ref{lem:that}, the claim holds unless $\hat y/\check y$ is a $k^{\text{th}}$ root of unity, where $k$ is the index of the lattice $\left\langle A_1\right\rangle$. Since $q_0\in\C^{A_2}$ is generic, the ratio $\hat y/\check y$ can only be an $l^{\text{th}}$ root of unity, where $l$ is the index of $\left\langle A_2\right\rangle$. Pick generators ${a \choose \,b\,}$ and ${0 \choose \,c\,}$ of $A_1$. 
Since the pair $A$ is reduced and $A_2=\{0\}\times\ell\Z$, we deduce successively that $a=1$, that $c=k$ and eventually that $k$ and $l$ are coprime. This proves the claim and thus the lemma.
$\hfill\square$

\begin{proof}[Proof of Theorem \ref{thm:reduciblegalois}] 
The map $\tau:\ttor \to \ttor$, defined by $(x,y) \mapsto \big(x^d(-1)^{n+\frac{n}{d}}, y\big)$, satisfies
\[
w(\tau(x,y)) = (-1)^{n+1}c_0(y) + c_n(y)x^d(-1)^{n+\frac{n}{d}} = v(x,y)(-1)^{n+\frac{n}{d}}.
\]
Consequently, the preimage of $\{v(x,y) = q(y) = 0\}$ under $\tau$ is $\{w(x,y) = q(y) = 0\}$.
By Lemma \ref{lem:iso}, the map $(v,q) \mapsto (w,q)$ induces an isomorphism $B_V \simeq B_W$.
This isomorphism fits into the following commutative diagram:
\begin{equation*}
    \xymatrixcolsep{3pc}
    \xymatrixrowsep{2pc}
    \xymatrix{
        B_A \ar[d]_{\ind_A} \ar@{->>}[r]^{\pin} & G_A \ar[d]^{\inds} \\
        B_W \simeq B_V \ar@{->>}[r]^{\pi_{\scalebox{0.6}{$\cM$}}} & G_V
    } \quad .
\end{equation*}
Since $\ind_A$ is surjective by Lemma \ref{lem:constraint}, it follows that $\inds$ is also surjective.
Thus, it remains to show that the projection $\pin: B_A \to G_A$ maps 
\[
\ker(\ind_A) = \prod_{y \in \cQ} \ker(\ind_{\cN_y})
\]
surjectively onto 
\[
\ker(\inds) = \prod_{y \in \cQ} \ker(\underline{\ind}_{\cN_y}).
\]
In the proof of Theorem \ref{thm:maingalois}, we showed that $\ker(\ind_{\cN_{\hat{y}}})$ maps surjectively onto $\ker(\underline{\ind}_{\cN_{\hat{y}}})$.
The result follows.
\end{proof}

\subsection{Two one-dimensional supports}\label{sec:reducibletwolines}

We now consider the same problem as in the previous section, in the special case where there exist two skew lines $L_1$ and $L_2 $ such that $A_i\subset L_i$ for $i\in\{1,2\}$. Without loss of generality, we assume that both $A_1$ and $A_2$ contain $(0,0)$ as an endpoint. 

Fix generators $(a,b)$ and $(u,v)$ for the lattices $\left\langle L_1\cap\Z^2\right\rangle$ and $\left\langle L_2\cap\Z^2\right\rangle$ respectively. Then, the covering $\emph{\tau}:\ttor\to\ttor$, $(x,y)\mapsto(x^ay^b,x^uy^v)$ maps the set of solutions to the system
\begin{equation}\label{e:system3}
p(x,y)=q(x,y)=0
\end{equation}
supported on $A$ to the set of solutions to a system
\begin{equation*}
\tilde p(x)=\tilde q(y)=0
\end{equation*}
where $\tilde p(x)$ and $\tilde q(y)$ are univariate polynomials. 
To describe $G_A$, consider the auxiliary systems
\begin{align}
    p(x,y)=0, \quad x^uy^v=1,\label{e:systems1}\\
    x^ay^b=1, \quad q(x,y)=0. \label{e:systems2}
\end{align}
Each solution to \eqref{e:system3} is a product in $\ttor$ of a solution to \eqref{e:systems1} with a solution to \eqref{e:systems2}, and vice versa. Indeed, the respective solution sets are  
\[
\bigcup_{\tilde p(x)=\tilde q(y)=0} \tau^{-1}(x,y)\, , \quad \bigcup_{\tilde p(x)=0} \tau^{-1}(x,1)\; \text{ and } \; \bigcup_{\tilde q(y)=0} \tau^{-1}(1,y)\,.
\]
Thus, the multiplication map $\ttor\times\ttor\to\ttor$ induces a surjection from the covering
\begin{equation}\label{e:covprod}
\bigsqcup_{(p,q)\in\sC_A}
\left\lbrace\begin{array}{l}p(x,y)=0,\\x^uy^v=1\end{array}\right\rbrace
\times
\left\lbrace\begin{array}{l}x^ay^b=1,\\q(x,y)=0\end{array}\right\rbrace
\subset \ttor\times\ttor\times\sC_A \longrightarrow \sC_A
\end{equation}
to the covering
\begin{equation}\label{e:covGA}
\bigsqcup_{(p,q)\in\sC_A}
\left\lbrace\begin{array}{l}p(x,y)=0,\\q(x,y)=0\end{array}\right\rbrace
\subset \ttor\times\sC_A \longrightarrow \sC_A\;.
\end{equation}
Let $\emph{G_1}$ and $\emph{G_2}$ be the respective monodromy groups of the coverings
\[ 
\bigsqcup_{p\in\sC_{A_1}}
\left\lbrace\begin{array}{l}p(x,y)=0,\\x^uy^v=1\end{array}\right\rbrace
\longrightarrow \sC_{A_1}
\quad\text{and}\quad
\bigsqcup_{q\in\sC_{A_2}}
\left\lbrace\begin{array}{l}x^ay^b=1,\\q(x,y)=0\end{array}\right\rbrace
\longrightarrow \sC_{A_2}\;.
\]
The map from \eqref{e:covprod} to \eqref{e:covGA} induces, at the level of monodromy groups, a surjective homomorphism $G_1\times G_2 \to G_A$, which we will use to describe $G_A$. We first describe the homomorphism, its kernel, and eventually the factors $G_1$ and $G_2$.

The map $G_1\times G_2 \to G_A$ admits the following description. Let $\emph{\cN}$, $\emph{\cN_1}$ and $\emph{\cN_2}$ denote the sets of solutions to \eqref{e:system3}, \eqref{e:systems1} and \eqref{e:systems2}, respectively. Each of these sets is acted upon by the group  $\emph{K}:=\ker(\tau)$\index{K@$K$}. We denote by $\emph{\sym_{\scalebox{0.6}{$\cN_1,K$}}}$, $\emph{\sym_{\scalebox{0.6}{$\cN_2,K$}}}$ and $\emph{\sym_{\scalebox{0.6}{$\cN,K$}}}$ the corresponding $K$-equivariant permutation groups. Since each element $(x,y)\in\cN$ can be expressed as a product $(x_1,y_2)\cdot(x_2,y_2)$ with $(x_i,y_i)\in\cN_i$, we have a map
\begin{equation}\label{e:prodmap}
\begin{array}{rcl}
\sym_{\scalebox{0.6}{$\cN_1,K$}}\times\sym_{\scalebox{0.6}{$\cN_2,K$}} & \longrightarrow & \sym_{\scalebox{0.6}{$\cN,K$}}\\[0,2cm]
(\sigma_1,\sigma_2)& \longmapsto & \left( (x,y)\mapsto \sigma_1(x_1,y_1)\cdot\sigma_2(x_2,y_2)  \right)
\end{array}\;.
\end{equation}
The $K$-equivariance of $\sigma_1$ and $\sigma_2$ ensures that the above map is well-defined. 

The monodromy groups $G_A$, $G_1$ and $G_2$ are $K$-equivariant, hence included in $\sym_{\scalebox{0.6}{$\cN_1,K$}}$, $\sym_{\scalebox{0.6}{$\cN_2,K$}}$ and $\sym_{\scalebox{0.6}{$\cN,K$}}$, respectively. The map $G_1\times G_2 \to G_A$ is the restriction of \eqref{e:prodmap}.

Clearly, the kernel of \eqref{e:prodmap} is a subgroup of $\sym_{\scalebox{0.6}{$\cN_1,K$}}\times\sym_{\scalebox{0.6}{$\cN_2,K$}}$ isomorphic to $K$, namely the image of the injection
\[  \xi \longmapsto \left((s_1,s_2)\in\cN_1\times\cN_2\mapsto (\xi\cdot s_1,\xi^{-1}\cdot s_2)\right).\]

\begin{theor}\label{t:reducible2lines}
Let $A:=(A_1,A_2)$ be any pair of finite subsets $A_i\subset \Z^2$ such that there exist skew lines $L_1$ and $L_2$ with $A_i\subset L_i$ for $i\in\{1,2\}$. Then, the product $G_1\times G_2$ contains the kernel of \eqref{e:prodmap} and $G_A$ fits into the short exact sequence
\begin{equation}\label{e:ses}
    1\longrightarrow  K\longrightarrow G_1\times G_2\longrightarrow G_A\longrightarrow 1.
\end{equation}
\end{theor}

\begin{rem} If the covering $\tau$ has degree $1$, then it is a coordinate change that transforms the system $p(x,y)=q(x,y)=0$ into the system $\tilde p(x)=\tilde q(y)=0$. The exact sequence \eqref{e:ses} then reduces to an isomorphism $G_1\times G_2\to G_A$, where $G_1$ and $G_2$ are the Galois groups of the univariate polynomials $\tilde p$ and $\tilde q$, respectively. These groups are determined in \cite{EL}.
\end{rem}

The explicit map \eqref{e:prodmap} and the above exact sequence provides a description of $G_A$, given a description of the groups $G_1$ and $G_2$, which we now provide. 
Define $\emph{d_i}:=\gcd(A_i)$, so that 
\[\tilde p(x)=\check p(x^{d_1}) \;\text{ and }\; \tilde q(y)=\check q(y^{d_2})\]
for some polynomials $\check p$ and $\check q$ with reduced supports. 
The set $\{p(x,y)=0,x^uy^v=1\}$ is the preimage of $\{\check{p}(x)=0,y=1\}$ under the covering 
\[\emph{\tau_1}:(x,y)\mapsto\big((x^ay^b)^{d_1},x^uy^v\big).\] 
In particular, the set $\cN_1$ is acted upon by $\emph{K_1}:=\ker(\tau_1)\simeq K\times U_{d_1}$ and $G_1$ is a subgroup of the $K_1$-equivariant permutations in $\sym_{\scalebox{0.6}{$\cN_1$}}$. The same reasoning applies to $G_2$. Denote by $\emph{\sym_{\scalebox{0.6}{$\cN_i,K_i$}}}\subset \sym_{\scalebox{0.6}{$\cN_i$}}$ the subgroup of $K_i$-equivariant permutations.

\begin{lemma}\label{l:gi}
For $i\in\{1,2\}$, the group $G_i$ equals the subgroup $\sym_{\scalebox{0.6}{$\cN_i,K_i$}}\subset \sym_{\scalebox{0.6}{$\cN_i$}}$.
\end{lemma}

\begin{proof}
We prove the lemma for $i=1$, the remaining case being similar. The subgroup of $K_1$-equivariant permutations in $\sym_{\scalebox{0.6}{$\cN_1$}}$ is isomorphic to the wreath product $K_1 \ \wr_{\scalebox{0.6}{$\check\cN_1$}}\ \sym_{\scalebox{0.6}{$\check\cN_1$}}$, where $\emph{\check\cN_1}$ is the set of roots of $\check p$. By $K_1$-equivariance, we have a surjective map $\sym_{\scalebox{0.6}{$\cN_1,K_1$}}\to\sym_{\scalebox{0.6}{$\check\cN_1$}}$ that records the induced permutation on the set $\check\cN_1$. To prove the statement, it suffices to show that the restriction of this map to $G_1$ is still surjective, and that $G_1$ contains its kernel, which is canonically isomorphic to $K_1^{\;\check\cN_1}$.

From \cite[Theorem 1]{EL}, we know two things. First, the Galois group of $\check p$ is the full symmetric group. This implies that $G_1$ surjects onto $\sym_{\scalebox{0.6}{$\check\cN_1$}}$. Second, for any tuple $\emph{t}$ in $\pi_1(\cC)^{\check{\cN}_1}$, we can find a loop $\emph{\ell}$ in the space of non-singular polynomials $\check p$ such that each root in $\check{\cN}_1$ traces out a loop whose class in  $\pi_1(\cC)\simeq \Z$ is prescribed by the corresponding component of the latter tuple. In particular, the image of $\ell$ in $G_1$ belongs to the kernel of $\sym_{\scalebox{0.6}{$\cN_1,K_1$}}\to\sym_{\scalebox{0.6}{$\check\cN_1$}}$. The image of $\ell$ in $K_1^{\;\check\cN_1}$ is the image of the tuple $t$ under the coordinate-wise monodromy map of the covering $\tau_1$ restricted to $\tau_1^{\;-1}(\cC\times\{1\})$. Since the monodromy group of the latter covering is $K_1$ and since $t\in\pi_1(\cC)^{\check{\cN}_1}$ can be chosen arbitrarily, we conclude that $G_1$ contains the kernel $K_1^{\;\check\cN_1}$ of $\sym_{\scalebox{0.6}{$\cN_1,K_1$}}\to\sym_{\scalebox{0.6}{$\check\cN_1$}}$.
\end{proof}

\begin{proof}[Proof of Theorem \ref{t:reducible2lines}]
A consequence of the lemma is that, for any $\xi\in K$, the group $G_i$ contains the permutation $s_i\mapsto\xi\cdot s_i$. It follows that the product $G_1\times G_2$ contains the kernel 
of \eqref{e:prodmap}, proving in turn the existence of the short exact sequence \eqref{e:ses}.
\end{proof} 

In certain situations, it is possible to find an explicit subgroup of $G_1\times G_2$ that maps bijectively to $G_A$. To see this, define the map 
\begin{equation*}
\begin{array}{rcl}
    \emph{\underline{\ind}_i}\; :\; \sym_{\scalebox{0.6}{$\cN_i,K_i$}} &\longrightarrow & K  \\
     \sigma & \longmapsto &  \prod_{s\in E}\frac{\sigma(s)}{s}
\end{array}
\end{equation*}
where $E\subset \cN_i$ is any subset that intersects each $K$-orbit in $\cN_i$ exactly once. The $K$-equivariance of $\sigma$ implies that $\underline{\ind}_i(\sigma)$ does not depend on the choice of $E$. To see that $\underline{\ind}_i(\sigma)$ actually takes values in $K$, observe that, in the case $i=1$, 
\[ \tau\left(\prod_{s\in E} s\right) = \prod_{s\in E} \tau(s) = \prod_{\tilde p(x)=0} (x,1) = \prod_{s\in \sigma(E)} \tau(s) = \tau\left(\prod_{s\in E} \sigma(s)\right).\]

\begin{utver}\label{p:GA}
The Galois group $G_A$ is isomorphic to $\ker(\underline{\ind}_1)\times G_2$ if $|K|$ is coprime to $\deg(\tilde p)$, and to  $G_1\times\ker( \underline{\ind}_2)$ if $|K|$ is coprime to $\deg(\tilde q)$.
\end{utver}

\begin{proof}
We prove the statement for $i=1$, the remaining case being similar. In view of the exact sequence \eqref{e:ses}, it suffices to show that the subgroup $\ker(\underline{\ind}_1)$ has index $\vert K\vert$ in $G_i$ and that $\ker(\underline{\ind}_1)\times G_2$ intersects the kernel of \eqref{e:prodmap} trivially.

The group $G_1$ is isomorphic to $K\ \wr_{\scalebox{0.6}{$\tilde\cN_1$}}\ \sym_{\scalebox{0.6}{$\tilde\cN_1, d_1$}}$, where $\tilde\cN_1$ is the set $\{\tilde p(x)=0\}$ and $\sym_{\scalebox{0.6}{$\tilde\cN_1, d_1$}}\subset \sym_{\scalebox{0.6}{$\tilde\cN_1$}}$ is the $U_{d_1}$-equivariant subgroup (recall that $\tilde p(x)=\check p(x^{d_1})$).
Via this isomorphism, the map $\underline{\ind}_1$ reads as
\[\big((\xi_s)_{s\in\tilde\cN_1}, \sigma\big)\longmapsto \prod_{s\in\tilde\cN_1}\xi_s.\]
In particular, it is surjective. Therefore, the subgroup $\ker(\underline{\ind}_1)$ has index $\vert K\vert$ in $G_1$.

The projection of the kernel of \eqref{e:prodmap} onto the first factor of $G_1\times G_2$ consists of elements of the form $\big((\xi)_{s\in\tilde\cN_1},\id_{\scalebox{0.6}{$\tilde\cN_1$}}\big) \in K \ \wr_{\scalebox{0.6}{$\tilde\cN_1$}} \ \sym_{\scalebox{0.6}{$\tilde\cN_1, d_1$}}$, for any $\xi\in K$. The image of such an element under $\underline{\ind}_1$ is $\xi^{|\tilde\cN_1|}=\xi^{\deg(\tilde p)}$. We conclude that that $\ker(\underline{\ind}_1)\times G_2$ intersects the kernel of \eqref{e:prodmap} trivially if $|K|$ and $\deg(\tilde p)$ are coprime.
\end{proof}

\subsection{The groups $B_W$ and $G_V$}

The description of the braid monodromy group $B_A$ in Theorem \ref{thm:reduciblebraid} depends on the braid monodromy group $B_W$ associated with the auxiliary system $w(x,y) = q(y) = 0$, supported on $W$. Similarly, the description of the Galois group $G_A$ in Theorem \ref{thm:reduciblegalois} depends on the Galois group $G_V$ associated with the auxiliary system $v(x,y) = q(y) = 0$, supported on $V$. In order to provide a complete description of $B_A$ and $G_A$, we now describe $B_W$ and $G_V$.

\begin{theor}\label{thm:bwgv}
Let $A := (A_1, A_2)$ be a pair such that $A_2$ lies on a line and $A_1$ does not. 

If $A_1$ is not sharp, then the braid group $B_W$ equals $\bl$ and the Galois group $G_V$ equals $\symmv$.

If $A_1$ is sharp, then $B_W$ is a subgroup of $\bl$ isomorphic to $\pi_1(\cC) \times B_{A_2}$, and $G_V$ is described in Theorem \ref{t:reducible2lines}.
\end{theor}

\begin{proof}
Assume that $A_1$ is not sharp. 

We begin with the computation of the group $B_W$. The projection $\ttor \to \cC$ onto the second factor induces a map $\sU_W \to \UConf_m(\cC)$, whose induced map on fundamental groups is a surjective morphism $\bl \to B_{A_2}$ onto the braid monodromy group of the univariate polynomial $q$ supported on $A_2$ (see \cite{EL}). The restriction of this map to $B_W$ is induced by the projection of the covering
\[
\begin{array}{rcl}
\bigsqcup_{(w,q)\in\sC_W} \{ w(x,y) = q(y) = 0 \} &\longrightarrow& \sC_W \\[0.2cm]
(x,y,w,q) &\longmapsto& (w,q)
\end{array}
\]
onto the covering
\[
\begin{array}{rcl}
\bigsqcup_{q\in\sC_{A_2}} \{ q(y) = 0 \} &\longrightarrow& \sC_{A_2} \\[0.2cm]
(y,q) &\longmapsto& q
\end{array}.
\]
In particular, the restriction of $\bl \to B_{A_2}$ to $B_W$ remains surjective. To prove the equality $B_W = \bl$, it suffices to show that $B_W$ contains the kernel of the map $\bl \to B_{A_2}$.

Let us describe this kernel. First, recall that the group $K(W)$ — the largest subgroup of $\ttor$ that leaves $\{ w(x,y) = q(y) = 0 \}$ invariant under multiplication — is of the form $\{1\} \times U_e$ for some positive integer $e$. This integer $e$ is the largest such that
\[
c_0(y) = \tilde c_0(y^e), \quad c_n(y) = \tilde c_n(y^e)\cdot y^\vr, \quad \text{and} \quad q(y) = \tilde q(y^e)
\]
for some Laurent polynomials $\tilde c_0$, $\tilde c_n$, $\tilde q$, and some integer $\vr$. In particular, the group $U_e$ acts on the set $\cQ$ of roots of $q_0$. The kernel of $\bl \to B_{A_2}$ is thus canonically isomorphic to the subgroup of $\pi_1(\ttor)^\cQ$ that is invariant under the action of $U_e$ by permutation of the factors.

To show that $B_W$ contains this kernel, we consider loops $\ell: t \mapsto (w_t, q_t)$ such that $q_t$ is constant and equal to $q_0$. This ensures that the image of $\ell$ in $B_W$ lies in the kernel of $\bl \to B_{A_2}$.

Let $\{ y_1, \ldots, y_m \}$ denote the set of roots of $q_0$. Denote by $C_0 \subset \Z$ the support of $c_0(y)$. Provided that $c_0(y)$ is not a monomial, each root $y_j$ defines a hyperplane $\{ \tilde c_0(y_j) = 0 \}$ in the space $\CC^{C_0}$ of coefficients of $c_0$. Let $g_0 := \gcd(A_2, C_0)$. In particular, we have $e \mid g_0$, and hence $U_e \subset U_{g_0}$ acts on $\pi_1(\ttor)^\cQ$.

Consider the map $\{ y_1, \ldots, y_m \} \to P(\C^{W_1})$ that assigns to each root $y_j$ of $q_0$ the hypersurface $\{ c_0(y_j) = 0 \}$ in the space of polynomials $w(x,y)$. Since $q_0$ is chosen generically, it is an elementary exercise to check that this map is $g_0$-to-1. Denote by $H_1, \ldots, H_k$ the collection of hyperplanes in the image of this map. For each $H_j$, we can construct a loop $\ell_j: t \mapsto (w_t, q_t)$ such that the winding number of $t \mapsto w_t$ around $H_i$ equals the Kronecker delta $\delta_{i,j}$. The image of the loops $\ell_j$ in $B_W$ generates the subgroup of $\pi_1(\ttor)^\cQ$ invariant under the action of $U_{g_0}$.

Repeating the same argument with the support $C_n$ of $c_n$, and letting $g_n := \gcd(A_2, C_n)$, we conclude that $B_W$ contains the $U_{g_n}$-invariant subgroup of $\pi_1(\ttor)^\cQ$. By construction, we have $\gcd(g_0, g_n) = e$. Consequently, the $U_{g_0}$- and $U_{g_n}$-invariant subgroups together generate the $U_e$-invariant subgroup of $\pi_1(\ttor)^\cQ$, which is precisely the kernel of the map $\bl \to B_{A_2}$. This concludes the proof that $B_W = \bl$.\smallskip

By Lemma \ref{lem:iso}, we obtain that $B_W = \bl$ implies $B_V = \bm$. Since the group $G_V$ is the image of $B_V$ under the surjective map $\pim: \bm \to \symmv$, we deduce that $G_V = \symmv$. \smallskip

Now assume that $A_1$ is sharp. In this case, we can write
\[
c_0(y) = c_0, \quad c_n(y) = c_n \cdot y^{me + \vt}, \quad \text{and} \quad q(y) = \tilde q(y^e),
\]
where $0 \leqslant  \vt < c$. The map $(x,y) \mapsto (x y^\vt, y^e)$ induces an isomorphism between $B_W$ and the braid monodromy group of the reducible and reduced system
\[
c_n x - (-1)^n c_0 = 0, \quad \tilde q(y) = 0.
\]
Clearly, the latter group is isomorphic to $\pi_1(\cC) \times B_{\tilde A_2}$.
\end{proof}

\bigskip
\textit{Acknowledgements.}
The authors are grateful to Rikard B\o{}gvad and Rolf K\"{a}llstr\"{o}m for stimulating discussions on isomonodromic deformations. We are also very grateful to the referee for a thorough reading and for valuable comments on earlier versions of the manuscript.
\medskip

\textit{Funding.}
The second author was partially supported by the Swedish Research Council through Grant No.~\href{https://www.vr.se/english/swecris.html?project=2023-04332_VR#/}{2023-04332}.

\bibliographystyle{plain}
\bibliography{Draft}

\vspace{2cm}
\noindent
A. Esterov\\
London Institute for Mathematical Sciences, UK.\\
\textit{Email}: aes@lims.ac.uk, alexander.esterov@gmail.com \\

\noindent
L. Lang\\
Department of Electrical Engineering, Mathematics and Science, University of G\"avle, Sweden.\\
\textit{Email}: lionel.lang@hig.se

\printindex

\end{document}

%% file: braid.pdf_tex
\begingroup%
  \makeatletter%
  \providecommand\color[2][]{%
    \errmessage{(Inkscape) Color is used for the text in Inkscape, but the package 'color.sty' is not loaded}%
    \renewcommand\color[2][]{}%
  }%
  \providecommand\transparent[1]{%
    \errmessage{(Inkscape) Transparency is used (non-zero) for the text in Inkscape, but the package 'transparent.sty' is not loaded}%
    \renewcommand\transparent[1]{}%
  }%
  \providecommand\rotatebox[2]{#2}%
  \newcommand*\fsize{\dimexpr\f@size pt\relax}%
  \newcommand*\lineheight[1]{\fontsize{\fsize}{#1\fsize}\selectfont}%
  \ifx\svgwidth\undefined%
    \setlength{\unitlength}{453.54330709bp}%
    \ifx\svgscale\undefined%
      \relax%
    \else%
      \setlength{\unitlength}{\unitlength * \real{\svgscale}}%
    \fi%
  \else%
    \setlength{\unitlength}{\svgwidth}%
  \fi%
  \global\let\svgwidth\undefined%
  \global\let\svgscale\undefined%
  \makeatother%
  \begin{picture}(1,0.1875)%
    \lineheight{1}%
    \setlength\tabcolsep{0pt}%
    \put(0,0){\includegraphics[width=\unitlength,page=1]{braid.pdf}}%
    \put(0.1855429,0.0724854){\color[rgb]{0,0,0}\makebox(0,0)[lt]{\lineheight{1.25}\smash{\begin{tabular}[t]{l}$\cdots$\end{tabular}}}}%
    \put(0.51627209,0.07248539){\color[rgb]{0,0,0}\makebox(0,0)[lt]{\lineheight{1.25}\smash{\begin{tabular}[t]{l}$\cdots$\end{tabular}}}}%
    \put(0.86353789,0.07248539){\color[rgb]{0,0,0}\makebox(0,0)[lt]{\lineheight{1.25}\smash{\begin{tabular}[t]{l}$\cdots$\end{tabular}}}}%
    \put(0.15108654,0.13804671){\color[rgb]{0,0.33333333,0.85490196}\makebox(0,0)[lt]{\lineheight{1.25}\smash{\begin{tabular}[t]{l}$b_1$\end{tabular}}}}%
    \put(0.48181571,0.13804671){\color[rgb]{0,0.33333333,0.85490196}\makebox(0,0)[lt]{\lineheight{1.25}\smash{\begin{tabular}[t]{l}$b_N$\end{tabular}}}}%
    \put(0.81254498,0.13804671){\color[rgb]{0,0.33333333,0.85490196}\makebox(0,0)[lt]{\lineheight{1.25}\smash{\begin{tabular}[t]{l}$\tau$\end{tabular}}}}%
  \end{picture}%
\endgroup%

%% file: coamoeba.pdf_tex
\begingroup%
  \makeatletter%
  \providecommand\color[2][]{%
    \errmessage{(Inkscape) Color is used for the text in Inkscape, but the package 'color.sty' is not loaded}%
    \renewcommand\color[2][]{}%
  }%
  \providecommand\transparent[1]{%
    \errmessage{(Inkscape) Transparency is used (non-zero) for the text in Inkscape, but the package 'transparent.sty' is not loaded}%
    \renewcommand\transparent[1]{}%
  }%
  \providecommand\rotatebox[2]{#2}%
  \newcommand*\fsize{\dimexpr\f@size pt\relax}%
  \newcommand*\lineheight[1]{\fontsize{\fsize}{#1\fsize}\selectfont}%
  \ifx\svgwidth\undefined%
    \setlength{\unitlength}{453.54330709bp}%
    \ifx\svgscale\undefined%
      \relax%
    \else%
      \setlength{\unitlength}{\unitlength * \real{\svgscale}}%
    \fi%
  \else%
    \setlength{\unitlength}{\svgwidth}%
  \fi%
  \global\let\svgwidth\undefined%
  \global\let\svgscale\undefined%
  \makeatother%
  \begin{picture}(1,0.21875)%
    \lineheight{1}%
    \setlength\tabcolsep{0pt}%
    \put(0,0){\includegraphics[width=\unitlength,page=1]{coamoeba.pdf}}%
    \put(-1.46810316,0.1037354){\color[rgb]{0,0,0}\makebox(0,0)[lt]{\lineheight{1.25}\smash{\begin{tabular}[t]{l}$\cdots$\end{tabular}}}}%
    \put(-1.13737396,0.10373539){\color[rgb]{0,0,0}\makebox(0,0)[lt]{\lineheight{1.25}\smash{\begin{tabular}[t]{l}$\cdots$\end{tabular}}}}%
    \put(-0.79010816,0.10373539){\color[rgb]{0,0,0}\makebox(0,0)[lt]{\lineheight{1.25}\smash{\begin{tabular}[t]{l}$\cdots$\end{tabular}}}}%
    \put(-1.50255879,0.16929671){\color[rgb]{0,0.33333333,0.85490196}\makebox(0,0)[lt]{\lineheight{1.25}\smash{\begin{tabular}[t]{l}$b_1$\end{tabular}}}}%
    \put(-1.17182993,0.16929671){\color[rgb]{0,0.33333333,0.85490196}\makebox(0,0)[lt]{\lineheight{1.25}\smash{\begin{tabular}[t]{l}$b_N$\end{tabular}}}}%
    \put(-0.84110107,0.16929671){\color[rgb]{0,0.33333333,0.85490196}\makebox(0,0)[lt]{\lineheight{1.25}\smash{\begin{tabular}[t]{l}$\tau$\end{tabular}}}}%
    \put(0,0){\includegraphics[width=\unitlength,page=2]{coamoeba.pdf}}%
    \put(0.52383686,0.11953128){\color[rgb]{0.30196078,0.30196078,0.30196078}\makebox(0,0)[lt]{\lineheight{1.25}\smash{\begin{tabular}[t]{l}$(0,0)$\end{tabular}}}}%
    \put(0.61569144,0.18567707){\color[rgb]{0.6,0.6,0.6}\makebox(0,0)[lt]{\lineheight{1.25}\smash{\begin{tabular}[t]{l}$S^1\times S^1$\end{tabular}}}}%
    \put(0.36393804,0.16914063){\color[rgb]{0,0.33333333,0.83137255}\makebox(0,0)[lt]{\lineheight{1.25}\smash{\begin{tabular}[t]{l}$\Arg_\ell$\end{tabular}}}}%
  \end{picture}%
\endgroup%

%% file: coamoeba2.pdf_tex
\begingroup%
  \makeatletter%
  \providecommand\color[2][]{%
    \errmessage{(Inkscape) Color is used for the text in Inkscape, but the package 'color.sty' is not loaded}%
    \renewcommand\color[2][]{}%
  }%
  \providecommand\transparent[1]{%
    \errmessage{(Inkscape) Transparency is used (non-zero) for the text in Inkscape, but the package 'transparent.sty' is not loaded}%
    \renewcommand\transparent[1]{}%
  }%
  \providecommand\rotatebox[2]{#2}%
  \newcommand*\fsize{\dimexpr\f@size pt\relax}%
  \newcommand*\lineheight[1]{\fontsize{\fsize}{#1\fsize}\selectfont}%
  \ifx\svgwidth\undefined%
    \setlength{\unitlength}{453.54330709bp}%
    \ifx\svgscale\undefined%
      \relax%
    \else%
      \setlength{\unitlength}{\unitlength * \real{\svgscale}}%
    \fi%
  \else%
    \setlength{\unitlength}{\svgwidth}%
  \fi%
  \global\let\svgwidth\undefined%
  \global\let\svgscale\undefined%
  \makeatother%
  \begin{picture}(1,0.375)%
    \lineheight{1}%
    \setlength\tabcolsep{0pt}%
    \put(0,0){\includegraphics[width=\unitlength,page=1]{coamoeba2.pdf}}%
    \put(-1.46810316,0.2599854){\color[rgb]{0,0,0}\makebox(0,0)[lt]{\lineheight{1.25}\smash{\begin{tabular}[t]{l}$\cdots$\end{tabular}}}}%
    \put(-1.13737396,0.25998539){\color[rgb]{0,0,0}\makebox(0,0)[lt]{\lineheight{1.25}\smash{\begin{tabular}[t]{l}$\cdots$\end{tabular}}}}%
    \put(-0.79010816,0.25998539){\color[rgb]{0,0,0}\makebox(0,0)[lt]{\lineheight{1.25}\smash{\begin{tabular}[t]{l}$\cdots$\end{tabular}}}}%
    \put(-1.50255879,0.32554671){\color[rgb]{0,0.33333333,0.85490196}\makebox(0,0)[lt]{\lineheight{1.25}\smash{\begin{tabular}[t]{l}$b_1$\end{tabular}}}}%
    \put(-1.17182993,0.32554671){\color[rgb]{0,0.33333333,0.85490196}\makebox(0,0)[lt]{\lineheight{1.25}\smash{\begin{tabular}[t]{l}$b_N$\end{tabular}}}}%
    \put(-0.84110107,0.32554671){\color[rgb]{0,0.33333333,0.85490196}\makebox(0,0)[lt]{\lineheight{1.25}\smash{\begin{tabular}[t]{l}$\tau$\end{tabular}}}}%
    \put(0,0){\includegraphics[width=\unitlength,page=2]{coamoeba2.pdf}}%
    \put(1.91290044,0.27578128){\color[rgb]{0.30196078,0.30196078,0.30196078}\makebox(0,0)[lt]{\lineheight{1.25}\smash{\begin{tabular}[t]{l}$(1,1)$\end{tabular}}}}%
    \put(2.00475305,0.34192707){\color[rgb]{0.6,0.6,0.6}\makebox(0,0)[lt]{\lineheight{1.25}\smash{\begin{tabular}[t]{l}$S^1\times S^1$\end{tabular}}}}%
    \put(1.75300008,0.32539063){\color[rgb]{0,0.33333333,0.83137255}\makebox(0,0)[lt]{\lineheight{1.25}\smash{\begin{tabular}[t]{l}$\Arg_\ell$\end{tabular}}}}%
    \put(0,0){\includegraphics[width=\unitlength,page=3]{coamoeba2.pdf}}%
    \put(0.16882333,0.31216145){\color[rgb]{0.30196078,0.30196078,0.30196078}\makebox(0,0)[lt]{\lineheight{1.25}\smash{\begin{tabular}[t]{l}$\Z^2$\end{tabular}}}}%
    \put(0.09461531,0.31216146){\color[rgb]{0,0.33333333,0.85490196}\makebox(0,0)[lt]{\lineheight{1.25}\smash{\begin{tabular}[t]{l}$A$\end{tabular}}}}%
    \put(0.46275593,0.38161458){\color[rgb]{0,0.33333333,0.85490196}\makebox(0,0)[lt]{\lineheight{1.25}\smash{\begin{tabular}[t]{l}$\Arg_\vp$\end{tabular}}}}%
    \put(0.56277484,0.38161458){\color[rgb]{0.6,0.6,0.6}\makebox(0,0)[lt]{\lineheight{1.25}\smash{\begin{tabular}[t]{l}$S^1\times S^1$\end{tabular}}}}%
    \put(0.13258915,0.31216145){\color[rgb]{0,0,0}\makebox(0,0)[lt]{\lineheight{1.25}\smash{\begin{tabular}[t]{l}$\subset$\end{tabular}}}}%
    \put(0.52946379,0.38161458){\color[rgb]{0,0,0}\makebox(0,0)[lt]{\lineheight{1.25}\smash{\begin{tabular}[t]{l}$\subset$\end{tabular}}}}%
    \put(0,0){\includegraphics[width=\unitlength,page=4]{coamoeba2.pdf}}%
    \put(0.86816416,0.30885417){\color[rgb]{0,0.33333333,0.85490196}\makebox(0,0)[lt]{\lineheight{1.25}\smash{\begin{tabular}[t]{l}$D$-geodesic\end{tabular}}}}%
    \put(0.86816416,0.19309899){\color[rgb]{0,0.33333333,0.85490196}\makebox(0,0)[lt]{\lineheight{1.25}\smash{\begin{tabular}[t]{l}$L$-geodesic\end{tabular}}}}%
    \put(0.86816416,0.0608073){\color[rgb]{0,0.33333333,0.85490196}\makebox(0,0)[lt]{\lineheight{1.25}\smash{\begin{tabular}[t]{l}$R$-geodesic\end{tabular}}}}%
  \end{picture}%
\endgroup%

%% file: gamma.pdf_tex
\begingroup%
  \makeatletter%
  \providecommand\color[2][]{%
    \errmessage{(Inkscape) Color is used for the text in Inkscape, but the package 'color.sty' is not loaded}%
    \renewcommand\color[2][]{}%
  }%
  \providecommand\transparent[1]{%
    \errmessage{(Inkscape) Transparency is used (non-zero) for the text in Inkscape, but the package 'transparent.sty' is not loaded}%
    \renewcommand\transparent[1]{}%
  }%
  \providecommand\rotatebox[2]{#2}%
  \newcommand*\fsize{\dimexpr\f@size pt\relax}%
  \newcommand*\lineheight[1]{\fontsize{\fsize}{#1\fsize}\selectfont}%
  \ifx\svgwidth\undefined%
    \setlength{\unitlength}{453.54330709bp}%
    \ifx\svgscale\undefined%
      \relax%
    \else%
      \setlength{\unitlength}{\unitlength * \real{\svgscale}}%
    \fi%
  \else%
    \setlength{\unitlength}{\svgwidth}%
  \fi%
  \global\let\svgwidth\undefined%
  \global\let\svgscale\undefined%
  \makeatother%
  \begin{picture}(1,0.4375)%
    \lineheight{1}%
    \setlength\tabcolsep{0pt}%
    \put(0,0){\includegraphics[width=\unitlength,page=1]{gamma.pdf}}%
    \put(-1.46810316,0.3224854){\color[rgb]{0,0,0}\makebox(0,0)[lt]{\lineheight{1.25}\smash{\begin{tabular}[t]{l}$\cdots$\end{tabular}}}}%
    \put(-1.13737396,0.32248539){\color[rgb]{0,0,0}\makebox(0,0)[lt]{\lineheight{1.25}\smash{\begin{tabular}[t]{l}$\cdots$\end{tabular}}}}%
    \put(-0.79010816,0.32248539){\color[rgb]{0,0,0}\makebox(0,0)[lt]{\lineheight{1.25}\smash{\begin{tabular}[t]{l}$\cdots$\end{tabular}}}}%
    \put(-1.50255879,0.38804671){\color[rgb]{0,0.33333333,0.85490196}\makebox(0,0)[lt]{\lineheight{1.25}\smash{\begin{tabular}[t]{l}$b_1$\end{tabular}}}}%
    \put(-1.17182993,0.38804671){\color[rgb]{0,0.33333333,0.85490196}\makebox(0,0)[lt]{\lineheight{1.25}\smash{\begin{tabular}[t]{l}$b_N$\end{tabular}}}}%
    \put(-0.84110107,0.38804671){\color[rgb]{0,0.33333333,0.85490196}\makebox(0,0)[lt]{\lineheight{1.25}\smash{\begin{tabular}[t]{l}$\tau$\end{tabular}}}}%
    \put(0,0){\includegraphics[width=\unitlength,page=2]{gamma.pdf}}%
    \put(1.91290044,0.33828128){\color[rgb]{0.30196078,0.30196078,0.30196078}\makebox(0,0)[lt]{\lineheight{1.25}\smash{\begin{tabular}[t]{l}$(1,1)$\end{tabular}}}}%
    \put(2.00475305,0.40442707){\color[rgb]{0.6,0.6,0.6}\makebox(0,0)[lt]{\lineheight{1.25}\smash{\begin{tabular}[t]{l}$S^1\times S^1$\end{tabular}}}}%
    \put(1.75300008,0.38789063){\color[rgb]{0,0.33333333,0.83137255}\makebox(0,0)[lt]{\lineheight{1.25}\smash{\begin{tabular}[t]{l}$\Arg_\ell$\end{tabular}}}}%
    \put(0,0){\includegraphics[width=\unitlength,page=3]{gamma.pdf}}%
    \put(3.29382562,0.37466145){\color[rgb]{0.30196078,0.30196078,0.30196078}\makebox(0,0)[lt]{\lineheight{1.25}\smash{\begin{tabular}[t]{l}$\Z^2$\end{tabular}}}}%
    \put(3.2196176,0.37466146){\color[rgb]{0,0.33333333,0.85490196}\makebox(0,0)[lt]{\lineheight{1.25}\smash{\begin{tabular}[t]{l}$A$\end{tabular}}}}%
    \put(3.58775834,0.44411458){\color[rgb]{0,0.33333333,0.85490196}\makebox(0,0)[lt]{\lineheight{1.25}\smash{\begin{tabular}[t]{l}$\Arg_\vp$\end{tabular}}}}%
    \put(3.68777724,0.44411458){\color[rgb]{0.6,0.6,0.6}\makebox(0,0)[lt]{\lineheight{1.25}\smash{\begin{tabular}[t]{l}$S^1\times S^1$\end{tabular}}}}%
    \put(3.25759152,0.37466145){\color[rgb]{0,0,0}\makebox(0,0)[lt]{\lineheight{1.25}\smash{\begin{tabular}[t]{l}$\subset$\end{tabular}}}}%
    \put(3.65446604,0.44411458){\color[rgb]{0,0,0}\makebox(0,0)[lt]{\lineheight{1.25}\smash{\begin{tabular}[t]{l}$\subset$\end{tabular}}}}%
    \put(0,0){\includegraphics[width=\unitlength,page=4]{gamma.pdf}}%
    \put(3.99316635,0.37135417){\color[rgb]{0,0.33333333,0.85490196}\makebox(0,0)[lt]{\lineheight{1.25}\smash{\begin{tabular}[t]{l}$A$-geodesic\end{tabular}}}}%
    \put(3.99316635,0.25559899){\color[rgb]{0,0.33333333,0.85490196}\makebox(0,0)[lt]{\lineheight{1.25}\smash{\begin{tabular}[t]{l}$B$-geodesic\end{tabular}}}}%
    \put(3.99316635,0.1233073){\color[rgb]{0,0.33333333,0.85490196}\makebox(0,0)[lt]{\lineheight{1.25}\smash{\begin{tabular}[t]{l}$C$-geodesic\end{tabular}}}}%
    \put(0,0){\includegraphics[width=\unitlength,page=5]{gamma.pdf}}%
    \put(0.35947766,0.37944791){\color[rgb]{0,0.33333333,0.85490196}\makebox(0,0)[lt]{\lineheight{1.25}\smash{\begin{tabular}[t]{l}$\Gamma$\end{tabular}}}}%
    \put(0.38803217,0.37944791){\color[rgb]{0,0,0}\makebox(0,0)[lt]{\lineheight{1.25}\smash{\begin{tabular}[t]{l}$\subset$\end{tabular}}}}%
    \put(0.41808157,0.37944791){\color[rgb]{0.6,0.6,0.6}\makebox(0,0)[lt]{\lineheight{1.25}\smash{\begin{tabular}[t]{l}$\C$\end{tabular}}}}%
    \put(0.31561515,0.30409066){\color[rgb]{0.30196078,0.30196078,0.30196078}\makebox(0,0)[lt]{\lineheight{1.25}\smash{\begin{tabular}[t]{l}$\sB$\end{tabular}}}}%
    \put(0.31133806,0.2060291){\color[rgb]{0,0.66666667,0.53333333}\makebox(0,0)[lt]{\lineheight{1.25}\smash{\begin{tabular}[t]{l}$y_0$\end{tabular}}}}%
    \put(0.26825593,0.21537802){\color[rgb]{0.30196078,0.30196078,0.30196078}\makebox(0,0)[lt]{\lineheight{1.25}\smash{\begin{tabular}[t]{l}$0$\end{tabular}}}}%
    \put(0.60718099,0.41169907){\color[rgb]{0,0.66666667,0.53333333}\makebox(0,0)[lt]{\lineheight{1.25}\smash{\begin{tabular}[t]{l}$\ell_0$\end{tabular}}}}%
    \put(0.83869128,0.41169907){\color[rgb]{0,0.66666667,0.53333333}\makebox(0,0)[lt]{\lineheight{1.25}\smash{\begin{tabular}[t]{l}$\ell_1$\end{tabular}}}}%
    \put(0.60718099,0.18018865){\color[rgb]{0,0.66666667,0.53333333}\makebox(0,0)[lt]{\lineheight{1.25}\smash{\begin{tabular}[t]{l}$\ell_2$\end{tabular}}}}%
    \put(0.83869128,0.18018865){\color[rgb]{0,0.66666667,0.53333333}\makebox(0,0)[lt]{\lineheight{1.25}\smash{\begin{tabular}[t]{l}$\ell_\delta$\end{tabular}}}}%
  \end{picture}%
\endgroup%

%% file: ordering.pdf_tex
\begingroup%
  \makeatletter%
  \providecommand\color[2][]{%
    \errmessage{(Inkscape) Color is used for the text in Inkscape, but the package 'color.sty' is not loaded}%
    \renewcommand\color[2][]{}%
  }%
  \providecommand\transparent[1]{%
    \errmessage{(Inkscape) Transparency is used (non-zero) for the text in Inkscape, but the package 'transparent.sty' is not loaded}%
    \renewcommand\transparent[1]{}%
  }%
  \providecommand\rotatebox[2]{#2}%
  \newcommand*\fsize{\dimexpr\f@size pt\relax}%
  \newcommand*\lineheight[1]{\fontsize{\fsize}{#1\fsize}\selectfont}%
  \ifx\svgwidth\undefined%
    \setlength{\unitlength}{453.54330709bp}%
    \ifx\svgscale\undefined%
      \relax%
    \else%
      \setlength{\unitlength}{\unitlength * \real{\svgscale}}%
    \fi%
  \else%
    \setlength{\unitlength}{\svgwidth}%
  \fi%
  \global\let\svgwidth\undefined%
  \global\let\svgscale\undefined%
  \makeatother%
  \begin{picture}(1,0.39375)%
    \lineheight{1}%
    \setlength\tabcolsep{0pt}%
    \put(0,0){\includegraphics[width=\unitlength,page=1]{ordering.pdf}}%
    \put(-1.46810316,0.2787354){\color[rgb]{0,0,0}\makebox(0,0)[lt]{\lineheight{1.25}\smash{\begin{tabular}[t]{l}$\cdots$\end{tabular}}}}%
    \put(-1.13737396,0.27873539){\color[rgb]{0,0,0}\makebox(0,0)[lt]{\lineheight{1.25}\smash{\begin{tabular}[t]{l}$\cdots$\end{tabular}}}}%
    \put(-0.79010816,0.27873539){\color[rgb]{0,0,0}\makebox(0,0)[lt]{\lineheight{1.25}\smash{\begin{tabular}[t]{l}$\cdots$\end{tabular}}}}%
    \put(-1.50255879,0.34429671){\color[rgb]{0,0.33333333,0.85490196}\makebox(0,0)[lt]{\lineheight{1.25}\smash{\begin{tabular}[t]{l}$b_1$\end{tabular}}}}%
    \put(-1.17182993,0.34429671){\color[rgb]{0,0.33333333,0.85490196}\makebox(0,0)[lt]{\lineheight{1.25}\smash{\begin{tabular}[t]{l}$b_N$\end{tabular}}}}%
    \put(-0.84110107,0.34429671){\color[rgb]{0,0.33333333,0.85490196}\makebox(0,0)[lt]{\lineheight{1.25}\smash{\begin{tabular}[t]{l}$\tau$\end{tabular}}}}%
    \put(0,0){\includegraphics[width=\unitlength,page=2]{ordering.pdf}}%
    \put(1.91290044,0.29453128){\color[rgb]{0.30196078,0.30196078,0.30196078}\makebox(0,0)[lt]{\lineheight{1.25}\smash{\begin{tabular}[t]{l}$(1,1)$\end{tabular}}}}%
    \put(2.00475305,0.36067707){\color[rgb]{0.6,0.6,0.6}\makebox(0,0)[lt]{\lineheight{1.25}\smash{\begin{tabular}[t]{l}$S^1\times S^1$\end{tabular}}}}%
    \put(1.75300008,0.34414063){\color[rgb]{0,0.33333333,0.83137255}\makebox(0,0)[lt]{\lineheight{1.25}\smash{\begin{tabular}[t]{l}$\Arg_\ell$\end{tabular}}}}%
    \put(0,0){\includegraphics[width=\unitlength,page=3]{ordering.pdf}}%
    \put(1.6240312,0.33091145){\color[rgb]{0.30196078,0.30196078,0.30196078}\makebox(0,0)[lt]{\lineheight{1.25}\smash{\begin{tabular}[t]{l}$\Z^2$\end{tabular}}}}%
    \put(1.54982417,0.33091146){\color[rgb]{0,0.33333333,0.85490196}\makebox(0,0)[lt]{\lineheight{1.25}\smash{\begin{tabular}[t]{l}$A$\end{tabular}}}}%
    \put(1.58779779,0.33091145){\color[rgb]{0,0,0}\makebox(0,0)[lt]{\lineheight{1.25}\smash{\begin{tabular}[t]{l}$\subset$\end{tabular}}}}%
    \put(0,0){\includegraphics[width=\unitlength,page=4]{ordering.pdf}}%
    \put(0.00094164,0.30445313){\color[rgb]{0,0,0}\makebox(0,0)[lt]{\lineheight{1.25}\smash{\begin{tabular}[t]{l}$\nu=\arg(y_0)$\end{tabular}}}}%
    \put(0.14469842,0.30279949){\color[rgb]{0,0.66666667,0.53333333}\makebox(0,0)[lt]{\lineheight{1.25}\smash{\begin{tabular}[t]{l}$1$\end{tabular}}}}%
    \put(0.20422964,0.30279949){\color[rgb]{0,0.66666667,0.53333333}\makebox(0,0)[lt]{\lineheight{1.25}\smash{\begin{tabular}[t]{l}$2$\end{tabular}}}}%
    \put(0.27037549,0.30279949){\color[rgb]{0,0.66666667,0.53333333}\makebox(0,0)[lt]{\lineheight{1.25}\smash{\begin{tabular}[t]{l}$3$\end{tabular}}}}%
    \put(0.33982864,0.30279949){\color[rgb]{0,0.66666667,0.53333333}\makebox(0,0)[lt]{\lineheight{1.25}\smash{\begin{tabular}[t]{l}$4$\end{tabular}}}}%
    \put(0.40597449,0.30279949){\color[rgb]{0,0.66666667,0.53333333}\makebox(0,0)[lt]{\lineheight{1.25}\smash{\begin{tabular}[t]{l}$5$\end{tabular}}}}%
    \put(0,0){\includegraphics[width=\unitlength,page=5]{ordering.pdf}}%
    \put(0.69370878,0.23665366){\color[rgb]{0.90196078,0.55686275,1}\makebox(0,0)[lt]{\lineheight{1.25}\smash{\begin{tabular}[t]{l}$3$\end{tabular}}}}%
    \put(0.76316185,0.23665366){\color[rgb]{0.9254902,0.75686275,0.53333333}\makebox(0,0)[lt]{\lineheight{1.25}\smash{\begin{tabular}[t]{l}$4$\end{tabular}}}}%
    \put(0.82930766,0.23665366){\color[rgb]{0,0.66666667,0.53333333}\makebox(0,0)[lt]{\lineheight{1.25}\smash{\begin{tabular}[t]{l}$5$\end{tabular}}}}%
    \put(0,0){\includegraphics[width=\unitlength,page=6]{ordering.pdf}}%
    \put(0.56803164,0.23665366){\color[rgb]{0,0.66666667,0.53333333}\makebox(0,0)[lt]{\lineheight{1.25}\smash{\begin{tabular}[t]{l}$1$\end{tabular}}}}%
    \put(0.62756293,0.23665366){\color[rgb]{0.90196078,0.55686275,1}\makebox(0,0)[lt]{\lineheight{1.25}\smash{\begin{tabular}[t]{l}$2$\end{tabular}}}}%
    \put(0,0){\includegraphics[width=\unitlength,page=7]{ordering.pdf}}%
    \put(0.68957038,0.01506512){\color[rgb]{0.4,0.4,0.4}\makebox(0,0)[lt]{\lineheight{1.25}\smash{\begin{tabular}[t]{l}$P_1$\end{tabular}}}}%
  \end{picture}%
\endgroup%

%% file: travel.pdf_tex
\begingroup%
  \makeatletter%
  \providecommand\color[2][]{%
    \errmessage{(Inkscape) Color is used for the text in Inkscape, but the package 'color.sty' is not loaded}%
    \renewcommand\color[2][]{}%
  }%
  \providecommand\transparent[1]{%
    \errmessage{(Inkscape) Transparency is used (non-zero) for the text in Inkscape, but the package 'transparent.sty' is not loaded}%
    \renewcommand\transparent[1]{}%
  }%
  \providecommand\rotatebox[2]{#2}%
  \newcommand*\fsize{\dimexpr\f@size pt\relax}%
  \newcommand*\lineheight[1]{\fontsize{\fsize}{#1\fsize}\selectfont}%
  \ifx\svgwidth\undefined%
    \setlength{\unitlength}{453.54330709bp}%
    \ifx\svgscale\undefined%
      \relax%
    \else%
      \setlength{\unitlength}{\unitlength * \real{\svgscale}}%
    \fi%
  \else%
    \setlength{\unitlength}{\svgwidth}%
  \fi%
  \global\let\svgwidth\undefined%
  \global\let\svgscale\undefined%
  \makeatother%
  \begin{picture}(1,0.375)%
    \lineheight{1}%
    \setlength\tabcolsep{0pt}%
    \put(0,0){\includegraphics[width=\unitlength,page=1]{travel.pdf}}%
    \put(-1.46810316,0.2599854){\color[rgb]{0,0,0}\makebox(0,0)[lt]{\lineheight{1.25}\smash{\begin{tabular}[t]{l}$\cdots$\end{tabular}}}}%
    \put(-1.13737396,0.25998539){\color[rgb]{0,0,0}\makebox(0,0)[lt]{\lineheight{1.25}\smash{\begin{tabular}[t]{l}$\cdots$\end{tabular}}}}%
    \put(-0.79010816,0.25998539){\color[rgb]{0,0,0}\makebox(0,0)[lt]{\lineheight{1.25}\smash{\begin{tabular}[t]{l}$\cdots$\end{tabular}}}}%
    \put(-1.50255879,0.32554671){\color[rgb]{0,0.33333333,0.85490196}\makebox(0,0)[lt]{\lineheight{1.25}\smash{\begin{tabular}[t]{l}$b_1$\end{tabular}}}}%
    \put(-1.17182993,0.32554671){\color[rgb]{0,0.33333333,0.85490196}\makebox(0,0)[lt]{\lineheight{1.25}\smash{\begin{tabular}[t]{l}$b_N$\end{tabular}}}}%
    \put(-0.84110107,0.32554671){\color[rgb]{0,0.33333333,0.85490196}\makebox(0,0)[lt]{\lineheight{1.25}\smash{\begin{tabular}[t]{l}$\tau$\end{tabular}}}}%
    \put(0,0){\includegraphics[width=\unitlength,page=2]{travel.pdf}}%
    \put(1.91290044,0.27578128){\color[rgb]{0.30196078,0.30196078,0.30196078}\makebox(0,0)[lt]{\lineheight{1.25}\smash{\begin{tabular}[t]{l}$(1,1)$\end{tabular}}}}%
    \put(2.00475305,0.34192707){\color[rgb]{0.6,0.6,0.6}\makebox(0,0)[lt]{\lineheight{1.25}\smash{\begin{tabular}[t]{l}$S^1\times S^1$\end{tabular}}}}%
    \put(1.75300008,0.32539063){\color[rgb]{0,0.33333333,0.83137255}\makebox(0,0)[lt]{\lineheight{1.25}\smash{\begin{tabular}[t]{l}$\Arg_\ell$\end{tabular}}}}%
    \put(0,0){\includegraphics[width=\unitlength,page=3]{travel.pdf}}%
    \put(1.6240312,0.31216145){\color[rgb]{0.30196078,0.30196078,0.30196078}\makebox(0,0)[lt]{\lineheight{1.25}\smash{\begin{tabular}[t]{l}$\Z^2$\end{tabular}}}}%
    \put(1.54982417,0.31216146){\color[rgb]{0,0.33333333,0.85490196}\makebox(0,0)[lt]{\lineheight{1.25}\smash{\begin{tabular}[t]{l}$A$\end{tabular}}}}%
    \put(1.58779779,0.31216145){\color[rgb]{0,0,0}\makebox(0,0)[lt]{\lineheight{1.25}\smash{\begin{tabular}[t]{l}$\subset$\end{tabular}}}}%
    \put(0,0){\includegraphics[width=\unitlength,page=4]{travel.pdf}}%
    \put(3.92454025,0.21790366){\color[rgb]{0,0.66666667,0.53333333}\makebox(0,0)[lt]{\lineheight{1.25}\smash{\begin{tabular}[t]{l}$1$\end{tabular}}}}%
    \put(3.98407138,0.21790366){\color[rgb]{0.90196078,0.55686275,1}\makebox(0,0)[lt]{\lineheight{1.25}\smash{\begin{tabular}[t]{l}$2$\end{tabular}}}}%
    \put(4.05021719,0.21790366){\color[rgb]{0.90196078,0.55686275,1}\makebox(0,0)[lt]{\lineheight{1.25}\smash{\begin{tabular}[t]{l}$3$\end{tabular}}}}%
    \put(4.11967034,0.21790366){\color[rgb]{0.9254902,0.75686275,0.53333333}\makebox(0,0)[lt]{\lineheight{1.25}\smash{\begin{tabular}[t]{l}$4$\end{tabular}}}}%
    \put(4.18581615,0.21790366){\color[rgb]{0,0.66666667,0.53333333}\makebox(0,0)[lt]{\lineheight{1.25}\smash{\begin{tabular}[t]{l}$5$\end{tabular}}}}%
    \put(0,0){\includegraphics[width=\unitlength,page=5]{travel.pdf}}%
    \put(1.59117252,0.56000527){\color[rgb]{0,0.66666667,0.53333333}\makebox(0,0)[lt]{\lineheight{1.25}\smash{\begin{tabular}[t]{l}$1$\end{tabular}}}}%
    \put(1.65070366,0.56000527){\color[rgb]{0,0.66666667,0.53333333}\makebox(0,0)[lt]{\lineheight{1.25}\smash{\begin{tabular}[t]{l}$2$\end{tabular}}}}%
    \put(1.71684947,0.56000527){\color[rgb]{0,0.66666667,0.53333333}\makebox(0,0)[lt]{\lineheight{1.25}\smash{\begin{tabular}[t]{l}$3$\end{tabular}}}}%
    \put(1.78630261,0.56000527){\color[rgb]{0,0.66666667,0.53333333}\makebox(0,0)[lt]{\lineheight{1.25}\smash{\begin{tabular}[t]{l}$4$\end{tabular}}}}%
    \put(1.85244843,0.56000527){\color[rgb]{0,0.66666667,0.53333333}\makebox(0,0)[lt]{\lineheight{1.25}\smash{\begin{tabular}[t]{l}$5$\end{tabular}}}}%
    \put(0,0){\includegraphics[width=\unitlength,page=6]{travel.pdf}}%
    \put(0.00080977,0.14359513){\color[rgb]{0,0,0}\makebox(0,0)[lt]{\lineheight{1.25}\smash{\begin{tabular}[t]{l}$\nu=\nu_1$\end{tabular}}}}%
    \put(0.00080977,0.04768378){\color[rgb]{0,0,0}\makebox(0,0)[lt]{\lineheight{1.25}\smash{\begin{tabular}[t]{l}$\nu=\nu_2$\end{tabular}}}}%
    \put(0.00080977,0.2891161){\color[rgb]{0,0,0}\makebox(0,0)[lt]{\lineheight{1.25}\smash{\begin{tabular}[t]{l}$\nu=\arg(y_0)$\end{tabular}}}}%
  \end{picture}%
\endgroup%

%% file: isomonodromy.pdf_tex
\begingroup%
  \makeatletter%
  \providecommand\color[2][]{%
    \errmessage{(Inkscape) Color is used for the text in Inkscape, but the package 'color.sty' is not loaded}%
    \renewcommand\color[2][]{}%
  }%
  \providecommand\transparent[1]{%
    \errmessage{(Inkscape) Transparency is used (non-zero) for the text in Inkscape, but the package 'transparent.sty' is not loaded}%
    \renewcommand\transparent[1]{}%
  }%
  \providecommand\rotatebox[2]{#2}%
  \newcommand*\fsize{\dimexpr\f@size pt\relax}%
  \newcommand*\lineheight[1]{\fontsize{\fsize}{#1\fsize}\selectfont}%
  \ifx\svgwidth\undefined%
    \setlength{\unitlength}{453.54330709bp}%
    \ifx\svgscale\undefined%
      \relax%
    \else%
      \setlength{\unitlength}{\unitlength * \real{\svgscale}}%
    \fi%
  \else%
    \setlength{\unitlength}{\svgwidth}%
  \fi%
  \global\let\svgwidth\undefined%
  \global\let\svgscale\undefined%
  \makeatother%
  \begin{picture}(1,0.3125)%
    \lineheight{1}%
    \setlength\tabcolsep{0pt}%
    \put(0,0){\includegraphics[width=\unitlength,page=1]{isomonodromy.pdf}}%
    \put(-1.46810316,0.1974854){\color[rgb]{0,0,0}\makebox(0,0)[lt]{\lineheight{1.25}\smash{\begin{tabular}[t]{l}$\cdots$\end{tabular}}}}%
    \put(-1.13737396,0.19748539){\color[rgb]{0,0,0}\makebox(0,0)[lt]{\lineheight{1.25}\smash{\begin{tabular}[t]{l}$\cdots$\end{tabular}}}}%
    \put(-0.79010816,0.19748539){\color[rgb]{0,0,0}\makebox(0,0)[lt]{\lineheight{1.25}\smash{\begin{tabular}[t]{l}$\cdots$\end{tabular}}}}%
    \put(-1.50255879,0.26304671){\color[rgb]{0,0.33333333,0.85490196}\makebox(0,0)[lt]{\lineheight{1.25}\smash{\begin{tabular}[t]{l}$b_1$\end{tabular}}}}%
    \put(-1.17182993,0.26304671){\color[rgb]{0,0.33333333,0.85490196}\makebox(0,0)[lt]{\lineheight{1.25}\smash{\begin{tabular}[t]{l}$b_N$\end{tabular}}}}%
    \put(-0.84110107,0.26304671){\color[rgb]{0,0.33333333,0.85490196}\makebox(0,0)[lt]{\lineheight{1.25}\smash{\begin{tabular}[t]{l}$\tau$\end{tabular}}}}%
    \put(0,0){\includegraphics[width=\unitlength,page=2]{isomonodromy.pdf}}%
    \put(1.91290044,0.21328128){\color[rgb]{0.30196078,0.30196078,0.30196078}\makebox(0,0)[lt]{\lineheight{1.25}\smash{\begin{tabular}[t]{l}$(1,1)$\end{tabular}}}}%
    \put(2.00475305,0.27942707){\color[rgb]{0.6,0.6,0.6}\makebox(0,0)[lt]{\lineheight{1.25}\smash{\begin{tabular}[t]{l}$S^1\times S^1$\end{tabular}}}}%
    \put(1.75300008,0.26289063){\color[rgb]{0,0.33333333,0.83137255}\makebox(0,0)[lt]{\lineheight{1.25}\smash{\begin{tabular}[t]{l}$\Arg_\ell$\end{tabular}}}}%
    \put(0,0){\includegraphics[width=\unitlength,page=3]{isomonodromy.pdf}}%
    \put(0.15189565,0.2364323){\color[rgb]{0.30196078,0.30196078,0.30196078}\makebox(0,0)[lt]{\lineheight{1.25}\smash{\begin{tabular}[t]{l}$\Z^2$\end{tabular}}}}%
    \put(0.07768923,0.23643231){\color[rgb]{0,0.33333333,0.85490196}\makebox(0,0)[lt]{\lineheight{1.25}\smash{\begin{tabular}[t]{l}$A$\end{tabular}}}}%
    \put(3.58775834,0.31911458){\color[rgb]{0,0.33333333,0.85490196}\makebox(0,0)[lt]{\lineheight{1.25}\smash{\begin{tabular}[t]{l}$\Arg_\vp$\end{tabular}}}}%
    \put(3.68777724,0.31911458){\color[rgb]{0.6,0.6,0.6}\makebox(0,0)[lt]{\lineheight{1.25}\smash{\begin{tabular}[t]{l}$S^1\times S^1$\end{tabular}}}}%
    \put(0.11566274,0.23643227){\color[rgb]{0,0,0}\makebox(0,0)[lt]{\lineheight{1.25}\smash{\begin{tabular}[t]{l}$\subset$\end{tabular}}}}%
    \put(3.65446604,0.31911458){\color[rgb]{0,0,0}\makebox(0,0)[lt]{\lineheight{1.25}\smash{\begin{tabular}[t]{l}$\subset$\end{tabular}}}}%
    \put(0,0){\includegraphics[width=\unitlength,page=4]{isomonodromy.pdf}}%
    \put(3.99316635,0.24635417){\color[rgb]{0,0.33333333,0.85490196}\makebox(0,0)[lt]{\lineheight{1.25}\smash{\begin{tabular}[t]{l}$A$-geodesic\end{tabular}}}}%
    \put(3.99316635,0.13059899){\color[rgb]{0,0.33333333,0.85490196}\makebox(0,0)[lt]{\lineheight{1.25}\smash{\begin{tabular}[t]{l}$B$-geodesic\end{tabular}}}}%
    \put(3.99316635,-0.0016927){\color[rgb]{0,0.33333333,0.85490196}\makebox(0,0)[lt]{\lineheight{1.25}\smash{\begin{tabular}[t]{l}$C$-geodesic\end{tabular}}}}%
    \put(0,0){\includegraphics[width=\unitlength,page=5]{isomonodromy.pdf}}%
    \put(1.93870941,1.09780726){\color[rgb]{0,0.33333333,0.85490196}\makebox(0,0)[lt]{\lineheight{1.25}\smash{\begin{tabular}[t]{l}$\Gamma$\end{tabular}}}}%
    \put(1.96726398,1.09780729){\color[rgb]{0,0,0}\makebox(0,0)[lt]{\lineheight{1.25}\smash{\begin{tabular}[t]{l}$\subset$\end{tabular}}}}%
    \put(1.9973133,1.09780728){\color[rgb]{0.6,0.6,0.6}\makebox(0,0)[lt]{\lineheight{1.25}\smash{\begin{tabular}[t]{l}$\C$\end{tabular}}}}%
    \put(1.89484697,1.02245003){\color[rgb]{0.30196078,0.30196078,0.30196078}\makebox(0,0)[lt]{\lineheight{1.25}\smash{\begin{tabular}[t]{l}$\sB$\end{tabular}}}}%
    \put(1.89056987,0.92438844){\color[rgb]{0,0.66666667,0.53333333}\makebox(0,0)[lt]{\lineheight{1.25}\smash{\begin{tabular}[t]{l}$t_0$\end{tabular}}}}%
    \put(1.84748778,0.93373736){\color[rgb]{0.30196078,0.30196078,0.30196078}\makebox(0,0)[lt]{\lineheight{1.25}\smash{\begin{tabular}[t]{l}$0$\end{tabular}}}}%
    \put(2.18641285,1.13005848){\color[rgb]{0,0.66666667,0.53333333}\makebox(0,0)[lt]{\lineheight{1.25}\smash{\begin{tabular}[t]{l}$\ell_0$\end{tabular}}}}%
    \put(2.41792298,1.13005848){\color[rgb]{0,0.66666667,0.53333333}\makebox(0,0)[lt]{\lineheight{1.25}\smash{\begin{tabular}[t]{l}$\ell_1$\end{tabular}}}}%
    \put(2.18641285,0.89854802){\color[rgb]{0,0.66666667,0.53333333}\makebox(0,0)[lt]{\lineheight{1.25}\smash{\begin{tabular}[t]{l}$\ell_2$\end{tabular}}}}%
    \put(2.41792298,0.89854802){\color[rgb]{0,0.66666667,0.53333333}\makebox(0,0)[lt]{\lineheight{1.25}\smash{\begin{tabular}[t]{l}$\ell_\delta$\end{tabular}}}}%
    \put(0,0){\includegraphics[width=\unitlength,page=6]{isomonodromy.pdf}}%
    \put(0.37679143,0.28273438){\color[rgb]{0.6,0.6,0.6}\makebox(0,0)[lt]{\lineheight{1.25}\smash{\begin{tabular}[t]{l}$\R^2$\end{tabular}}}}%
    \put(0.30258505,0.28273439){\color[rgb]{0.30196078,0.30196078,0.30196078}\makebox(0,0)[lt]{\lineheight{1.25}\smash{\begin{tabular}[t]{l}$C$\end{tabular}}}}%
    \put(0.34055857,0.28273437){\color[rgb]{0,0,0}\makebox(0,0)[lt]{\lineheight{1.25}\smash{\begin{tabular}[t]{l}$\subset$\end{tabular}}}}%
    \put(0,0){\includegraphics[width=\unitlength,page=7]{isomonodromy.pdf}}%
    \put(0.59594868,0.28421327){\color[rgb]{0,0.33333333,0.85490196}\makebox(0,0)[lt]{\lineheight{1.25}\smash{\begin{tabular}[t]{l}$\Gamma$\end{tabular}}}}%
    \put(0.62450448,0.28421409){\color[rgb]{0,0,0}\makebox(0,0)[lt]{\lineheight{1.25}\smash{\begin{tabular}[t]{l}$\subset$\end{tabular}}}}%
    \put(0.65455163,0.28421347){\color[rgb]{0.6,0.6,0.6}\makebox(0,0)[lt]{\lineheight{1.25}\smash{\begin{tabular}[t]{l}$\C$\end{tabular}}}}%
    \put(0.66629257,0.18198015){\color[rgb]{0.30196078,0.30196078,0.30196078}\makebox(0,0)[lt]{\lineheight{1.25}\smash{\begin{tabular}[t]{l}$\sB$\end{tabular}}}}%
    \put(0,0){\includegraphics[width=\unitlength,page=8]{isomonodromy.pdf}}%
    \put(-0.6419211,0.09584717){\color[rgb]{0.37254902,0.55294118,0.82745098}\makebox(0,0)[lt]{\lineheight{1.25}\smash{\begin{tabular}[t]{l}$\conv(\wt A)$\end{tabular}}}}%
    \put(0,0){\includegraphics[width=\unitlength,page=9]{isomonodromy.pdf}}%
    \put(0.08832507,0.01930354){\color[rgb]{0,0,0}\makebox(0,0)[lt]{\lineheight{1.25}\smash{\begin{tabular}[t]{l}$(a)$\end{tabular}}}}%
    \put(0.33495576,0.01690028){\color[rgb]{0,0,0}\makebox(0,0)[lt]{\lineheight{1.25}\smash{\begin{tabular}[t]{l}$(b)$\end{tabular}}}}%
    \put(0.60171642,0.01930354){\color[rgb]{0,0,0}\makebox(0,0)[lt]{\lineheight{1.25}\smash{\begin{tabular}[t]{l}$(c)$\end{tabular}}}}%
    \put(0.85705605,0.01930354){\color[rgb]{0,0,0}\makebox(0,0)[lt]{\lineheight{1.25}\smash{\begin{tabular}[t]{l}$(d)$\end{tabular}}}}%
  \end{picture}%
\endgroup%